\documentclass{amsart}
\usepackage{amsmath}
\usepackage{amsthm}
\usepackage{amsbsy}
\usepackage{amsopn}
\usepackage{amsmath,amscd}
\usepackage{amssymb}
\usepackage{amsfonts}
\usepackage{multirow}
\usepackage[official]{eurosym}
\usepackage{comment}
\usepackage{pdfpages}

\newcommand{\ki}[2]{{\color{green!80!black} #1
}{\color{blue!90!black}#2 
}}

\usepackage{tikz}
\usetikzlibrary{arrows.meta}
\usetikzlibrary{patterns}
\usetikzlibrary{intersections}
\usetikzlibrary{matrix}
\usepackage{pgfplots}
\pgfplotsset{compat=1.14}
\tikzset{
  partial ellipse/.style args={#1:#2:#3}{
    insert path={+ (#1:#3) arc (#1:#2:#3)}
  }
}
\tikzset{
  partial ellipsecake/.style args={#1:#2:#3:#4}{
    insert path={+ (#1:#3) arc (#1:#2:#3 and #4) -- (0,0)  -- (#3,0)}
  }
}
\makeatletter
\tikzset{
  use path for main/.code={%
    \tikz@addmode{%
      \expandafter\pgfsyssoftpath@setcurrentpath\csname tikz@intersect@path@name@#1\endcsname
    }%
  },
  use path for actions/.code={%
    \expandafter\def\expandafter\tikz@preactions\expandafter{\tikz@preactions\expandafter\let\expandafter\tikz@actions@path\csname tikz@intersect@path@name@#1\endcsname}%
  },
  use path/.style={%
    use path for main=#1,
    use path for actions=#1,
  }
}
\makeatother
\usepackage{tikz-cd}

\newcommand{\ve}{\varepsilon}

\newcommand{\mat}[1]{\ensuremath{
\left[\begin{matrix}#1
\end{matrix}\right]
}}

\newcommand\cC{\mathcal{C}}

\newcommand\cF{\mathcal{F}}

\newcommand\cH{\mathcal{H}}

\newcommand\cM{\mathcal{M}}

\newcommand\cQ{\mathcal{Q}}

\newcommand\cT{\mathcal{T}}

\def\bC{\mathbf{C}}
\def\bD{\mathbf{D}}

\def\bN{\mathbf{N}}

\def\bQ{\mathbf{Q}}
\def\bR{\mathbf{R}}

\def\bZ{\mathbf{Z}}

\usepackage{graphicx}

\newcommand{\xxrarrow}[1]{\ensuremath{
\mathop{\longrightarrow}\limits^{#1}}}

\usetikzlibrary{arrows,decorations.pathmorphing,backgrounds,fit,positioning,shapes.symbols,chains,calc}
\usepackage[all]{xy}
\usepackage[hmargin=3cm,vmargin=3cm]{geometry}
\usepackage{setspace,kantlipsum}

\usepackage{booktabs}

\newtheorem{theorem}{Theorem}[section]
\newtheorem{lemma}[theorem]{Lemma}
\newtheorem{conjecture}[theorem]{Conjecture}
\newtheorem{proposition}[theorem]{Proposition}
\newtheorem{corollary}[theorem]{Corollary}

\theoremstyle{definition}
\newtheorem{definition}[theorem]{Definition}
\theoremstyle{remark}

\newtheorem{problem}[theorem]{Problem}

\newtheorem{example}[theorem]{Example}
\newtheorem{inductive lemma}[theorem]{Inductive Lemma}
\newtheorem{remark}[theorem]{Remark}

\newtheorem{warning}[theorem]{Warning}
\newtheorem{question}[theorem]{Question}

\title{Legendrian and Lagrangian higher torsion}

\author{Daniel  \'Alvarez-Gavela}

\author{Kiyoshi Igusa}

\author{Michael G Sullivan}

\begin{document}
\maketitle

\begin{abstract}
Let $M$ be a closed manifold. We introduce a family of Legendrian isotopy invariants for Legendrians in $J^1M$, which we collectively call Legendrian higher torsion. Given a choice of a class $\cF$ of fibre bundles over $M$, equipped with suitable unitary local systems, the Legendrian higher torsion of a Legendrian $\Lambda \subset J^1M$ is the subset of $H^*(M;\bR)$ consisting of higher Reidemeister torsion cohomology classes of fibre bundles $W$ over $M$ in the class $\cF$ such that $\Lambda$ admits a generating function on a stabilization of $W$.  For the class of tube bundles in the sense of Waldhausen we call the invariant tube torsion. We show that the tube torsion of a nearby Lagrangian $L \subset T^*M$ is well-defined when the stable Gauss map $L \to U/O$ is trivial (for example when $L$ is a nearby Lagrangian homotopy sphere) and it consists of a union of cosets of a normalized version of the Pontryagin character. We also identify a distinguished coset, invariant under Hamiltonian isotopy of $L$, which we call nearby Lagrangian torsion. We do not know whether nearby Lagrangians must have trivial tube torsion, as would follow from the nearby Lagrangian conjecture. 
However, we show that there exist Legendrians $\Lambda \subset J^1M$ with nontrivial tube torsion whose projection $\Lambda \to M$ is homotopic to a diffeomorphism.

 \end{abstract}
 
 \onehalfspacing
 \tableofcontents
 
 \section{Introduction}\label{sec: intro}
 \subsection{Main results}\label{sec: main results}
 
 Fix a closed, connected, smooth manifold $M$. We denote by $J^1M$ its 1-jet space, where we recall $J^1M = T^*M \times \bR$ with coordinates $(q,p,z)\in J^1M$ corresponding to $q \in M$, $p \in T^*_qM$ and $z \in \bR$, with the contact form $dz-pdq$ endowing $J^1M$ with its standard contact structure. In this article we introduce a family of new K-theoretic Legendrian isotopy invariants for Legendrian submanifolds $\Lambda \subset J^1M$, which we call {\em Legendrian higher torsion}. These invariants are defined in terms of generating functions for $\Lambda$ and have several incarnations depending on the class of fiber bundles one chooses to consider as the possible domains of generating functions for $\Lambda$. 
 
 In general, one has notion of a {\em torsion pair}, which consists of a pair $(W,\rho)$ where $F \to W \to M$ is a fiber bundle and $\rho : \pi_1 W \to U(1)$ is a rank 1 unitary local system such that $H_*(F;\bC^\rho)$ is a unipotent $\pi_1M$-module. If $F$ is closed no more data is needed, but if $F$ has boundary or is non-compact one should also fix the behavior near the boundary or at infinity for the class of functions $f:W \to \bR$ we allow ourselves to consider as possible generating functions. For a torsion pair $(W,\rho)$ the theory of higher torsion characteristic classes due to the second author and J. Klein produces cohomology classes $\tau_k(W,\rho) \in H^{2k}(M;\bR)$ which are invariants of $(W,\rho)$ up to fibered diffeomorphism of fiber bundles (compatible with the fixed behavior at infinity) and up to the notion of stable equivalence which replaces $W$ with $W \times \bR^4$ and $f:W \to \bR$ with $f^{\text{s}}(w,x_1,x_2,y_1,y_2)=f(w)+x_1^2+x_2^2-y_1^2-y_2^2$ \cite{I02,I05,IK93a,IK93b,K89}. 
 
 When $H_*(F,\bC^\rho)=0$, for $k=0$ the class $\tau_0(W,\rho) \in H^0(M;\bR) \simeq \bR$ recovers the usual Reidemeister torsion of the fiber $F$ with respect to $\rho|_{\pi_1 F}$ and for $k>0$ the class $\tau_k(W,\rho)$ may be thought of as a parametrized generalization of this Reidemeister torsion. 
 
Equipped with the theory of higher torsion classes of fiber bundles, we define the $k$-th {\em Legendrian higher torsion} of a Legendrian $\Lambda \subset J^1M$ in a fixed collection $\cF$ of torsion pairs to be the subset $\tau_k(\Lambda,\cF) \subset H^{2k}(W,\rho)$  consisting of all degree $2k$ higher torsion characteristic classes $\tau_k(W,\rho) \in H^{2k}(M;\bR)$ of torsion pairs $(W,\rho)$ in the class $\cF$ such that $\Lambda$ admits a generating function on a stabilization of the fiber bundle $W \to M$. We denote $\tau(\Lambda,\cF) = \coprod_k  \tau_k(\Lambda,\cF)  \subset H^*(M;\bR)$ the {\em total Legendrian higher torsion} of $\Lambda$ in the fixed collection $\cF$. 

The starting point of the theory of Legendrian higher torsion is the following consequence of the homotopy lifting property for generating functions under Legendrian isotopies (up to stabilization). 
 
 \begin{theorem} The subset $\tau(\Lambda,\cF) \subset H^*(M;\bR)$ is a Legendrian isotopy invariant of $\Lambda$. 
 \end{theorem}
 
For example, in the article \cite{AI} the first and second authors implicitly gave a concrete application of these ideas in the special case in which $\cF$ consists of oriented circle bundles $S^1 \to W \to \Sigma$ over a fixed orientable surface $\Sigma$, equipped with rank 1 unitary local systems which send the oriented fiber to a fixed $n$-root of unity $\neq 1$ for $\pm n$ the Euler number of $W$. In that specific case it is computed in \cite{AI} that for Legendrians in a certain class (called {\em mesh Legendrians}), all of which in particular admit generating functions on circle bundles, the degree 2 Legendrian higher torsion $\tau_1(\Lambda,\cF)$ of a mesh Legendrian $\Lambda \subset J^1 \Sigma$ in $H^2(\Sigma;\bR) \simeq \bR$ records the Euler number (up to sign)  of the circle bundle that was used to generate it. It follows from this computation that each mesh Legendrian admits a generating function on a circle bundle of Euler number $\pm n$ for a unique positive integer $n$, {\em even after stabilization of the circle bundle}. In particular it follows that this number is an invariant of the mesh Legendrian.

 In fact, in \cite{AI} a variant of Legendrian higher torsion called {\em Legendrian Turaev torsion} was defined for a general class of Legendrians called {\em Euler Legendrians}. For mesh Legendrians the Legendrian Turaev torsion allowed one to also recover the sign of this Euler number. The Legendrian Turaev torsion was then used in \cite{AI} to distinguish pairs of Legendrians which were formally Legendrian isotopic but not Legendrian isotopic, and in fact such that the isomorphism class of the Legendrian contact homology dga and microlocal sheaf category was unable to distinguish them (at least with traditional coefficients).  
 
Although different flavors of the invariant may be useful for different purposes, our main focus in the present article is a specific version of Legendrian higher torsion which we call {\em tube  torsion}. This version is particularly relevant to the nearby Lagrangian conjecture, the study of which is the primary motivation for this work.  Tube torsion is Legendrian higher torsion for the class $\cF$ of tube bundles in the sense of Waldahusen \cite{W82}. Tube bundles have gained direct relevance to the nearby Lagrangian conjecture after the existence theorem for (twisted) generating functions on tube bundles for nearby Lagrangians proved by Abouzaid, Courte, Guillermou and Kragh in \cite{ACGK}. When working with tube bundles we will omit the class $\cF$ from the notation, and as will be explained, the data of the unitary local system may also be disregarded. Although the precise definition of the invariant will be given in Definition \ref{def: tube torsion} once the notions of tube bundles and generating functions of tube type are properly formulated, we give for now the following provisional definition. 

 \begin{definition}\label{def: higher whitehead}
The {\em tube torsion} of a Legendrian $\Lambda \subset J^1M$ is the collection $\tau(\Lambda) \subset H^*(M;\bR)$ of higher torsion classes $\tau_k(W) \in H^{2k}(M;\bR)$ of tube bundles $W \to M$ on which $\Lambda$ admits a generating function. 
 \end{definition}
 
 \begin{remark} Since the class of tubes is {\em stable}, i.e. closed under stabilization, we do not have to add {\em up to stabilization} in the above definition.  \end{remark}
 
The main properties of tube torsion are summarized in the following result. We will denote by $p:KO^0(M) \to H^*(M;\bR)$ a suitably normalized version of the Pontryagin character adapted to the definition of higher torsion, see Appendix \ref{sec: higher torsion of bundles}. Let $\zeta(s)=\sum_{n=1}^\infty n^{-s}$ be the standard zeta function. For $E \to M$ a real vector bundle we denote $p(E)=\sum_k p_k(E)$ with $p_k(E) \in H^{4k}(M;\bR)$ given by 
$$p_k(E) = \frac{1}{2} (-1)^k \zeta(2k+1) ch_{2k}(E \otimes \bC) .$$
 
   \begin{theorem}\label{thm: main 1}
The tube torsion of a Legendrian $\Lambda \subset J^1M$ satisfies the following:
 \begin{enumerate}
 \item\label{item: Leg-inv} $\tau(\Lambda)$  is invariant under Legendrian isotopy of $\Lambda$.
   \item\label{item: coset}  $\tau(\Lambda) $ is a union of cosets of $p$, so we may view $\tau(\Lambda)$ as a subset of $\text{coker}(p)$. 
 \item\label{item: zero-sec triv} for the zero section $M \subset J^1M $ we have $\tau(\Lambda) =\text{im}(p)$, the trivial coset.
 \end{enumerate}
 \end{theorem}
 
 \begin{remark}
For $W$ a tube bundle the lack of a unitary local system means we are drawing invariants from $K(\bR)$ rather than $K(\bC)$ and this forces the odd classes $\tau_{2k+1}(W)$ to be zero. So the only possible nontrivial classes are the even classes $\tau_{2k}(W) \in H^{4k}(M;\bR)$ and hence the tube torsion of a Legendrian $\Lambda \subset J^1M$ is a disjoint union $\tau(\Lambda) = \coprod_k \tau_{2k}(\Lambda)$ with $\tau_{2k}(\Lambda) \subset H^{4k}(M;\bR)$. 
 \end{remark}

Note that by definition, we have $\tau(\Lambda) =\varnothing$ if $\Lambda$ does not admit a generating function on any tube bundle. For example, since the fiberwise Morse theory of a tube bundle in the sense considered in the present article has Euler characteristic $\pm 1$  we will certainly have $\tau(\Lambda)=\varnothing$ if the projection $\Lambda \to M$ does not have degree $\pm 1$. 
When $\tau(\Lambda)=\varnothing$ we say that the tube torsion is {\em not defined}. 

So for which Legendrians  is tube torsion defined? A particularly elusive class of Legendrians is given by the Legendrian lifts of nearby Lagrangians. We recall that for $L$ and $M$ closed, connected, smooth $n$-dimensional manifolds, a nearby Lagrangian is an exact Lagrangian embedding $L \subset T^*M$. 
A nearby Lagrangian $L \subset T^*M$ admits a generating function on some tube bundle if and only if the stable Gauss map $L \to U/O$ is nullhomotopic \cite{ACGK}. Therefore the tube torsion of the Legendrian lift $\Lambda \subset J^1 M$ of any nearby Lagrangian $L \subset T^*M$ with stably trivial Gauss map is defined, and we denote it $\tau(L) \subset H^*(M;\bR)$. 
 
 Furthermore, in \cite{ACGK} a procedure for producing a generating function on a tube bundle for a nearby Lagrangian $L \subset T^*M$ is given once the input of a nullhomotopy of the stable Gauss map $L \to U/O$ is fixed, using the {\em doubling trick} which Guillermou introduced in \cite{Gui12}. We do not know whether any generating function on a tube bundle for a nearby Lagrangian is given by this procedure. We say that a generating function for $L$ which does arise by this procedure is {\em natural} and denote by $\nu(L)  \subset \tau(\Lambda) $ the subset consisting of those higher torsion classes in $\tau(\Lambda)$ of tube bundles $W \to M$ where the generating function for $L$ on $W$ is natural. We call $\nu(L)$  the {\em nearby Lagrangian torsion} of $L$. It satisfies the following properties:
 
 
  \begin{theorem}\label{thm: main 2} Let  $L \subset T^*M$ be a nearby Lagrangian with trivial stable Gauss map. The nearby Lagrangian torsion of $L$ satisfies the following:
  \begin{enumerate}
   \item $\nu(L)$ is invariant under Hamiltonian isotopies of $L$. 
  \item $\nu(L)$ is a single coset of $p$, so we get a well-defined element $\nu(L)  \in \text{coker}(p)$.
  \item for the zero section we have $\nu(M)=\tau(M)=\text{im}(p)$.
  \end{enumerate}
 \end{theorem}
 
 \begin{remark} Although the tube torsion $\tau(\Lambda) $ of the Legendrian lift $\Lambda \subset J^1M$ of a nearby Lagrangian $L \subset T^*M$ is invariant under Legendrian isotopies of $\Lambda \subset J^1M$, the nearby Lagrangian torsion $\nu(L)$ is not obviously invariant under Legendrian isotopies of $\Lambda \subset J^1M$. Concretely, if $L_1, L_2 \subset T^*M$ are two nearby Lagrangians with Legendrian isotopic Legendrian lifts, then we do not know whether it must be the case that $\nu(L_1) = \nu(L_2)$, even though Theorem \ref{thm: main 1} guarantees that $\tau(L_1)=\tau(L_2)$. 
 \end{remark}
 

 \begin{example}   In \cite{ACGK} it was also shown that the stable Gauss map $L \to U/O$ of a nearby Lagrangian $L \subset T^*M$ is always trivial on homotopy groups. Since $L$ and $M$ are homotopy equivalent, this implies that $L \to U/O$ is always nullhomotopic when $M$ has the homotopy type of a sphere. So for nearby Lagrangian homotopy spheres, generating functions on tube bundles always exist and therefore the tube torsion and nearby Lagrangian torsion of a nearby Lagrangian homotopy sphere are always defined. Both invariants are automatically trivial for degree reasons unless the dimension of $M$ is of the form $n=4k$. For concreteness, say $M=S^{4k}$ is the standard sphere (though the discussion is similar if $M$ is any exotic sphere of dimension divisible by 4). The only possible nontrivial higher torsion class of a tube bundle $W \to S^{4k}$ is $\tau_{2k}(W) \in H^{4k}(S^{4k};\bR) \simeq \bR$. Hence for a nearby Lagrangian homotopy sphere $\Sigma \subset T^*S^{4k}$ we have a distinguished element $$\nu(\Sigma) \in \text{coker}\left(p:KO^0(S^{4k}) \to H^{4k}(S^{4k};\bR)  \right).$$
Alternatively, we may view
  $$\nu(\Sigma) \in \text{coker}\left( \bZ \xrightarrow{\lambda_k} \bR \right) = \bR / \lambda_k \bZ$$ where $\lambda_k=p(E) \cdot [S^{4k}]  \in \bR$ for $E \to S^{4k}$ the generator of $KO^0(S^{4k}) \simeq \bZ$. So the nearby Lagrangian torsion of any nearby Lagrangian homotopy sphere $\Sigma \subset T^*S^{4k}$ is well-defined and it is just a real number modulo $\lambda_k \bZ$. 
   \end{example}

From  \cite{AD22} it is known that after a Hamiltonian isotopy we may assume that the front of any nearby Lagrangian homotopy sphere $\Sigma \subset T^*S^{4k}$ only has cusp type singularities. When the front of $\Sigma$ only has cusp type singularities, any generating function $f:W \to \bR$ for $\Sigma$ will have moderate singularities, i.e. Morse or Morse birth/death. See Figure \ref{Figure1905} below.

When a fibered function of tube type $f$  has moderate singularities, from the family of fiberwise Morse chain complexes $C_*(f_m)$, $m \in S^{4k}$, one may construct a homotopy class of map 
$$ M \to \text{Wh}^h(\bC,1) $$
and from this one may extract a cohomology class $\tau_{2k}(f) \in H^{2k}(M;\bR)$ as reviewed in Appendix \ref{sec: higher torsion of bundles}.

  \begin{theorem}\label{thm: compute} For a nearby Lagrangian homotopy sphere $\Sigma \subset T^*S^{4k}$, after a suitable Hamiltonian isotopy of $\Sigma$ if necessary to ensure the front only has cusp singularities, we have:
  \begin{enumerate} 
  \item The tube torsion $\tau(\Sigma) \subset \bR/\lambda_k \bZ$ is given by the collection of classes $\tau_{2k}(f) \in \bR/\lambda_k \bZ$,  where $f$ ranges over all possible generating functions of tube type $f:W \to \bR$ for $\Sigma$. 
  \item The nearby Lagrangian torsion $\nu_{2k}(\Sigma) \in \bR/\lambda_k \bZ$ is given by the class $\tau_{2k}(f) \in  \bR/\lambda_k \bZ$  for  any given natural generating function of tube type $f:W \to \bR$ for $\Sigma$.
\end{enumerate}
  \end{theorem} 

  {
\begin{figure}[htbp]
\begin{center}
\begin{tikzpicture}[scale=.7]
\begin{scope}[xshift=-5.5cm]
\draw[thick,-{stealth}] (-1,0)--(3,0);
\draw[thick,-{stealth}] (0,-1)--(0,3);
\draw (0.1,3) node[right]{$\bR$};
\draw (3,0) node[right]{$M$};
\draw (2,2.8) node{$\pi_F(\Lambda)$};
\draw[thick] (1,1.5)..controls (2,1.5) and (2.5,1.8)..(3,2.2);
\draw[thick] (1,1.5)..controls (2,1.5) and (2.5,1.2)..(3,0.8);
\end{scope}

\begin{scope}
\draw[thick] (-1,0)--(3,0);
\draw[thick] (0,-1)--(0,3);
\draw (0.1,3) node[right]{$\bR^k$};
\draw (3,0) node[right]{$M$};
\draw[thick] (1,1.5)..controls (2,1.5) and (2.5,1.8)..(3,2.2);
\draw[thick] (1,1.5)..controls (2,1.5) and (2.5,1.2)..(3,0.8);
\draw[thick] (2,2.5)..controls (2,1.7) and (2.2,.85)..(2.4,1.5)..controls (2.6,2.55) and (2.8,1)..(2.8,.5);
\draw (1,-.15)--(1,.15);
\draw (1,-.15) node[below]{$p$};
\end{scope}

\begin{scope}[xshift=5.5cm]
\draw[thick] (-1,0)--(3,0);
\draw[thick] (0,-2)--(0,2);
\foreach \x in {.4,.7,1,1.3,1.6,1.9,2.2,2.5}
\draw (\x,-2)--(\x,2);
\draw (1,-2.3)--(1,2.3);
\draw (1,2.3)node[above]{$T_p^\ast M$};
\draw (3,0) node[right]{$M$};
\draw (3.5,2.8) node[right]{$T^\ast M$};
\draw[thick] (1,0)..controls (1,1) and (2.5,1.5)..(3,1.6);
\draw (3,1.6) node[right]{$L$};
\draw[thick] (1,0)..controls (1,-1) and (2.5,-1.5)..(3,-1.6);
\end{scope}

\end{tikzpicture}
\caption{The front projection $\pi_F(\Lambda) \subset M \times \bR$ of a Legendrian $\Lambda \subset J^1M$ near a cuspidal singularities is shown on the left. The realization of $\pi_F(\Lambda)$ as the Cerf diagram of a family of functions with moderate singularities is shown in middle. The corresponding Lagrangian $L \subset T^\ast M$ has moderate tangencies with respect to the foliation by fibers $T_q^*M$, $q \in M$, as shown on the right.}
\label{Figure1905}
\end{center}
\end{figure}
}
  
  In particular, when the front of a nearby Lagrangian homotopy sphere $\Sigma \subset T^*S^{4k}$ only has cusp type singularities, it is possible to compute tube torsion and nearby Lagrangian torsion of $\Sigma$ from the parametrized Morse theory of generating functions for $\Sigma$. A more precise version of this result will be stated and proved in Section \ref{sec: nearby} below.

  Next, we introduce some terminology to distinguish from the case in which nearby Lagrangian torsion isn't defined.

\begin{definition} We say that a Legendrian $\Lambda \subset J^1 M$ has {\em trivial tube torsion} when $\tau(\Lambda) =\text{im}(p)$. We say that a nearby Lagrangian $L \subset T^*M$ has {\em trivial nearby Lagrangian torsion} when $\nu(L)=\text{im}(p)$. \end{definition}

\begin{remark} If a nearby Lagrangian $L \subset T^*M$ has trivial tube torsion, then it automatically has trivial nearby Lagrangian torsion. \end{remark}

We may formulate the resulting obstruction to Legendrian and Hamiltonian isotopy to the zero section as follows:
 
 \begin{corollary} If a Legendrian $\Lambda \subset J^1M$ is Legendrian isotopic to the zero section, then it has trivial tube torsion. In particular, if a nearby Lagrangian $L \subset T^*M$ is Hamiltonian isotopic to the zero section, then it has trivial nearby Lagrangian torsion. \end{corollary}
 
 The famous nearby Lagrangian conjecture states that any nearby Lagrangian $L \subset T^*M$ is Hamiltonian isotopic to the zero section. The following question is therefore natural.
  
 \begin{question}\label{ques: nearby Lag}
Must a nearby Lagrangian $L \subset T^*M$ have trivial nearby Lagrangian torsion? 
\end{question}

An affirmative answer to this question would provide further evidence towards the nearby Lagrangian conjecture, whereas a negative answer would disprove it. Note that according to the nearby Lagrangian conjecture the Legendrian lift $\Lambda \subset J^1M$ of a nearby Lagrangian $L \subset T^*M$ would in particular  be Legendrian isotopic to $M$, though even this weaker form of the conjecture is not known. It seems conceivable that there could exist a nearby Lagrangian with trivial nearby Lagrangian torsion but nontrivial tube torsion, which would also disprove the nearby Lagrangian conjecture. 

\begin{remark}
The invariant $\nu(L)$ is an invariant of a {\em single} nearby Lagrangian $L \subset T^*M$. It is a parametrized torsion invariant of a family of acyclic chain complexes associated to $L$. The space of parameters is $M$, and for a single parameter $m \in M$ the relevant chain complex is (at least conjecturally) $CF^*(T_mM,L)$, see Section \ref{sec: hol}. However, we take the generating function viewpoint in the present article and use the fiberwise Morse complex $C_*(f_m)$ of a natural generating function $f$ for $L$ instead. The work by Courte and Porcelli \cite{CP25} gives a vanishing result for the parametrized higher torsion of a {\em family} of nearby Lagrangians $L_z \subset T^*M$,$z \in Z$, where for a single parameter the relevant chain complex is $CF^*(M,L_z)$ for $M \subset T^*M$ the zero section, though \cite{CP25} also takes the viewpoint of generating functions. The work \cite{CP25} is closer to a parametric generalization of the result \cite{AK} of vanishing of Whitehead torsion for a single nearby Lagrangian $L \subset T^*M$, which concerns $CF^*(M,L)$. Indeed, the invariants $\tau(L)$ and $\nu(L)$ have no {\em non-parametric} analogue; the fiber of a tube bundle is simply connected so its Whitehead group is trivial. It's possible the two viewpoints may be reconciled by some assembly type map.  
\end{remark}

In the broader class of Legendrians which are not necessarily lifts of nearby Lagrangians, we show that the theory of tube torsion is nontrivial in that Legendrians with nontrivial tube torsion abound. In fact, we establish the existence of Legendrians which have nontrivial tube torsion and which are diffeomorphic to the base, so that the conjectural triviality of tube torsion for nearby Lagrangians would in any case not follow from known constraints on the topology of nearby Lagrangians.

To formulate the precise statement, we will use the viewpoint of Waldhausen \cite{W82} where a fibration sequence $\cH_\infty \to \cT_\infty \to BG$ is constructed. Here $\cH_\infty$ is the space of stable h-cobordisms of a point (which classifies bundles of h-cobordisms of a disk up to stable equivalence), $\cT_\infty$ is the stable space of tubes (which classifies tube bundles up to stable equivalence), and $BG$ is the classifying space for stable spherical fibrations, so $G=\lim_n G_n$ for $G_n$ the space of self homotopy-equivalences of $S^n$. In particular, this means given a bundle $H \to V \to M$ of h-cobordisms of a disk, one may attach $V$ to a trivial tube bundle $M \times T$, where $T$ is the standard tube, to obtain a tube bundle $T \to W \to M$. The tube bundle $W$ has the same higher torsion as $V$. 

\begin{theorem}\label{thm: existence of example} 
Let $H \to V \to M$ be a bundle of $h$-cobordisms of a disk and let $T \to W \to M$ the resulting tube bundle. There exists a Legendrian $\Lambda \subset J^1M$ which admits a generating function on $W$ and whose projection to the base $\Lambda \to M$ is homotopic to a diffeomorphism. In particular the higher torsion classes $\tau_k(V) \in H^{2k}(M;\bR)$ of $V$ are contained in the tube torsion $\tau(\Lambda) \subset H^*(M;\bR)$ of $\Lambda$.
\end{theorem}

The higher torsion of tube bundles arising from h-cobordism bundles as in Theorem \ref{thm: existence of example} is in general not contained in the image of the (normalized) Pontryagin character, and hence this result shows that Legendrians with nontrivial tube torsion exist in abundance. Unfortunately, we do not know how to visualize a single explicit example of a Legendrian with nontrivial tube torsion. The simplest possible example would live in $J^1S^4$.

 \begin{problem}\label{pros: exhibit}
Describe explicitly, or draw the front of a Legendrian in $J^1S^4$ with nontrivial tube torsion.
\end{problem}

Although we do not know whether nearby Lagrangians must always have trivial tube torsion, or even whether they must always have trivial nearby Lagrangian torsion, in the upcoming work \cite{AIS} we will prove the triviality of nearby Lagrangian torsion under the additional assumption that the front admits a ruling satisfying certain technical conditions, which we call a {\em restricted ruling}.

Finally we note that if a nearby Lagrangian $L \subset T^*M$ has a stable Gauss map $L \to U/O$ which is not nullhomotopic, then it cannot admit any generating function, so that tube torsion is not defined. However, such an $L$ would already be a counterexample to the nearby Lagrangian conjecture, since a Hamiltonian isotopy between $L$ and the zero section would in particular yield a nullhomotopy of the stable Gauss map.

\subsection{Structure of the article}\label{sec: structure}
In Section \ref{sec: Reidemeister} we define a general class of invariants which we collectively call {\em Legendrian higher torsion}. In Section \ref{sec: tube} we specialize to a specific version tailored to the nearby Lagrangian conjecture, which we call {\em tube torsion}, and prove the existence of Legendrians which are diffeomorphic to the base and which have nontrivial tube torsion. In Section \ref{sec: nearby} we identify a distinguished coset of the tube torsion of a nearby Lagrangian, which we call {\em nearby Lagrangian torsion}, and prove that it is a Hamiltonian isotopy invariant. 
In Appendix \ref{sec: moderate singularities} we give some background on functions with moderate singularities and in  Appendix \ref{sec: higher torsion of bundles} we give some background on the definition of higher torsion invariants of fiber bundles.

\subsection{Acknowledgements}
We are grateful to Sander Kupers for private communication related to the numerical comparison of regulators and higher torsion. The first author is grateful for the support of the NSF under grant DMS-2203455, the Simons Foundation and the K-Theory foundation.  The second author is grateful for the support of the Simons Foundation, grant \# 686616. The third author is grateful for the support of the Simons Foundation, grant \# 019208.

\section{Legendrian higher torsion}\label{sec: Reidemeister}

In this section we present a general construction of Legendrian invariants using higher torsion classes, which are certain characteristic classes of fiber bundles of smooth manifolds arising from algebraic K-theory. 

\subsection{Fiber bundles and fibrations at infinity}
Let $M$ be a closed, connected smooth manifold. We consider fiber bundles of smooth manifolds $F \to W \to M$ and functions $f:W \to \bR$ on the total space, which we will often think of as families of functions $f_m:F_m \to \bR$ on the fiber $F_m \simeq F$, parametrized by $m \in M$. In the simplest case $F$ is a closed manifold. When $F$ is not a closed manifold we will have to add conditions to the possible functions $g:F \to \bR$ we consider. 
We use the terminology of \cite{EG98} in the following definition.

\begin{definition}
\label{def: Fibration at Infinity}
A function $f:F \to \bR$ is a  {\em fibration at infinity} if there exists a finite segment $[-a,a] \subset \bR$ and a compact subset $K \subset f^{-1}[-a,a] \subset F$ such that the restriction of $f$ to the following three subsets fibers them over their respective images:
\begin{enumerate}
\item $f^{-1}(-\infty,a] \to (-\infty,a]$
\item $f^{-1}[a,\infty) \to [a,\infty)$
\item $(f^{-1}[-a,a] ) \setminus K \to [-a,a]$
\end{enumerate}
\end{definition}

As noted in \cite[Section 0.2.1]{EG98}, a function $f:F \to \bR$ is a fibration at infinity if and only if $F$ admits a complete Riemannian metric with respect to which $\| df \| \geq \varepsilon>0$ outside of a compact subset.

\begin{example} The following are fibrations at infinity.
\begin{enumerate}
\item A non-degenerate quadratic form $q:\bR^N \to \bR$.
\item A function $f:\bR^N \to \bR$ of class $C^1$ which is 2-homogeneous, i.e. $f(\lambda v) = \lambda^2f(v)$, outside of a compact subset $K \subset \bR^N$ and such that the function $g:S^{N-1} \to \bR$ given by $g(v)=\lim_{\lambda \to +\infty} f(\lambda v)$ is a smooth function which has $0$ as a regular value. 
\item A function $f:\bR^N \to \bR$ which is a bounded $C^1$-distance from a function as in (2). 
\item The direct sum of two fibrations at infinity on the product domain.
\end{enumerate}
\end{example}

For a first contact with higher Reidemeister torsion it will be sufficient to consider, for $F$ a closed manifold, the class of fiber bundles $W \to M$ over $M$ whose fiber is diffeomorphic to $F$, together with their stabilizations. By the stabilization of a fiber bundle we will mean the fiberwise product with a trivial $\bR^{4N}$ bundle, which yields a new fiber bundle
$$ F \times \bR^{4N} \to W \times \bR^{4N} \to M.$$
On this bundle (with non-compact fibers) we distinguish the following class of functions.

\begin{definition}\label{def: class}
Given $M$ and $F$ closed, smooth manifolds and $F \to W \to M$ a fiber bundle, we denote by $\mathcal{S}(W)$ the class of all of functions $f:W \times \bR^{4N} \to \bR$ which are fibrations at infinity of the form
$$
f(w, x, y) = g(w) + \varepsilon(w,x,y) + \|x\|^2 - \|y\|^2, \qquad (w,x,y) \in W \times  \bR^{2N} \times \bR^{2N} 
$$
where $g:W \to \bR$ is a smooth function and $\varepsilon$ is a smooth function with compact support. Functions $f:W \times \bR^{4N} \to \bR$ in the class $\mathcal{S}(W)$ will be said to be {\em{admissible.}}
\end{definition}

\begin{remark} Note that if $f:W \times \bR^{4N} \to \bR$ is admissible then in particular with respect to any Riemannian metric on $W$ it remains a bounded $C^1$-distance from the standard fiberwise quadratic form: $$q(w, x,y) = \|x\|^2-\|y\|^2, \qquad (w,x,y) \in W \times  \bR^{2N} \times \bR^{2N}.$$ \end{remark}

\begin{remark}
If instead of starting with a fiber bundle of closed manifolds $F \to W \to M$ we started with a fiber bundle whose fiber $F$ has boundary (or is non-compact), then some fixed asymptotic behavior near the boundary (or at infinity) for the functions $g:W \to \bR$ as in the above definition should be specified. 
\end{remark}

\begin{warning} We should really write $\mathcal{S}(\pi)$ for $\pi:W \to M$ the fibration but we will abuse notation and use $\mathcal{S}(W)$  instead of $\mathcal{S}(\pi)$ throughout. \end{warning}

\subsection{The Legendrian invariant}\label{sec: Reid inv}



 Let us first consider the case where the fiber $F$ is closed.
 
 
 \begin{definition}
 A {\em torsion pair} $(W,\rho)$ consists of a fiber bundle $F \to W \to M$ and a unitary local system  $\rho : \pi_1 W \to U(1)$ such that $H_*(F;\bC^\rho)$ is a unipotent $\pi_1M$-module.
 \end{definition}
 
Here $H_*(F;\bC^\rho)$ denotes the homology of the fiber with the twisted complex coefficients given by $\rho|_{\pi_1F}$ and by {\em unipotent} we mean that the $\pi_1M$-module $H_*(F;\bC^\rho)$ admits a filtration by $\pi_1M$ submodules so that the action of $\pi_1M$ on the successive subquotients is trivial (a basepoint $b \in M$ should be chosen but this choice is immaterial). 
 
 \begin{warning}
We should really write $(\pi,\rho)$ for $\pi:W \to M$ the fibration but we will abuse notation and write $(W, \rho)$ instead of $(\pi, \rho)$ throughout. 
 \end{warning}
 
 Given a torsion pair $(W,\rho)$, one has its higher Reidemeister torsion, which consists of cohomology classes $\tau_k(W,\rho) \in H^{2k}(M;\bR)$, whose construction is reviewed in Appendix \ref{sec: higher torsion of bundles}. 
 
 By a {\em Legendrian submanifold} we will mean throughout a Legendrian embedding $\Lambda \hookrightarrow J^1M$ of a closed smooth manifold $\Lambda$, i.e. compact and without boundary. By an {\em immersed Legendrian} we will mean a Legendrian immersion $\Lambda \to J^1M$ of a closed smooth manifold $\Lambda$, though we will abuse notation and denote both embedded and immersed Legendrians by $\Lambda \subset J^1 M$. To make the connection between functions on fiber bundles over $M$ and Legendrians in $J^1M$ we need an additional condition on the functions we consider. The condition is generic, and is the standard definition of a {\em generating function} for a Legendrian in a 1-jet space. Recall that the {\em vertical differential} of a function $f:W \to \bR$ on the total space of a fiber bundle $\pi: W \to M$ is the restriction $d_vf$ of the usual differential $df:TW \to \bR$ to the vertical distribution $\ker(d \pi) \subset TW$. 

\begin{definition} A function $f:W \to \bR$ is said to be a {\em{generating function}} if the vertical differential $d_v f$ is transverse to zero. 
\end{definition}

If $f$ is a generating function, then $f$ determines a Legendrian immersion from the fiberwise singular locus $\Sigma_f = \{ d_v f = 0 \} \subset W$ into $J^1M$. Indeed, the condition that $d_v f \pitchfork 0$ means that the graph $\Gamma(df) \subset T^*W$ intersects transversely the coisotropic sumanifold $X \subset T^*W$ consisting of covectors which are zero on the fibers $T_mF$, $m \in M$. The symplectic reduction of the coisotropic submanifold $X$ is naturally identified with $T^*M$ and the image of the submanifold $X \cap \Gamma(df) =\Gamma(df)|_{\Sigma_f}$ under this reduction is an immersed exact Lagrangian submanifold of $T^*M$ whose Legendrian lift (via the values of $f$) is an immersed Legendrian submanifold $\Lambda \subset J^1M$.  

When moreover $\Lambda$ is an embedded Legendrian submanifold of $J^1M$ (which is the case generically), one says that the Legendrian submanifold  $\Lambda \subset J^1M$ is {\em generated} by $f$. 

Now fix a collection $\cF$ of torsion pairs $(W,\rho)$. Examples of reasonable collections to consider include but are not limited to:
\begin{enumerate}
\item The singleton consisting of a fixed $W$ and a fixed $\rho$.
\item The collection of all $(W,\rho)$ for a fixed $W$ but all possible $\rho$ for that $W$. 
\item The collection of all possible torsion pairs $(W,\rho)$ with fiber a fixed closed manifold $F$.
\end{enumerate}

\begin{definition}  Let  $\Lambda \subset J^1M$ be a Legendrian submanifold. The {\em{Legendrian higher torsion}} of $\Lambda$ in the class $\cF$ is the subset $\tau(\Lambda,\cF) \subset H^*(M;\bR)$ consisting of all higher Reidemeister torsion classes $\tau_k(W,\rho) \in H^{2k}(M;\bR)$ of torsion pairs $(W,\rho)$ in $\cF$ such that $\Lambda$ admits a generating function $f \in \mathcal{S}(W)$.
\end{definition}

\begin{theorem}
If $\Lambda_0,\Lambda_1 \subset J^1M$ are Legendrian isotopic, then $\tau(\Lambda_0,\cF)=\tau(\Lambda_1,\cF)$. 
\end{theorem}

\begin{proof}As follows from the homotopy lifting property for generating functions under Legendrian isotopy up to stabilization, if $\Lambda_0$ admits a generating function in the class $\mathcal{S}(W)$, then so does $\Lambda_1$. The desired result then follows immediately. The homotopy lifting property for generating functions under Legendrian isotopies has a long history \cite{S86, C96}, but the formulation which is needed at this level of generality may be found in \cite{EG98}.
\end{proof}

 
 \subsection{Suitable classes of fiber bundles}\label{sec: class}

Of course, given a Legendrian $\Lambda \subset J^1M$, for most choices of $\cF$ we will have $\tau(\Lambda)=\varnothing$, since many Legendrians don't admit any generating function and even those which do will only admit generating functions on certain fiber bundles. More generally, if one wishes to use generating function theory to study a particular class of Legendrians, then one should make a judicious choice for the class $\cF$ of fiber bundles and notion of admissible generating function under consideration.

As a first example, if $\Lambda$ is disjoint from a single fiber $T_xM \times \bR \subset J^1M$ then it won't admit a generating function on any fiber bundle $F \to W \to M$ with fiber $F$ a closed manifold, even after stabilization. In particular, if one wishes to study compact Legendrian submanifolds of $J^1\bR^n=\bR^{2n+1}$ one should consider a theory of generating functions with zero fiberwise Morse homology. In the $n=1$ case one could consider as in \cite{JT} generating functions $f:\bR \times \bR^{k+1}\to \bR$ on the trivial $\bR^{k+1}$-bundle over $\bR$ which are linear-quadratic at infinity, i.e. such that for $(w,v,z) \in \bR \times \bR^k \times \bR$ there holds outside of a compact subset $f(w,v,z)=Q(v)+L(z)$ for $Q$ a nondegenerate quadratic form and $L$ a nondegenerate linear form. 

If instead one is interested in studying Legendrian braids in $J^1S^1$, then one could consider  as in  \cite{T01} generating functions $f:S^1 \times \bR^k \to \bR$ on the trivial $\bR^k$-bundle over $S^1$ which are quadratic at infinity, i.e. such that for $(w,v) \in S^1 \times \bR^k$ there holds outside of a compact subset $f(w,v)=Q(v)$ for $Q$ a nondegenerate quadratic form.

Other classes of fiber bundles may be useful for other purposes. We explain the simplest example of a nontrivial application of the theory of Legendrian higher torsion known to us, in which the class of fiber bundles under consideration consists of circle bundles over orientable surfaces. In \cite{AI} the first two authors investigated a class of Legendrian submanifolds which admit generating functions on such circle bundles. Given an oriented circle bundle over an oriented surface $S^1 \to W \to \Sigma$, with Euler number $n\in \bZ$, and given a rank 1 unitary local system $\rho : \pi_1 W \to U(1)$ such that the oriented fiber $S^1$ gets sent to $u^{-1}$ for $u \neq 1$ an $n$-th root of unity, the higher Reidemeister torsion $\tau_1(W,\rho) \in H^2(\Sigma;\bR) \simeq \bR$ is a constant multiple of the Euler number of the circle bundle up to sign, as was computed by the second author and J. Klein in \cite{IK93a, IK93b}. Concretely, we have
$$ \tau_1(W,\rho) = n \, \,  \text{Im} \left(  \mathcal{L}_2(u)  \right) $$
where $\mathcal{L}_2$ is the dilogarithm
$$ \mathcal{L}_2(z) = \sum_{k=1}^\infty \frac{z^k}{k^2} .$$
That the higher Reidemeister torsion up to sign survives as a stable invariant of the circle bundle and is equal to the Euler class is a nontrivial fact. This fact is a special case of a deep result of B\"okstedt \cite{B82}, and may also be deduced from the theory of higher Reidemeister torsion, see \cite{I02}, concretely Chapter 7 (complexified B\"okstedt theorem). In  \cite{IK93a, IK93b} a more direct calculation for this special case is given.

The class of Legendrians $\Lambda \subset J^1 \Sigma$ which were studied in \cite{AI} are called {\em mesh Legendrians}. Each mesh Legendrian $\Lambda$ is determined by combintarial data and admits an explicit generating function on a circle bundle of Euler number  $\pm e_\Lambda$ for a non-negative integer $e_\Lambda  \in \bZ_{\geq0}$ which is defined in terms of the combinatorial data. In \cite{AI} a calculation was carried out to determine the possible 2-parametric fiberwise Morse complex bifurcations that a generating function for the mesh Legendrian $\Lambda$ could support, on any such circle bundle over $\Sigma$. 

In particular it was shown that for any $n \in \bZ_{\geq 0}$ a mesh Legendrian $\Lambda$ only admits a generating function on the circle bundle $S^1 \to W \to \Sigma$ of Euler number $\pm n$ if $n=e_\Lambda$, {\em even after stabilization} of the circle bundle. Hence for  the singleton class $\cF=\{(W,\rho)\}$ consisting of a circle bundle  $S^1 \to W \to \Sigma$ of Euler number $\pm n$ and  $\rho : \pi_1 W \to U(1)$  a rank 1 unitary local system which sends the oriented fiber $S^1$ to $u^{-1}$ for $u \neq 1$ an $n$-th root of unity, the Legendrian higher torsion of a mesh Legendrian $\Lambda \subset J^1 \Sigma$ in the class $\cF$ consists of the singleton
 $$\tau_1(\Lambda,\cF)=\{ e_\Lambda \text{Im}( \mathcal{L}_2(u))  \} \subset H^2(\Sigma;\bR) \simeq \bR$$
 if $n=e_\Lambda$, and is empty otherwise. This proves:
 \begin{theorem}[implicit in \cite{AI}] The Legendrian higher torsion of a mesh Legendrian $\Lambda \subset J^1 \Sigma$ recovers (and uniquely determines) the Euler number (up to sign) of the unique circle bundle that may be used to generate it, even after stabilization of the tube bundle and/or Legendrian isotopy of $\Lambda$.  \end{theorem}

In fact in this specific situation the invariant can be refined to a Turaev type torsion which  is also sensitive to the sign of the Euler number,  and in this refined version the invariant was defined and used in \cite{AI} to distinguish Legendrian isotopy classes of mesh Legendrians in $J^1\Sigma$ in the same formal Legendrian class. The computations of \cite{AI} can be thought of as a proof of concept that higher torsion invariants may sometimes be calculated for explicit Legendrians. 

Our final example of a class $\cF$ of fiber bundles which one may want to consider for the definition of a Legendrian higher torsion invariant, and the most important class for the purposes of the present article, is the class of {\em tube bundles}. Roughly speaking, tube bundles are the most general class of fiber bundles on which one might hope to construct generating functions for {\em nearby Lagrangians}. Hence tube bundles provide the natural class of fiber bundles in which to define a Legendrian higher torsion invariant that may be relevant in the context of the  nearby Lagrangian conjecture. We discuss tube bundles and the associated Legendrian higher torsion invariant, which we call {\em tube torsion}, in Section \ref{sec: tube} below.

\subsection{Stable Gauss map} \label{sec: stable gauss map}

Given a Legendrian $\Lambda \subset J^1M$, for the Legendrian higher torsion to be defined, i.e. for $\tau(\Lambda,\cF)$ to be nonempty, in addition to the necessity of making a suitable choice for the class $\cF$ of fiber bundles (and unitary local systems) under consideration, it is also necessary that the stable Gauss map $\Lambda \to U/O$ is trivial, i.e. nullhomotopic Indeed, the triviality of the stable Gauss map is a necessary condition for the existence of {\em any} generating function, even in a very weak sense. Let us briefly recall the definition of the stable Gauss map.

 Let $\pi_{\text{Lag}} : J^1M = T^*M \times \bR \to T^*M$ denote the Lagrangian projection and consider the symplectic vector bundle $E=\pi_{\text{Lag}} ^*T(T^*M)|_\Lambda$ over $\Lambda$. Note that $E$ is equipped with two polarizations, i.e. Lagrangian sub-bundles, namely $G=T\Lambda$ (since $\pi_{\text{Lag}} |_\Lambda : \Lambda \to T^*M$ is a Lagrangian immersion) and $V=\pi_{\text{Lag}} ^*\nu|_\Lambda$, where $\nu$ is the Lagrangian distribution on $T^*M$ tangent to the fibers of the cotangent bundle projection $T^*M \to M$. The stable Gauss map measures the stable difference between these two polarizations. 
 
 Concretely, pick a Lagrangian sub-bundle $H$ of $E$ which is transverse to $V$, so that we have a isomorphism of symplectic vector bundles $E \simeq H \oplus H^*$ under which $V$ gets identified with $H^*$. There exists some real vector bundle $H'$ over $\Lambda$ such that $H \oplus H'$ is a trivial real vector bundle. Then $H' \oplus (H')^* \oplus E \simeq (H \oplus H') \oplus (H \oplus H')^*$ is isomorphic to a trivial symplectic vector bundle $\Lambda \times \bC^m$. The Lagrangian sub-bundle $H' \oplus G  \subset (H' \oplus (H')^*) \oplus E$ may thus be viewed as a Lagrangian sub-bundle of the trivial symplectic vector bundle $\Lambda \times \bC^m$, and hence gives a map $\Lambda \to U_m/O_m$. After letting $m \to \infty$ we get the stable Gauss map $\Lambda \to U/O$, which is easily verified to be independent of choices up to homotopy.  

We say that the stable Gauss map $\Lambda \to U/O$ of a Legendrian $\Lambda \subset J^1M$ is {\em trivial} when it is nullhomotopic. That the triviality of the stable Gauss map $\Lambda \to U/O$ is a necessary condition for $L$ to admit any kind of generating function was shown by Giroux in \cite{G90}. In particular, no matter what  the class of fiber bundles $\cF$ under consideration is, for  Legendrian higher torsion to be defined, i.e. for $\tau(\Lambda,\cF)$ to be nonempty, we must have that the stable Gauss map $\Lambda \to U/O$ is trivial. 

We will say that the front of Legendrian $\Lambda \subset J^1M$ has {\em moderate singularities} if the front projection $$\Lambda \hookrightarrow J^1M \to J^0M = M \times \bR$$ only has cusp singularities. This property is closely related to the triviality of the stable Gauss map. Indeed, for $\Lambda \subset J^1M$ a Legendrian with moderate singularities, consider the co-oriented hypersurface $(S,\nu)$ where $S \subset \Lambda$ is the locus of points whose front have a cusp and $\nu$ is the Maslov co-orientation of $S$ in $\Lambda$, see \cite{A}. For simplicity let us assume that $\Lambda$ is closed, connected and orientable. The co-oriented cycle $(S,\nu)$ determines a class in $H_{n-1}(\Lambda;\bZ)$, where $n=\dim \Lambda$, which is Poincar\'e dual to the Maslov class $\mu_1 \in H^1(\Lambda,\bZ)$. We have $\mu_1=0$ if and only if there exists a bounding chain for $(S,\nu)$, i.e. a collection of subdomains $\Omega_1, \ldots, \Omega_k \subset \Lambda$ with disjoint boundaries such that $\coprod_j \partial \Omega_j = S$ and for each component of $S$ realized as (part of) the boundary of $\Omega_j$ the Maslov-coorientation $\nu$ on this component agrees with the outwards pointing coorientation determined by $\Omega_j$. We then have:

\begin{proposition}\label{prop: simp stab triv}
Let $\Lambda$ be a closed, connected, smooth manifold and $\Lambda \subset J^1M$ a Legendrian embedding. If the front of $\Lambda$ has moderate singularities and if its Maslov class $\mu_1 \in H^1(\Lambda,\bZ)$ is trivial, then the stable Gauss map $\Lambda \to U/O$ is trivial. In fact, any bounding chain for the pair $(S,\nu)$ determines a homotopy class of nullhomotopy of the Gauss map $\Lambda \to U/O$. \end{proposition} 

\begin{proof}
When $\Lambda$ has the homotopy type of a sphere of dimension $>1$ we have $\mu_1=0$ for trivial reasons and the argument is given in \cite{AD22}. The general case may be proved by a similar line of reasoning, see \cite{B27} for details. \end{proof}

\begin{remark}
To be more precise, the statement of Proposition \ref{prop: simp stab triv} should be that any bounding chain determines a homotopy of Lagrangian distributions between the stabilizations of $G$ and $V$ in the stabilization of the symplectic vector bundle $E$, with notation as in the definition of the stable Gauss map above. So although the stable Gauss map $\Lambda \to U/O$ is only well-defined up to homotopy, we may still speak of a homotopy class of nullhomotopy of the stable Gauss map by defining this to be a homotopy of Lagrangian distributions between the stabilizations of $G$ and $V$. 
\end{remark}

Note that if $L \subset T^*M$ is a nearby Lagrangian, then the front of the Legendrian lift of $L$ has moderate singularities if and only if the tangencies of $L$ with respect to the foliation of $T^*M$ by the cotangent leaves $T^*_mM$, $m \in M$, has {\em quadratic} tangencies. When this is the case we will say that $L \subset T^*M$ has {\em moderate tangencies}. From Proposition \ref{prop: simp stab triv} and the fact that nearby Lagrangians always trivial Maslov class \cite{A12} it follows that any nearby Lagrangian with moderate tangencies has stably trivial Gauss map $L \to U/O$, and moreover there is a distinguished choice of homotopy class of nullhomotopy given any bounding chain for the cusp locus.

\section{Legendrian tube torsion}\label{sec: tube}

In this section we focus on a specific version of Legendrian higher torsion which we call {\em Legendrian tube torsion}.

\subsection{Generating functions for nearby Lagrangians}\label{sec: gen fun nearby}

We recall that a nearby Lagrangian is an exact Lagrangian submanifold $L \subset T^*M$, where $L$ and $M$ smooth manifolds which are closed, connected and of the same dimension. The nearby Lagrangian conjecture states that any nearby Lagrangian is isotopic through nearby Lagrangians to the zero section.

For a nearby Lagrangian $L \subset T^*M$ to admit a generating function on a fiber bundle $F \to W \to M$, it is necessary that the fiberwise Morse homology of the generating function have Euler characteristic $\pm 1$. Indeed, the projection $L \to M$ is known to be a homotopy equivalence \cite{FSS, Na09, A12, Kra13, Gui12}, and in particular has degree $\pm 1$.

The most basic example to keep in mind is the trivial bundle $M \times \bR^{4N}$ with a generating function $f:M \times \bR^{4N} \to \bR$ which is equal to the standard fiberwise quadratic form $\|x\|^2-\|y\|^2$ outside of a compact subset, $(x,y) \in \bR^{2N} \times \bR^{2N}$. We will call such generating functions {\em standard quadratic at infinity}.  Note that the zero-section can be generated by such a function (namely the constant family of standard quadratic forms $F(m,x,y)=\|x\|^2-\|y\|^2$). Therefore, if the nearby Lagrangian conjecture is true, then by the homotopy lifting property for generating functions up to stabilization it would follow that any nearby Lagrangian admits a generating function standard quadratic at infinity.

While existence of such generating functions for nearby Lagrangians is not known, there has been recent progress on existence of generating functions for nearby Lagrangians on the more general class of so-called {\em tube bundles}. The class of tube bundles is a class of fiber bundles which was introduced by Waldhausen in his manifold approach to the algebraic K-theory of spaces \cite{W82}. The precise definition of a tube bundle that we will work with in the present article will be given in Section \ref{sec: tube bundles} below, but let us note that in particular the fiberwise Morse homology of our notion of {\em generating function of tube type} has rank 1 concentrated in a single degree, just like a nondegenerate quadratic form.  The foundational result of the article \cite{ACGK} by Abouzaid, Courte, Guillermou and Kragh is the following. 

\begin{theorem}[\cite{ACGK}]\label{thm: existence generating functions tube}  A nearby Lagrangian $L \subset T^*M$ with trivial stable Gauss map $L \to U/O$ admits a generating function of tube type. \end{theorem}

So the condition that $L \to U/O$ is trivial, which as explained in Section \ref{sec: stable gauss map} is necessary for the existence of {\em any} kind of generating function, is in fact both necessary and sufficient for a nearby Lagrangian $L \subset T^*M$ to admit a generating function of tube type.  It is not known whether every nearby Lagrangian $L \subset T^*M$ must have stably trivial Gauss map. However, it was shown in \cite{ACGK} that if  $L \subset T^*M$ is a nearby Lagrangian, then $L \to U/O$ is trivial on homotopy groups.  In particular, if $M$ is a homotopy sphere, then any nearby Lagrangian $L \subset T^*M$ has stably trivial Gauss map and hence there holds:

\begin{theorem}[\cite{ACGK}] If $M$ is a homotopy sphere, then any nearby Lagrangian $L \subset T^*M$ admits a generating function of tube type. \end{theorem} 

Finally, the discussion at the end of Section 3 in \cite{ACGK}  (see the proof of Theorem 3.26 and preceding lemmas) makes it apparent that tube bundles are in some sense the most general class of fiber bundles on which one may expect a nearby Lagrangian to admit a generating function. The class of tube bundles therefore seems a reasonable class of fiber bundles to study in the context of the nearby Lagrangian conjecture, particularly in the case of nearby Lagrangian homotopy spheres. 

\subsection{Functions of tube type}\label{sec: tube bundles}

Rather than the model for tube bundles and functions of tube type used in \cite{ACGK} we  will use a convenient model for the space of tubes which was developed in work of the first author joint with Abouzaid, Courte and Kragh in \cite{AACK}. That Theorem \ref{thm: existence generating functions tube} also holds for this model was proved in \cite{AACK}. First we will need to introduce a few preliminary notions. 

\begin{definition}\label{def: quad tube}
A $C^1$-function $g:\bR^N \to \bR$ is a {\em{function of quadratic tube type}} if it satisfies the following:
\begin{enumerate}
\item $g$ is homogeneous of degree 2, i.e. $g(\lambda v) = \lambda^2 g(v)$.
\item $g$ has no critical point other than the origin.
\item $g$ is homotopic through functions satisfying (1) and (2) to a non-degenerate quadratic form. 
\end{enumerate}
\end{definition}

Note that in general the homogenization of a smooth function $h:S^{N-1} \to \bR$, i.e. $g(v)=\|v\|^2h(v/\|v\|)$ for $v \neq 0$ and $g(0)=0$, will only be of class $C^1$. 

\begin{definition}\label{def: tube}
A smooth function $f:\bR^N \to \bR$ is a {\em function of tube type} if there exists a function of quadratic tube type $g:\bR^N \to \bR$ such that $\| \nabla f - \nabla g \| \leq C$ for some constant $C>0$.
\end{definition}

\begin{remark} Note that if $g$ exists, then it is uniquely determined by $f$, indeed $g(v) = \lim_{\lambda \to \infty}  \frac{1}{\lambda^2}f(\lambda v)$. \end{remark}

Let $\cT_{k,N}^q$ denote the space of functions of quadratic tube type on $\bR^N$ such that the non-degenerate quadratic form in condition (3) has index $k$, and let $\cT_{k,N}$ be the space of functions of tube type on $\bR^N$ such that the associated function of quadratic tube type is in $\cT_{k,N}^q$. We also denote:
$$ \cT^q_N = \coprod_{k=0}^N \cT^q_{k,N}, \quad \cT^q = \coprod_{N=0}^\infty \cT^q_N , \quad \cT_N = \coprod_{k=0}^N \cT_{k,N}, \quad \cT = \coprod_{N=0}^\infty \cT_N. $$
Note that if $f_i:\bR^{N_i} \to \bR$ is a function of tube type for $i=1,2$, then $f_1 \oplus f_2 : \bR^{N_1} \times \bR^{N_2} \to \bR$ is also of tube type, so we have an operation $\cT_{k_1,N_1} \times \cT_{k_2,N_2} \mapsto \cT_{k_1+k_2,N_1+N_2}$ which endows $\cT$ with the structure of a topological monoid. We also have the topological submonoid of functions of quadratic tube type $\cT^q \subset \cT$.  Clearly the inclusion $\cT^q \subset \cT$ is an equivalence as the linear interpolation $(1-t)f+tg$ of a function of tube type $f$ with its associated function of quadratic tube type $g$ gives a deformation retraction of $\cT$ onto $\cT^q$.

We also denote $\cQ_{k,N}$ for the space of non-degenerate quadratic forms $q:\bR^N \to \bR$ of index $k$ and
$$ \cQ_N = \coprod_{k=0}^N \cQ_{k,N} , \qquad \cQ= \coprod_{N=0}^\infty \cQ_N. $$
Note that $\cQ_{k,N} \subset \cT_{k,N}^q$, and that $\cQ$ is a submonoid of $\cT^q$ and hence also of $\cT$.   

\begin{definition}
Let $M$ be a closed, connected, smooth manifold.
\begin{enumerate}
\item A {\em{fibered function of tube type}}  is a smooth function $f:M \times \bR^N \to \bR$ such that for each $m \in M$ the restriction $f_m : \bR^N \to \bR$ is of tube type, i.e. such that $f_m \in \cT_N$ for each $m \in M$. 
\item A {\em{fibered function of quadratic tube type}}  is a $C^1$ function $f:M \times \bR^N \to \bR$ such that for each $m \in M$ the restriction $f_m : \bR^N \to \bR$ is quadratic of tube type, i.e. such that $f_m \in \cT_N^q$ for each $m \in M$. 
\item A {\em{fibered function of rigid tube type}}  is a smooth function $f:M \times \bR^N \to \bR$ such that for each $m \in M$ the restriction $f_m : \bR^N \to \bR$ is a non-degenerate quadratic form, i.e. if $f_m \in \cQ_N$ for each $m \in M$. 
\end{enumerate}
\end{definition}

\begin{remark} Since $M$ is connected, the index $k$ of any such fibered function of tube type is well-defined. \end{remark}


To a fibered function of tube type $f:M \times \bR^N \to \bR$ we may associate the fibered function of quadratic tube type $g:M \times \bR^N \to \bR$ such that for each $m \in M$ the restriction $g_m:\bR^N \to \bR$ is the function of quadratic tube type associated to the restriction $f_m:\bR^N \to \bR$. 

Fibered functions of tube type over $M$ form a monoid. Fibered functions of quadratic tube type over $M$ form submonoid of the monoid of fibered functions of tube type over $M$, and fibered quadratic functions of rigid tube type form yet a smaller submonoid.

We will reserve the following terminology for the monoid operation on fibered functions of tube type when one of the two functions is rigid. 

\begin{definition}
\label{def: TwistedSister}
Let $f:M \times \bR^{N_1} \to \bR$ be a fibered function of tube type and let $q:M \times \bR^{N_2}\to \bR$ be a fibered function of rigid tube type. We call $f\oplus q:M \times \bR^{N_1} \times \bR^{N_2} \to \bR$ the {\em{twisted stabilization}} of $f$ by $q$. 
\end{definition}

By a {\em{stabilization}} without the adjective {\em twisted} we will mean a stabilization with respect to a constant fibered function of rigid tube type $q:M \times \bR^N \to \bR$, i.e. $q(x,v)=Q(v)$ for some fixed non-degenerate quadratic form $Q: \bR^N \to \bR$. In fact, typically we will stabilize by the standard form on $\bR^4$ given by  $q(x,y)=\|x\|^2-\|y\|^2$ for $(x,y) \in \bR^{2} \times \bR^{2}$, and will refer to such a stabilization as a {\em standard stabilization}.

\begin{definition}
Let $q:M \times \bR^N \to \bR$ be a fibered function of rigid tube type, so that $q(m,x)=Q_m(x)$ for $Q_m:\bR^N \to \bR$ a family of non-degenerate quadratic forms. Let $E \to M$ be the real vector bundle whose fiber $E_m$ over $m \in M$ is the negative eigenspace of the matrix $(\partial^2 Q_m/\partial x_i\partial  x_j)_{i,j}$. We call $E \to M$ the {\em stable bundle} of $Q$.
\end{definition}

\begin{remark}
Alternatively, a fibered function of ridid tube type is classified by a map $M \to \cQ$, the notion of stable bundle for a non-degenerate quadratic form determines a map $\cQ_{k,N} \to BO(k,N)$, and hence the stable bundle of a fibered function of ridid tube type is classified by the composition $M \to \cQ \to \bZ \times BO$. Note that the map $\cQ_{k,N} \to BO(k,N)$ is a homotopy equivalence, with homotopy inverse the map $BO(k,N) \to \cQ_{k,N}$ which sends a vector subspace $E \subset \bR^N$ of dimension $k$ to the quadratic form $Q(x)=\|x_{E^\perp} \|^2-\|x_E \|^2$, where $x=x_E+x_{E^\perp}$ is the decomposition of $x \in \bR^N$ into its $E$ and $E^{\perp}$ components.
\end{remark}

\subsection{Torsion of fibered functions of tube type}

Note that the second condition in Definition \ref{def: quad tube} is equivalent to asking that $g$ is transverse to zero outside of the origin, so for a function of quadratic tube type $g:\bR^N \to \bR$ the set $T(g)=\{g \leq 0 \} \cap S^{N-1}$ is a codimension-zero submanifold of $S^{N-1}$ (a priori only $C^1$, but $C^1$ submanifolds have unique smoothings up to contractible choice of isotopy). By the third condition of Definition \ref{def: quad tube} this submanifold is isotopic in $S^{N-1}$ to $T(q)=\{ q \leq 0 \} \cap S^{N-1}$ for $q$ a non-degenerate quadratic form. If $E^-$ is the negative eigenspace of $q$, then this domain is a tubular neighborhood of the equator $E^- \cap S^{N-1}$ and in particular is diffeomorphic to $S^{k-1} \times D^{N-k}$ for $k$ the index of $q$ and $D^m=\{ \|x\| \leq 1\} \subset \bR^m$ the unit disk. 

\begin{definition}
A {\em{rigid tube}} is the codimension zero submanifold $T(q) \subset S^{N-1}$ for $q$ a non-degenerate quadratic form. The {\em index} of a rigid tube is the index of $q$.
\end{definition}

\begin{definition} A {\em tube} is a codimension zero submanifold $T \subset S^{N-1}$ which is isotopic in $S^{N-1}$ to a rigid tube. The {\em index} of a tube is the index of any rigid tube it is isotopic to.
\end{definition}

{
\begin{figure}[htbp]
\begin{center}
\begin{tikzpicture}[scale=1.5]
\begin{scope}
\clip (1,4) circle[radius=1.6cm];
\draw[dashed,thick](1,4) ellipse[x radius=1.6cm,y radius=.5cm];
\draw[fill,gray!50!white] (1,3.6) ellipse[x radius=1.55cm,y radius=.5cm];
\draw[fill,gray!50!white] (1,3.8) ellipse[x radius=1.7cm,y radius=.5cm];
\draw[fill,white] (1,4) ellipse[x radius=1.6cm,y radius=.5cm];
\begin{scope}
\clip (-.5,3.8) rectangle (2.4,4.8);
\draw[dashed] (1,3.6) ellipse[x radius=1.55cm,y radius=.5cm];
\end{scope}
\draw[fill,white] (.35,3) rectangle (2,4.6);
\draw[fill,gray!50!white] (.34,3.15)..controls(.8,3.1)and(.9,3.8)..(.85,4.2)..controls(.85,4.6)and(1.5,4.6)..(1.35,4.2)..controls (1.3,4)and(1.35,3.2)..(2.03,3.227)--(2.03,3.615)..controls(1.7,3.6) and (1.75,3.9)..(1.85,3.9)..controls (2.1,4) and (2.1,5)..(1,5)..controls (.34,5)and(.34,4.27)..(.4,4.2)..controls (.5,4)and(.47,3.53)..(.34,3.545)--(.34,3.15);
\draw[thick] (1,4) circle[radius=1.6cm];
\end{scope}

\end{tikzpicture}
\caption{A tube $T \subset S^{N-1}$.}
\label{Figure1909}
\end{center}
\end{figure}
}

{
\begin{figure}[htbp]
\begin{center}
\begin{tikzpicture}[scale=.7]
%
\begin{scope}[xshift=9cm]
\clip (0,0) circle[radius=2cm];
\draw[dashed] (0,.3) ellipse[x radius=2cm,y radius=5mm];
\draw[fill,gray!50!white] (0,-.5) ellipse[x radius=2cm,y radius=5mm];
\draw[fill,gray!50!white] (-2,-.5)rectangle(2,.5);
\begin{scope}
\clip (-2,-.5)rectangle(2,.51);
\draw[fill,white] (0,.5) ellipse[x radius=2cm,y radius=5mm];
\end{scope}
\draw[thick] (0,0) circle[radius=2cm];
\end{scope}
\begin{scope}[xshift=9cm, yshift=7cm] 
\draw[thick] (0,0) circle[radius=1.5cm];
\draw[fill,gray!50!white] (1,0)ellipse[x radius=2.5mm,y radius=5mm];
\draw[thick](1,0)ellipse[x radius=2.5mm,y radius=5mm];
\draw[dashed,thick] (-1,0)ellipse[x radius=2.5mm,y radius=5mm];
\end{scope}
{
\begin{scope}[yshift=7cm] 
\begin{scope}
\draw[thick](1,0)ellipse[x radius=2.5mm,y radius=5mm];
\draw[fill,white] (1,-.6) rectangle (1.3,.6);
\draw[dashed,thick] (-1,0)ellipse[x radius=2.5mm,y radius=5mm];
\draw[thick] (0,0) circle[radius=1.5cm];
\end{scope}
\begin{scope}
\clip (1,-1.5)rectangle(2.4,1.5);
\draw[fill,white] (0,.2)..controls (.8,.15) and (1.9,1.2)..(2.3,1.7)--(2.3,-1.7)..controls (1.9,-1.2) and (.8,-.15)..(0,-.2);
\draw[thick] (0,.2)..controls (.8,.15) and (1.9,1.2)..(2.3,1.7);
\draw[thick] (0,-.2)..controls (.8,-.15) and (1.9,-1.2)..(2.3,-1.7);
\end{scope}

\begin{scope}
\draw[thick] (2.3,0) ellipse[x radius=1.4mm, y radius=4mm];
\draw[thick] (2.3,.4)..controls (2,.4) and (1.6,.3)..(1.6,0)..controls (1.6,-.3) and (2,-.4)..(2.3,-.4);
\draw[fill,white] (1.8,-.44)rectangle(1.9,.44);
\draw[thick] (2.3,0) ellipse[x radius=4.5mm,y radius=16mm];
\end{scope}

\begin{scope}
\clip (-2.3,-1.75)rectangle (-1.35,1.75);
\draw[thick] (0,.2)..controls (-.8,.15) and (-1.9,.9)..(-2.3,1.7);
\draw[thick] (0,-.2)..controls (-.8,-.15) and (-1.9,-.9)..(-2.3,-1.7);
\end{scope}

\begin{scope}
\draw[dashed] (-2.3,0) ellipse[x radius=1.4mm, y radius=4mm];
\draw[thick] (-2.3,.4)..controls (-2,.4) and (-1.7,.3)..(-1.7,0)..controls (-1.7,-.3) and (-2,-.4)..(-2.3,-.4);
\clip (-2.45,-.41) rectangle (-2.3,.41);
\draw[thick] (-2.3,0) ellipse[x radius=1.4mm, y radius=4mm];
\end{scope}

\begin{scope}
\draw[thick] (0,2.5) ellipse[x radius=7mm,y radius=2.5mm];
\draw[thick] (-.7,2.5)..controls (-.7,1.7) and (.7,1.7).. (.7,2.5);
\end{scope}

\begin{scope}
\draw[dashed] (0,-2.5) ellipse[x radius=7mm,y radius=2.5mm];
\draw[thick] (-.7,-2.5)..controls (-.7,-1.7) and (.7,-1.7).. (.7,-2.5);
\end{scope}
\clip (-.75,-5)rectangle (.75,-2.5);
\draw[thick] (0,-2.5) ellipse[x radius=7mm,y radius=2.5mm];

\end{scope}
}
{
\begin{scope}
\draw[dashed] (0,0) circle[radius=2cm];

\begin{scope}
\clip (-3.3,-2.61)rectangle(3.3,-.55);
\draw[thick] (0,0)..controls (1.6,0) and (2.9,-1)..(3.3,-2.6);
\draw[thick] (0,0)..controls (-1.6,0) and (-2.9,-1)..(-3.3,-2.6);
\draw[thick] (0,-.5) ellipse[x radius=1.9cm, y radius=5mm];
\draw[dashed] (0,-2.6) ellipse[x radius=33mm,y radius=5mm];
\end{scope}

\begin{scope}
\draw[thick] (0,0.2)..controls (1.7,0.2) and (2.9,.8)..(3.3,2.6);
\draw[thick] (0,0.2)..controls (-1.7,0.2) and (-2.9,.8)..(-3.3,2.6);
\end{scope}

\begin{scope}
\clip (-3.3,-2.6)rectangle (3.3,-3.1);
\draw[thick] (0,-2.6) ellipse[x radius=33mm,y radius=5mm];
\end{scope}

\begin{scope}
\draw[thick] (0,2.6) ellipse[x radius=33mm,y radius=5mm];
\end{scope}

\begin{scope}
\clip (-2.1,-.55)rectangle(2.,.58);
\draw[thick] (0,0) circle[radius=2cm];
\end{scope}

\begin{scope}
\draw[fill,white] (-1,2.1)--(-1,1)--(3.2,2.4)--(-1,2.4)--cycle;
\draw[thick] (-1,2.1)--(-1,1)--(3.2,2.46);
\clip (-1,2.1)--(-1,1)--(3.2,2.4)--(-1,2.4)--cycle;
\draw[thick] (0,0) circle[radius=2cm];
\end{scope}

\draw[thick] (-4.5,2)..controls (-2.8,1)and(-2.8,-1)..(-4.5,-2);

\draw[thick] (4.5,2)..controls (2.8,1)and(2.8,-1)..(4.5,-2);
\draw[thick] (4.53,0) ellipse[x radius=4mm,y radius=2cm];

\begin{scope}[xshift=-5mm]
\clip (-4.5,-2.5) rectangle (-4,-1.5);
\draw[thick] (-4.5,-2) ellipse[x radius=5mm,y radius=2mm];
\end{scope}

\begin{scope}[yshift=4cm,xshift=-5mm]
\clip (-4.5,-2.5) rectangle (-4,-1.5);
\draw[thick] (-4.5,-2) ellipse[x radius=5mm,y radius=2mm];
\end{scope}

\end{scope} 
}
\end{tikzpicture}
\caption{Level sets of standard quadratic forms and corresponding rigid tubes.}
\label{Figure1910}
\end{center}
\end{figure}
}
\begin{definition}
A  {\em{tube bundle}} is a fiber bundle $W \to M$ over a smooth manifold $M$ such that 
\begin{enumerate}
\item $W$ is a codimension zero submanifold of $M \times S^{N-1}$
\item the map $W \to M$ is the restriction to $W$ of the projection $M \times S^{N-1} \to M$.
\item for each $x \in M$ the fiber over $x$ is a  tube $T_x \subset S^{N-1}$.
\end{enumerate}
\end{definition}

If for each $x \in M$ the fiber of a tube bundle $W \to M$ is a rigid tube $T_x=T(q_x) \subset S^{N-1}$, where $q_x:\bR^N \to \bR$ is a family of nondegenerate quadratic forms, then we call $W$ a {\em{rigid tube bundle.}} If $g: M \times \bR^N \to \bR$ is a fibered function of quadratic tube type, then we will denote by $T(g) \to M$ the associated tube bundle $T(g)=\{ g \leq 0\} \cap S^{N-1}$. If $f:M \times \bR^N \to \bR$ is a fibered function of tube type, we will denote $T(f)=T(g)$ the tube bundle associated to the underlying fibered function of quadratic tube type $g$. 

\begin{warning}
If $f:\bR^N \to \bR$ is a function of tube type, then $T(f)$ is not $\{ f \leq 0\} \cap S^{N-1}$. It is $\{ g \leq 0\} \cap S^{N-1}$ for $g(x) = \|x\|^2 \lim_{ \lambda \to \infty} f(\lambda x)$. The same comment applies (fiberwise) for fibered functions of tube type $f:M \times \bR^N \to \bR$ and their associated tube bundle $T(f) \to M$.
\end{warning}

\begin{definition}
\label{def: OrientedTubeBundle}
We will say that a tube bundle $W \to M$ of index $k$ is {\em orientable} if the action of $\pi_1 M$ on the homology of the fiber $H_{k-1}(T;\bZ) \simeq \bZ$ is trivial. 
\end{definition}

If $W \to M$ is an orientable tube bundle, then the action of $\pi_1M$ on the homology $H_*(F)$ of the fiber is trivial and hence in particular unipotent, so we may associate to $W$ its higher Reidemeister torsion classes $\tau_k(W) \in H^{2k}(M;\bR)$. We emphasize that for orientable tube bundles, no unitary local system is needed. As explained in Appendix \ref{sec: higher torsion of bundles}, this causes the classes in degrees congruent to 2 modulo 4 to be automatically trivial. So $\tau_{2k+1}(W)=0$ in $H^{4k+2}(M;\bR)$ and we are left with the classes $\tau_{2k}(W) \in H^{4k}(M;\bR)$. 

Recall that a fibered function of tube type $f:M \times \bR^{N} \to \bR$ has an associated fibered function of quadratic tube type $g:M \times \bR^{N} \to \bR$ and a corresponding tube bundle $W=T(f)=T(g)$ pver $M$.
\begin{definition} Let $f:M \times \bR^N \to \bR$ be a fibered function of tube type, with an associated fibered function of quadratic tube type $g:M \times \bR^N \to \bR$ and a corresponding tube bundle $W=T(g)$. If $W$ is orientable, then $f$ is called an {\em orientable} fibered function of tube type.\end{definition}

To an  oriented fibered function of tube type we may associate to $f$ the higher torsion classes $\tau_k(W)$ of the oriented tube bundle $W=T(g)$. These classes are invariant under homotopy of $f$ through oriented fibered functions of tube type. For $q:M \times \bR^N \to \bR$ a fibered function of rigid tube type, being orientable amounts to asking that the stable bundle $E \to M$ of $q$ is orientable.

Let us record the following basic fact.

\begin{lemma} Let $W_1=T(f_1)$ and $W_2=T(f_2)$ be tube bundles associated to orientable fibered functions of tube type $f_1:M \times \bR^N \to \bR$ and $f_2:M \times \bR^N \to \bR$. If $f_1$ is homotopic to $f_2$ through fibered functions of tube type, then $W_1$ and $W_2$ have the same higher torsion classes. \end{lemma}

\begin{proof} Such a homotopy $f_t$ yields an isotopy of codimension zero submanifolds $W_t=T(f_t)$ of $M \times S^{N-1}$ and hence by isotopy extension a fibered diffeomorphism $W_1 \simeq W_2$ of tube bundles over $M$. Higher torsion classes are invariant under fibered diffeomorphism. \end{proof}

If $f:M \times \bR^{N_1} \to \bR$ is an orientable fibered function of tube type and $q:M \times \bR^{N_2} \to \bR$ is an orientable fibered function of rigid tube type, then $f \oplus q:M \times \bR^{N_1+N_2} \to \bR$ is an orientable fibered function of tube type. Indeed for $k_1$ the index of $f$ and $k_2$ the index of $q$ we have $\Sigma T(g \oplus q) \simeq \Sigma T(g) \land \Sigma T(q)$ as $S^{k_1+k_2}$-fibrations over $M$ and hence $$H^{k_1+k_2-1}(T(g \oplus q)_m ;\bR) \simeq H^{k_1-1}(T(g)_m;\bR) \otimes H^{k_2-1}(T(q)_m;\bR)$$
Hence if both the rank 1 real vector bundles $H^{k_1-1}(T(g)_m;\bR)$ and $H^{k_2-1}(T(q)_m;\bR)$ over $M$  are trivial, then so is $H^{k_1+k_2-1}(T(g \oplus q)_m ;\bR)$. At this point it is not yet clear how the higher torsion of $T(g)$ and $T(g \oplus q)$ should be related. For this purpose we introduce a computational tool. 

If an orientable fibered function of tube type $f:M \times \bR^{N} \to \bR$ has {\em moderate singularities} (i.e. Morse or Morse birth/death, see Appendix \ref{sec: moderate singularities}), then it is possible to compute the resulting higher torsion classes from the parametrized Morse theory of $f$ directly, and this will be useful to us in what follows. We use the same notation as in the statement of the {\em framing principle} given in Theorem \ref{thm: framing}. We recall briefly the relevant notions:
\begin{enumerate}
\item $\tau_k(f)$ is a cohomology class in $H^{2k}(M;\bR)$ obtained from the family $C_*(f_m)$, $m \in M$, of Morse chain complexes for the fiberwise restrictions $f_m: \bR^{N} \to \bR$, $m \in M$.
\item $\gamma_f$ is a stable real vector bundle over the fiberwise singular set $\Sigma_f \subset M \times \bR^{N}$ formed by the negative eigenspaces of the fiberwise Hessian of $f$. 
\item $\pi_*^{\Sigma f}$ is the push-down operator in cohomology $H^*(\Sigma_f) \to H^*(M)$ induced by $\pi|_{\Sigma_f} : \Sigma_f \to M$.
\item for $E \to X$ a real vector bundle over a manifold $X$, $p_k(E)  \in H^{4k}(M;\bR)$ is the normalized Pontryagin character:
$$p_k(E) = \frac{1}{2}(-1)^k  \zeta(2k+1)ch_{2k}(E\otimes \bC) .$$
\end{enumerate}

\begin{remark}
To simplify issues with signs, from now on we will restrict to the case of even index and ambient dimension a multiple of four. Standard stabilization is by $\bR^4$ with the quadratic form $\|x\|^2-\|y\|^2$, $(x,y) \in \bR^2 \times \bR^2$, which will preserve higher torsion, including its sign. 
\end{remark}

\begin{proposition}\label{prop: stab framing}
Let $f:M \times \bR^{4N} \to \bR$ be an orientable fibered function of tube type with moderate singularities, with associated fibered function of quadratic tube type $g:M \times \bR^{4N} \to \bR$ and tube bundle $W=T(g)$. The tube torsion $\tau(W) = \sum_k \tau_k(W)$ of $W$ may be computed from $f$ by the following formula: 
$$\tau(W) =    -\left(\tau_{2k}(f)  +   \pi_*^{\Sigma f} p_k(\gamma_f)\right).$$
\end{proposition}

\begin{proof}
The discussion of relative higher torsion in Appendix \ref{sec: rel torsion} applies to the trivial disk bundle $D^{4N} \to X \to M$, i.e. $X=M \times D^{4N}$, and the sub-bundle of the sphere bundle $S^{4N-1} \to \partial W \to M$ given by $W = T(g) \subset M \times S^{4N-1}$. Since $X$ is trivial we have $0 = \tau_k(X) = \tau_k(X,W) + \tau_k(W)$ so that $\tau_k(W) = - \tau_k(X,W)$. The proof is them completed by applying the framing principle Theorem \ref{thm: framing} to the bundle pair $(X,W)$, since we may compute the relative higher torsion $\tau_k(X,W)$ using the function with moderate singularities $f$. 
\end{proof}

\begin{proposition}\label{prop: effect stab}
Let $f:M \times \bR^{4N_1} \to \bR$ be an orientable  fibered  function of tube type and let $q:M \times \bR^{4N_2} \to \bR$ be a constant fibered function of rigid tube type, i.e. $q(m,v)=Q(v)$ for $Q:\bR^{4N_2} \to \bR$ a fixed non-degenerate quadratic form of even index. The higher  torsions of the tube bundles $T(f)$ and $T(g \oplus q)$ are equal:
$$ \tau_k(T(f \oplus q) ) = \tau_k(T(f))  $$
\end{proposition}

\begin{proof}
By the framed function Theorem \ref{thm: framed function} there exists a function $h:M \times \bR^{4N_1} \to \bR$ which is equal to $f$ outside of a compact set and which has moderate singularities. By Proposition \ref{prop: stab framing} we may compute the torsion of $T(f)=T(h)$ using $h$. Both $\tau_{2k}(h)$ and the stable bundle term $\gamma_h$ are unchanged after stabilization, see Appendix \ref{sec: higher torsion of bundles}. The conclusion follows.
\end{proof}

In particular, the classes $\tau_k(W)$ of $W=T(f)$ are invariant under stabilization of $f:M \times \bR^{4N} \to \bR$ by the standard quadratic form $\|x\|^2-\|y\|^2$ on $\bR^2\times \bR^2$, which has index 2. More generally, we have:

\begin{proposition}\label{prop: effect twisted stab}
Let $f:M \times \bR^{4N_1} \to \bR$ be an orientable  fibered  function of tube type of even index and let $q:M \times \bR^{4N_2} \to \bR$ be an orientable fibered function of rigid tube type of even index. The higher torsions of the tube bundles $T(f)$ and $T(f \oplus q)$are related by the formula
$$ \tau_k(T(f \oplus q) ) = \tau_k(T(f))  - p_k(E) $$
where $E \to M$ is the stable bundle of $q$.
\end{proposition}

\begin{proof}
By the framed function Theorem \ref{thm: framed function} there exists a function $h:M \times \bR^{4N_1} \to \bR$ which is equal to $f$ outside of a compact set and which has moderate singularities  Note that $\gamma_{h \oplus q} = \gamma_h \oplus (\pi|_{\Sigma h} )^*E$ as stable real vector bundles over the fiberwise critical set $\Sigma_h \subset M \times \bR^{N_1}$ and hence we have $$\pi_*^{\Sigma h}ch_{2k}(\gamma_{h \oplus q} \otimes \bC ) =  \pi_*^{\Sigma h}  ch_{2k}( \gamma_h  \otimes \bC) + \pi_*^{\Sigma h} \pi^* ch_{2k}(E \otimes \bC). $$
The proof is completed by applying Proposition \ref{prop: stab framing} to $h$ and $h \oplus q$ while noting that $ \pi_*^{\Sigma h} \pi^* $ is multiplication by the Euler characteristic of the fibrewise Morse complex which is $+1$ (since the index is even).
\end{proof}




\subsection{Generating functions of tube type}

\begin{definition}
A generating function $f:W \to \bR$ for a Legendrian $\Lambda \subset J^1M$ is a {\em{generating function of tube type}} if $W=M \times \bR^N$ and $f:M \times \bR^N \to \bR$ is a fibered function of tube type. 
\end{definition}

\begin{lemma}\label{lem: homotopy lifting}
Suppose a closed Legendrian $\Lambda \subset J^1M$ is generated by a fibered function of tube type $f:M \times \bR^N \to \bR$, and let $\Lambda_t \subset J^1M$ be a Legendrian isotopy. Then there exists a homotopy of fibered functions of tube type $f_t:M \times \bR^{N} \times \bR^{4k} \to \bR$ such that:
\begin{enumerate}
\item $f_t$ generates $\Lambda_t$
\item $f_0(m,x,u,v)=f(m,x)+Q(u,v)$ for the standard quadratic form $Q(u,v)=\|u\|^2-\|v\|^2$, where $(m,x) \in M \times \bR^N$ and $(u,v) \in \bR^{2k} \times \bR^{2k}$.
\item $f_t(m,x,u,v)=f(m,x)+\ve_t(m,x,u,v) + Q(u,v)$ for $\ve_t$ a compactly supported function. 
\end{enumerate}
\end{lemma}

\begin{proof} This follows immediately from the general form of the homotopy lifting property for generating functions as is formulated in Theorem 4.1.1. of \cite{EG98}. \end{proof} 

\begin{remark}\label{rem: monoid act} Note that if $\Lambda \subset J^1M$ is generated by a fibered function of tube type $f:M \times \bR^{4N_1} \to \bR$, then any twisted stabilization $f \oplus q : \bR^{4N_1+4N_2} \to \bR$ of $f$ by a fibered function of rigid tube type $q:M \times \bR^{4N_2} \to \bR$ will also generate $\Lambda$. Hence the monoid of fibered functions of rigid tube type over $M$ acts on the space of generating functions of tube type for the fixed Legendrian $\Lambda$.  \end{remark}

\subsection{The Legendrian invariant}

Let $\Lambda \subset J^1M$ be a closed Legendrian. Denote by $\cT(\Lambda)$ the collection of all tube bundles $W \to M$ of the form $W=T(f)$ for $f:M \times \bR^{4N} \to \bR$ an orientable fibered function of tube type of even index which is also a generating function for $\Lambda$.

\begin{definition}\label{def: tube torsion}
We define the Legendrian tube torsion of $\Lambda$, or {\em tube torsion} for short, to be the subset $\tau(\Lambda)\subset H^*(M;\bR)$ consisting of all {\em negatives} of higher torsion classes $-\tau_k(W) \in H^{2k}(M;\bR)$ for $W \to M$ an orientable tube bundle  $W\in  \cT(\Lambda)$. 
\end{definition}


\begin{remark}
The reason for defining $\tau_k(\Lambda)$ to consist of the negatives $-\tau_k(W)$ instead of $\tau_k(W)$ is  Proposition \ref{prop: stab framing}.
\end{remark}

\begin{proof}[Proof of Theorem \ref{thm: main 1}]
We prove one property at a time:
\begin{enumerate}
\item That the tube torsion $\tau(\Lambda)$ of $\Lambda$ is invariant under Legendrian isotopies of $\Lambda$ follows from Lemma \ref{lem: homotopy lifting} and Proposition \ref{prop: effect stab}. 
\item That $\tau(\Lambda)$ is always a (possibly empty) union of cosets of $p$ follows from Remark \ref{rem: monoid act} and Proposition \ref{prop: effect twisted stab}. 
\item That the tube torsion of the zero section $M \subset J^1 M$ is the image of $p$ follows from Proposition \ref{prop: stab framing}, since the parametrized Morse theory of a generating function $f:M \times \bR^{4N} \to \bR$ for the zero section is constant, so $\tau_{2k}(f)=0$, and hence the only term to contribute is $p_k(\gamma_f)$. Note that $\gamma_f$ may be any stable real vector bundle over $\Sigma_f \simeq M$ as any family of non-degenerate quadratic forms $Q_m:\bR^{4N} \to \bR$, $m \in M$, generates the zero-section. 
\end{enumerate}
\end{proof}

\begin{proposition}\label{prop: or aut} Suppose that a Legendrian $\Lambda \subset J^1M$ admits a generating function of tube type. Then it admits an orientable generating function of tube type. \end{proposition}

\begin{proof}
Suppose that a generating function of tube type $f:M \times \bR^{4N} \to \bR$ for $\Lambda$ is not orientable. Let $E  \to M$ denote the rank 1 vector bundle formed by the real degree $k-1$ cohomology of the fiber $F$ of the tube bundle $W \to M$, where $k$ is the index, i.e. $E_m=H^{k-1}(T(g)_m;\bR)$ for $g:M \times \bR^{4N} \to \bR$ the underlying fibered function of quadratic tube type.  Let $Q_m:\bR^{4d} \to \bR$, $m \in M$ be a family of non-degenerate quadratic forms of even index $2r$ whose stable bundle is (stably) isomorphic to $E$. Put $q(m,x)=Q_m(x)$, a fibered function of rigid tube type. Then $f \oplus q$ also generates $\Lambda$, and is orientable because for $m \in M$ we have isomorphisms $$H^{k+r-1}(T(g \oplus q)_m ;\bR) \simeq H^{k-1}(T(g)_m;\bR) \otimes H^{r-1}(T(q)_m;\bR) \simeq E_m \otimes E_m$$ and $E \otimes E$ is orientable. Hence $f \oplus q$ is an orientable generating function of tube type for $\Lambda$. \end{proof} 

For the next result note first that if the front of  Legendrian $\Lambda \subset J^1M$ only has cusp singularities, then any generating function of tube type $f: M \times \bR^{4N} \to \bR$ for $\Lambda$ automatically has moderate singularities. Recollections on functions with moderate singularities are given in Appendix \ref{sec: moderate singularities}. 

\begin{proposition}\label{prop: framing for special Leg}
Let $\Lambda \subset J^1M$ be a Legendrian whose front  $\Lambda \hookrightarrow J^1M \to J^0M$ only has cusp singularities and whose projection to the base $\Lambda \to M$ is a homotopy equivalence. If $\Lambda$ admits a generating function of tube type $f:M \times \bR^{4N_1} \to \bR$, then for a suitable fibered function of rigid tube type $q:M \times \bR^{4N_2} \to \bR$ the direct sum $f \oplus q : M \times \bR^{4N_1+4N_2} \to \bR$ is a frameable generating function of tube type , i.e. such that the stable bundle $\gamma_{f \oplus q}$ is trivial. \end{proposition}

\begin{proof} We have a canonical diffeomorphism $\Sigma_f \to \Lambda$ under which the projection $\pi:\Lambda \to M$ is the projection $\pi^{\Sigma f}:\Sigma_f \to M$. So $\pi^{\Sigma f}$ is a homotopy equivalence. Hence there exists a real stable bundle $E \to M$ such that $(\pi^{\Sigma f})^*E \simeq \gamma_f$. 

Now let $E' \to M$ be a stable even dimensional real vector bundle such that $E \oplus E'$ is trivial. Let $q:M \times \bR^{4N_2} \to \bR$ be a fibered quadratic generating function of rigid tube type whose stable bundle is isomorphic to $E'$. Then $\gamma_{f \oplus q} \simeq \gamma_f \oplus (\pi^{\Sigma f})^*E' \simeq (\pi^{\Sigma f})^*(E \oplus E')$ is trivial. \end{proof}

Using Proposition \ref{prop: framing for special Leg} we obtain in particular a result which relates $\tau(\Lambda)$ to the parametrized theory of a generating function of tube type with moderate singularties $f$ for $\Lambda$. Concretely, in Appendix \ref{sec: higher torsion of bundles} it is explained how to such an $f$ one obtains cohomology classes $\tau_{2k}(f) \in H^{4k}(M;\bR)$ out of the parametrized Morse complex $C_*(f_m;\bR)$, $m \in M$. 

\begin{proposition}\label{prop: torsion for special Leg}
Let $\Lambda \subset J^1M$ be a Legendrian whose front projection $\Lambda \hookrightarrow J^1M \to J^0M$ only has cusp singularities and whose projection to the base $\Lambda \to M$ is a homotopy equivalence. The higher torsion of $\Lambda$ is the subset $\tau_{2k}(\Lambda) \subset \text{coker}(p_k)$ consisting of the cosets of $\text{im}(p_k)$ with representatives given by the cohomology classes $$\tau_{2k}(f)\in H^{4k}(M;\bR),$$ where $f:M \times \bR^{4N} \to \bR$ ranges through the collection of orientable generating functions of tube type (of even index) for $\Lambda$. 
\end{proposition}

\begin{proof}
First we show that for any such $f \times \bR^{4N} \to \bR$ we have $\tau_{2k}(f) \in \tau_{2k}(\Lambda)$. Now, we certainly have that $W=T(f)$ is in $\cT(\Lambda)$ and hence $-\tau_{2k}(W) \in \tau_{2k}(\Lambda)$. By Proposition \ref{prop: stab framing} we have
$$ -\tau_{2k}(W) =\tau_{2k}(f) + (\pi^{\Sigma f})_*p_k(\gamma_f). $$
Since $\pi^{\Sigma f}: \Sigma_f \to M$ is a homotopy equivalence, we have $\gamma_f=(\pi^{\Sigma f})^*E$ for some real vector bundle $E \to M$ and hence
$$ (\pi^{\Sigma f})_*p_k(\gamma_f) = (\pi^{\Sigma f})_*(\pi^{\Sigma f})^*p_k(E)=p_k(E), $$
where the last equality follows from the fact that $(\pi^{\Sigma f})_*(\pi^{\Sigma f})^*$ is multiplication by the Euler characteristic of the fiber, i.e. $+1$. Hence $\tau_{2k}(f)$ is in the same coset of $\text{im}(p_k)$ as $-\tau_k(W)$ and therefore $\tau_{2k}(f) \in \tau_{2k}(\Lambda)$ as desired.

Conversely, let $W \in \cT(\Lambda)$. Given any generating function of tube type $f:M \times \bR^{4N_1} \to \bR$ for $\Lambda$ with $W=T(f)$, Proposition \ref{prop: effect twisted stab} implies that for any fibered function of rigid tube type $q:M \times \bR^{4N_2} \to \bR$ the torsion class $\tau_{2k}(W)$ is equal to the torsion class $\tau_{2k}(T(f \oplus q))$  up to an element of $\text{im}(p_k)$. If we take $q$ so that $f \oplus q$ is frameable as in Proposition \ref{prop: framing for special Leg} then $\tau_{2k}(T(f \oplus q))$ and $\tau_{2k}(f \oplus q)$ are equal (the extra term in Theorem \ref{thm: framing} from the stable bundle vanishes since $\gamma_{f \oplus q}$ is trivial). Finally $\tau_{2k}(f)=\tau_{2k}( f \oplus q)$ because the fiberwise Morse theory is unchanged by stabilization (and the parity of indices is unchanged). Hence $\tau_{2k}(W)$ and $\tau_{2k}(f)$ are in the same coset of $\text{im}(p_k)$ and we are done. 
\end{proof}





\subsection{The space of tubes} \label{sec: SpaceOfTubes}

We briefly compare the model for the space of tubes considered in this article with that of Waldhausen \cite{W82} in order to translate some relevant results from the literature. Recall that $\cT_{N,k}$ is the space of all functions $f:\bR^N \to \bR$ of tube type, see Definition \ref{def: tube}. The space $\cT_{N,k}$ deformation retracts onto the subspace $\cT_{N,k}^q$ of functions $g:\bR^N \to \bR$ of quadratic tube type, see Definition \ref{def: quad tube}. 

Any tube $T \subset S^{N-1}$ has two non-negative integers associated to it, namely the ambient dimension $N$, so that $T \subset S^{N-1}$ is a codimension zero submanifold, and its index $k$, which is the integer $k\geq 0$ such that $T \subset S^{N-1}$ is isotopic to $T(q) \subset S^{N-1}$ for $q:\bR^N \to \bR$ a non-degenerate quadratic form of index $k$. The {\em standard tube} of type $(N,k)$ is $T_{N,k}=T(q_{k,N})=\{q_{k,N} \leq 0\} \cap S^{N-1}$ for $q_{k,N}:\bR^N \to \bR$ the standard index $k$ quadratic form $\|x\|^2-\|y\|^2$, $(x,y) \in \bR^{N-k} \times \bR^k$. By definition a {\em tube} of type $(N,k)$ is a codimension zero submanifold $T \subset S^{N-1}$ which is isotopic to the standard tube of type $(N,k)$. One may consider the space $\widetilde{\cT}^q_{N,k}$ of all tubes of type $(N,k)$, for example using a simplicial space model, or otherwise. 

Any reasonable model for $\widetilde{\cT}^q_{N,k}$ would come equipped with a map $\cT_{N,k}^q \to \widetilde{\cT}^q_{N,k}$ given by $g \mapsto T(g)$ whose fiber may be identified with the space of all functions $h:S^{N-1}\to \bR$ such that $T=\{ h \leq 0 \}$, with  $\partial T = \{ h=0 \}$ a regular level set of $h$. For fixed $T$, the space of such $h$ is convex, hence contractible. Such a function $h$ determines a unique function of quadratic tube type $g(v)=\|v\|^2h(v/\|v\|)$. For any reasonable model the map $\cT_{N,k}^q \to \widetilde{\cT}^q_{N,k}$ is a Serre fibration and hence  one has a weak equivalence $\cT_{N,k}^q \simeq \widetilde{\cT}^q_{N,k}$ and hence also a weak equivalence $\cT_{N,k} \simeq \widetilde{\cT}^q_{N,k}$.  

Waldhausen considered in \cite{W82} the following similar though slightly different notion of tubes in terms of partitions. For $E$ a $k$-dimensional linear subspace of $\bR^{N-2}$, consider the codimension zero submanifold with boundary $T_E \subset \bR^{N-1}$ obtained by attaching to the half-space $\{x_{N-1} \leq 0 \}$ inside $\bR^{N-1}$ a standard $(N-1)$-dimensional index $k$ handle along the unit sphere of $E \subset \{x_{N-1}=0\} = \bR^{N-2}$. More concretely, $T_E$ is the union in $\bR^{N-1}$ of the half-space $\{x_{N-1} \leq 0\}$ and an $\varepsilon$-neighborhood of the $k$-dimensional unit sphere of the $(k+1)$-dimensional subspace of $\bR^{N-1}$ spanned by $E \oplus 0 \subset \bR^{N-1}$ and $\partial / \partial x_{N-1}$, suitably smoothed so that $T_E$ is a smooth codimension zero submanifold with boundary. Waldhausen called such a codimension zero submanifold $T_E \subset \bR^{N-1}$ a {\it rigid tube}. Since the thickening and smoothing are homotopically canonical, for fixed $k$ and $N$ the space of rigid tubes $T_E$ in $\bR^{N-1}$ as above has the homotopy type of $BO(k,N-1)$, the Grassmannian of real $k$-dimensional linear subspaces of $\bR^{N-1}$. 

Waldhausen defined a {\em tube} to be any codimension zero submanifold $T \subset \bR^{N-1}$ which is the image of a rigid tube under a smooth isotopy of $\bR^{N-1}$, fixed outside of a compact set. We will denote the space of tubes in the sense of Waldhausen by $\widetilde{\cT}_{N,k}^p$, where $p$ stands for {\em partition}. We may also form a suitable space of functions of tube type  $f:\bR^{N-1} \to \bR$ in the sense of Waldhausen, (so that in particular  $\{f \leq 0\} \subset \bR^{N-1}$ is a tube in the sense of Waldhausen) and denote it by $\cT_{N,k}^p$. Any reasonable model would again come equipped with an equivalence $\cT_{N,k}^p \to \widetilde{\cT}_{N,k}^p$ given by $f \mapsto \{f \leq 0\}$. See \cite{ACGK} for a precise definition of a space of functions of tube type very close to the space $\cT_{N,k}^p$ described above.

The difference between the model  $\widetilde{\cT}^q_{N,k}$ for the space of tubes considered in this article and the Waldhausen model  $\widetilde{\cT}_{N,k}^p$ in terms of partitions, at the $(N,k)$ level, amounts to requiring the tubes in  $\widetilde{\cT}^q_{N,k}$  to be standard in a fixed $(N-1)$-dimensional closed ball $B \subset S^{N-1}$, say $B=\{x_{N} \geq 0 \} \cap S^{N-1}$. More concretely, with respect to the fixed diffeomorphism $\varphi:D^{N-1} \to B$ given by $$\varphi(x_1,\ldots,x_{N-1})=\left(x_1,\ldots,x_{N-1},\sqrt{1-\sum_{i=1}^{N-1}x_i^2}\right),$$ the precise condition is that a tube $T \subset S^{N-1}$ must intersect $B$ in the half-space $\{x_{N-1} \leq 0\} \cap D^{N-1} \subset D^{N-1}$ with respect to the identification $\varphi$. One may then forget about $T \cap B$ and look at the intersection of $T$ with the other hemisphere $\{x_{N} \leq 0\} \cap S^{N-1}$, which becomes a partition tube in the sense of Waldhausen upon identifying the open hemisphere $\{x_N <0\} \cap S^{N-1}$ with $\bR^{N-1}$ by a fixed radial diffeomorphism $D^{N-1} \xrightarrow{\simeq} \bR^{N-1}$, so that the image of $T \cap \{x_N<0\}$ is standard at infinity as in \cite{W82}.

If we define $\widetilde{\cT}_{N,k}^p\subset \widetilde{\cT}^q_{N,k}$ to be the subspace of $\widetilde{\cT}^q_{N,k}$ consisting of tubes which are standard on $B$ as above we get an map given by inclusion $\widetilde{\cT}_{N,k}^p \to \widetilde{\cT}^q_{N,k}$, see Figure \ref{Figure1908}. 

{
\begin{figure}[htbp]
\begin{center}
\begin{tikzpicture}[scale=.8]
%
\draw (.9,2.4) node{$\longrightarrow$};
\begin{scope}[xshift=-5cm,yshift=-5mm]
\draw[fill, gray!30!white] (0,2.3)--(0,3)--(.5,3)..controls (1.5,3) and (.5,4.5)..(2,4.5)..controls (3.5,4.5) and (2.5,3)..(3.5,3)--(4,3)--(4,2.3)--(0,2.3);
\draw[thick] (0,2.3)--(0,3)--(.5,3)..controls (1.5,3) and (.5,4.5)..(2,4.5)..controls (3.5,4.5) and (2.5,3)..(3.5,3)--(4,3)--(4,2.3)--(0,2.3);
\draw[white,fill] (2,3.5)circle[radius=5mm];
\draw[thick] (2,3.5)circle[radius=5mm];
\end{scope}
\begin{scope}[xshift=-5cm]
\draw (0,0)--(0,5)--(4,5)--(4,0)--cycle;
\end{scope}
%
%
\begin{scope}[xshift=3cm,yshift=-1cm]
\clip (2,3.2) circle[radius=2.5cm];
\draw (-.5,3.5)..controls (-.3,4) and (4.3,4)..(4.5,3.5);

\draw[fill, gray!30!white] (-.4,0)--(-.5,3.5)..controls (-.3,3) and (.5,3)..(.5,3)..controls (1.5,3) and (.5,4.5)..(2,4.5)..controls (3.5,4.5) and (2.5,3)..(3.5,3)..controls (3.5,3) and (4.3,3)..(4.5,3.5)--(4.5,0)--(-.4,0);
\draw[thick] (-.4,0)--(-.5,3.5)..controls (-.3,3) and (.5,3)..(.5,3)..controls (1.5,3) and (.5,4.5)..(2,4.5)..controls (3.5,4.5) and (2.5,3)..(3.5,3)..controls (3.5,3) and (4.3,3)..(4.5,3.5)--(4.5,0)--(-.4,0);

\draw[white,fill] (2,3.5)circle[radius=5mm];
\draw[thick] (2,3.5)circle[radius=5mm];
\draw (2,3.2) circle[radius=2.5cm];

\end{scope}
%
\end{tikzpicture}
\caption{Tube in square and tube in a sphere, yielding the map $\widetilde{\cT}_{N,k}^p \to \widetilde{\cT}^q_{N,k}$.}
\label{Figure1908}
\end{center}
\end{figure}
}

The map $\widetilde{\cT}_{N,k}^p \to \widetilde{\cT}^q_{N,k}$ is not a homotopy equivalence, however by a standard argument the connectivity tends to infinity as you stabilize $N,k \to \infty$ (in the model $\cT_{N,k}^q$ the stabilization is given by direct sum with standard quadratic forms $\|x\|^2-\|y\|^2$, $(x,y) \in \bR^a \times \bR^b$). Indeed, given a tube bundle $W \subset M \times S^{N-1} \to M$ over a compact manifold $M$, so for each $m \in M$ we have a tube $T_m \subset M$. First, since the fiber of the fiber bundle $\partial W \to M$ is diffeomorphic to $S^{k-1} \times S^{N-k-1}$, by increasing $N$ and $k$ we may always find a section $s:M \to \partial W$ of the fibration $\partial W \to M$. Second, as a section of the trivial bundle $M \times S^{N-1}  \to M$ the section $s$ is homotopic through sections to a constant section after increasing $N$ if necessary. By isotopy extension this means we may assume that there is a point, say $x=(0,\ldots,0,1) \in S^{N-1}$ such that $x \in \partial T_m$ for each tube $T_m \subset S^{N-1}$, $m \in M$. Next, note that we may assume that $T_x (\partial T_m) = \bR^{N-2} \times 0 \times 0 \subset \bR^{N-1} \times 0 =  T_xS^{N-1}$ for all $m \in M$. Indeed, again by isotopy extension the only obstruction is linear algebraic and consists of trivializing the outwards pointing normal vector to $\partial T_m$ at $x$, which lives in $[M,S^{N-2}]$ and hence dies if $N$ becomes large enough. At this point it is straightforward to make $W$ standard near an $\varepsilon$-ball about $x \in B \subset S^{N-1}$ and then by a suitably isotopy of $S^{N-1}$ it is then straightforward to make it standard on the entire ball $B \subset S^{N-1}$. If $M$ has boundary this construction all works relative to $\partial M$ and so this proves that the inclusion $\widetilde{\cT}_{N,k}^p\subset \widetilde{\cT}^q_{N,k}$ becomes arbitrarily connected as $N,k \to \infty$.

Therefore we may speak unambiguously of the stable space of tubes $\cT_\infty$, which is the result of stabilizing $N,k\to \infty$, as this stable space is common to both of these models. The monoid $\cT^q=\coprod_{N,k} \cT_{N,k}^q$ consisting of functions of quadratic tube type which we consider in the present article was introduced in \cite{AACK}, and the function space version of $\cT_\infty$ is realized by implementing stabilization of $\cT^q$ as direct sum with suitable standard non-degenerate quadratic forms. The idea of thinking of tubes as codimension zero submanifolds of the sphere had already been introduced in the model of Rognes in \cite{R92}. The advantage of this model is that it is better suited for the monoid (and indeed infinite loopspace) structures at play than the partition model used by Waldhausen. We also have a stable space $\cQ_\infty$ given by stabilizing $N,k \to \infty$ in $\cQ_{N,k}$ in the same way, and the inclusion $\cQ_{k,N} \to \cT_{k,N}$ stabilizes to a map $\cQ_\infty \to \cT_\infty$.


In particular, given a fibered function of tube type over $M$ we get an associated map $M \to \cT_\infty$. If the fibered function is of rigid tube type then we get a lift $M \to \cQ_\infty$ through $Q_\infty \to \cT_\infty$. In this case, we may compose this with the equivalence $\cQ_\infty \to BO$ given by taking the negative eigenspace of a non-degenerate quadratic form, and the resulting map $M \to BO$ classifies the stable bundle of the family of quadratic forms $q_m: \bR^N \to \bR$ comprising the fibered function of rigid tube type. 
Similarly, in the general case of a fibered function of tube type which is not necessarily rigid, one may consider the associated map $M \to \cT_\infty$ and it is shown in \cite{W82} that the homotopy class of this map classifies the fibered function of tube type over $M$ up to stable equivalence.

The map $\cQ_\infty \to \cT_\infty$ called the (stable) {\em rigid tube map}, or more informally we may denote $BO \to \cT_\infty$. It was proved by B\"okstedt that the stable rigid tube map is a rational homotopy equivalence \cite{B82}, a fact which was also proved by the second author using higher Reidemeister torsion in \cite{I02}. For example, for homotopy groups this means that $$\pi_n \cT_\infty \otimes \bQ   \simeq \pi_n BO \otimes \bQ $$ 
which is $\bQ$ if $n$ is of the form $4k$ and is zero otherwise. For $n=4k$, modulo the torsion subgroup, we have a monomorphism of free abelian groups
$$\pi_{4k} BO \xrightarrow{\sim}  (\pi_{4k} \cT_\infty )/\ \text{Tors} $$ 
i.e. up to $\pm 1$ a monomorphism $\bZ \to \bZ$. Let $J_k \in \bZ_{>0}$ be the index of the image. Higher torsion gives an injection $\bZ/J_k\bZ \to \bR/\lambda_k\bZ$ where $\lambda_k = p_k(E)$ for $ E \to S^{4k}$ the generator ok $KO^0(S^{4k})$. So the higher torsion classes are a subtle enough invariant to detect the failure for the rigid tube map $BO \to \cT_\infty$ to be surjective on homotopy groups modulo the  torsion subgroup $\text{Tors} (  \pi_{4k} \cT_\infty )$.

Following \cite{W82}  consider  $\cH_\infty$, the  {\em stable space of h-cobordisms of a point}, denoted $\cH_\infty$, which is a classifying space for fiber bundles $V \to M$ whose fiber is an h-cobordism of a disk $D^k$, i.e. diffeomorphic to $D^k \times [0,1]$ with a fixed trivialization of the lower boundary $D^k \times \{0 \}$, up to stable equivalence. Waldhausen defined a map $\cH_\infty \to \cT_\infty$ given by stacking an h-cobordism of a disk on top of a standard tube, see Figure \ref{Figure1903}.

{
\begin{figure}[htbp]
\begin{center}
\begin{tikzpicture}

\begin{scope}[xshift=1cm,yshift=.25cm]
\draw[thick] (-5,1.75)--(-5,-2.25)--(-2,-2.25)--(-2,1.75)--cycle;
\draw[thick,fill,gray!50!white] (-5,-2.25)--(-5,-.25)..controls (-3.7,-.25) and (-3.7,0.2)..(-4,0.2)..controls (-4.2,.25) and (-4.2,.35)..(-4,.4)..controls (-3.6,.45) and (-3.6,.55)..(-4.2,.6)..controls(-4.3,.7) and (-3.65,.75)..(-3.75,.8)..controls (-3.8,1.1) and(-4,1.1)..(-3.4,1.2)..controls (-3,1.3) and (-2.9,.6).. (-3.2,.3)..controls (-3.5,0) and (-3.1,0)..(-3.4,-.2)..controls (-3.8,-.4) and (-3.8,-.6)..(-3,-.6)..controls (-2.5,-.5) and (-2.9,-.2)..(-2.6,0)..controls (-2.3,0) and (-2.3,-.25)..(-2,-.25)--(-2,-2.25)--cycle;
\end{scope}

\begin{scope}
\clip (1,4) circle[radius=1.6cm];
\draw[fill,gray!50!white] (1,3.6) ellipse[x radius=1.55cm,y radius=.5cm];
\draw[fill,gray!50!white] (1,3.8) ellipse[x radius=1.7cm,y radius=.5cm];
\draw[fill,white] (1,4) ellipse[x radius=1.6cm,y radius=.5cm];
\draw(1,4) ellipse[x radius=1.6cm,y radius=.5cm];
\draw (1,3.6) ellipse[x radius=1.55cm,y radius=.5cm];
\draw[thick] (1,4) circle[radius=1.6cm];
\end{scope}

\begin{scope}[rotate=8, scale=.8,yshift=-1.5cm] 
\begin{scope}[xshift=2.3cm,yshift=-5.2cm,rotate=-8,scale=1.6]
\clip (1,4) circle[radius=1.6cm];
\draw[fill,gray!50!white] (1,3.6) ellipse[x radius=1.55cm,y radius=.5cm];
\draw[fill,gray!50!white] (1,3.8) ellipse[x radius=1.7cm,y radius=.5cm];
\draw[fill,white] (1,4) ellipse[x radius=1.6cm,y radius=.5cm];
\draw(1,4) ellipse[x radius=1.6cm,y radius=.5cm];
\draw (1,3.6) ellipse[x radius=1.55cm,y radius=.5cm];
\draw[thick] (1,4) circle[radius=1.6cm];
\end{scope}
\begin{scope}[xshift=9cm,yshift=.5cm]
\clip (-3.2,-.5) circle[radius=.6cm];
\draw[fill,white] (-5,-.25)..controls (-3.7,-.25) and (-3.7,0.2)..(-4,0.2)..controls (-4.2,.25) and (-4.2,.35)..(-4,.4)..controls (-3.6,.45) and (-3.6,.55)..(-4.2,.6)..controls(-4.3,.7) and (-3.65,.75)..(-3.75,.8)..controls (-3.8,1.1) and(-4,1.1)..(-3.4,1.2)..controls (-3,1.3) and (-2.9,.6).. (-3.2,.3)..controls (-3.5,0) and (-3.1,0)..(-3.4,-.2)..controls (-3.8,-.4) and (-3.8,-.6)..(-3,-.6)..controls (-2.5,-.5) and (-2.9,-.2)..(-2.6,0)..controls (-2.3,0) and (-2.3,-.25)..(-2,-.25)..controls (-3,-.4) and (-4,-.4)..(-5,-.25);
\end{scope}
\begin{scope}[xshift=9cm,yshift=.5cm]
\draw[thick,fill,gray!50!white] (-5,-.25)..controls (-3.7,-.25) and (-3.7,0.2)..(-4,0.2)..controls (-4.2,.25) and (-4.2,.35)..(-4,.4)..controls (-3.6,.45) and (-3.6,.55)..(-4.2,.6)..controls(-4.3,.7) and (-3.65,.75)..(-3.75,.8)..controls (-3.8,1.1) and(-4,1.1)..(-3.4,1.2)..controls (-3,1.3) and (-2.9,.6).. (-3.2,.3)..controls (-3.5,0) and (-3.1,0)..(-3.4,-.2)..controls (-3.8,-.4) and (-3.8,-.6)..(-3,-.6)..controls (-2.5,-.5) and (-2.9,-.2)..(-2.6,0)..controls (-2.3,0) and (-2.3,-.25)..(-2,-.25)..controls (-3,-.85) and (-4,-.85)..(-5,-.25);
\end{scope} 
\end{scope} 

\begin{scope}[xshift=3cm,yshift=4cm]
\draw (0,.2) node[right]{standard rigid tube};
\draw (0,-.2) node[right]{$T(q)$ of index $k$.};
\end{scope}

\end{tikzpicture}
\caption{On the left is an $h$-cobordism of a disk representing an element of $\cH_\infty$ as a partition of $[-1,1]^{n-1}\times[0,1]$. On the right is the tube of index $k$ obtained by pasting the $h$-cobordism into $S^{n}$ on top of the standard rigid tube of index $k$ shown in the middle. This yields a map $\cH_\infty \to \cT_\infty$.}
\label{Figure1903}
\end{center}
\end{figure}
}

Consider also the forgetful map $\cT_\infty \to \cT_\infty^{\text{TOP}}$ in which $ \cT_\infty^{\text{TOP}}$ is the stable space of tubes in the topological category. It was shown by Waldhausen in \cite{W82} that $\cT_\infty^{\text{TOP}} \simeq BG$ for $BG$ the classifying space for spherical fibrations, i.e. $G=\lim_N G_N$ for $G_N$ the space of maps $S^N \to S^N$ of degree $\pm 1$.

\begin{theorem}[\cite{W82}]\label{thm: wald fib} $\cH_\infty \to \cT_\infty \to BG$ is a fibration sequence. \end{theorem}

Consider the map $BG \to B (\bZ/2)$ induced by the degree map $G \to \bZ/2$. To a tube bundle $W \to M$, with stable equivalence class classified by a map $M \to \cT_\infty$, is associated a map $M \to B (\bZ/2)$. This map is nullhomotopic precisely when $W$ is orientable. Higher torsion is not defined for non-orientable tube bundles, i.e. higher torsion is a characteristic class coming from the cohomology of the fiber of the map $\cT_\infty \to B(\bZ/2)$. Note that the composition $\cH_\infty \to \cT_\infty \to BG \to B(\bZ/2)$ is trivial and so tube bundles which come from h-cobordism bundles are automatically orientable. 

We also note that the higher torsion of a tube bundle changes sign under stabilization by the standard form $x^2-y^2$ on $\bR^2$ since the index $k$ changes parity. For this reason it is better to assume that all ambient dimensions are even and always stabilize by a non-degenerate form of even index in an even dimensional ambient space, which always leaves torsion unchanged. With this modification of $\cT_\infty$ (and similarly for $\cH_\infty$ and $\cQ_\infty$), any homotopy class of map $M \to \cT_\infty$ such that the composite $M \to B(\bZ/2)$ is null determines higher torsion classes in $M$. In particular any homotopy class of map $M \to \cH_\infty$ determines higher torsion classes  in $M$. The following observation, which will be our source of examples of Legendrians with non-trivial tube torsion in Section \ref{sec: nontriv} below. 

\begin{proposition} If $V \to M$ is a bundle of h-cobordisms of a disk classified by $M \to \cH_\infty$ and $W \to M$ is the stable tube bundle classified by the composition $M \to \cH_\infty \to  \cT_\infty$, then there holds $\tau_k(V)=\tau_k(W)$.  \end{proposition}

\begin{proof} The tube bundle $W$ is obtained by gluing $V$ to the trivial tube bundle $M \times T$ for $T$ the standard tube. By Theorem \ref{thm: framed function} we may choose a function $f:V \to \bR$ with moderate singularities which is standard near the top and bottom of each h-cobordism, i.e. such that on each fiber $D^m \times [0,1]$ the restriction of $f$ is standard near top and bottom. We may glue such a function $f$ to the standard quadratic form $q:M \times T \to \bR$ on the trivial tube bundle $M \times T$  to obtain a function with moderate singularities $h:W \to \bR$ whose paramerized Morse theory consists of the union of that of $f$ and (the constant family) $q$. The proof is then completed by applying the framing principle Theorem \ref{thm: framing} to compute $\tau_k(W)$ using the function $h:W \to \bR$ on the glued up tube bundle and to compute $\tau_k(V)$ using the function $f:V \to \bR$ on the original bundle of h-cobordisms. The two computations yields the same expression because the extra contribution from $q$ to both the parametrized Morse complex and to the stable bundle is trivial.   \end{proof}

\begin{remark} Given a Legendrian $\Lambda \subset J^1M$, consider the subset $\cT(\Lambda) \subset [M,\cT_\infty]$ consisting of homotopy classes of maps classifying a stable tube bundle on which $\Lambda$ admits a generating function. The homotopy lifting theorem for generating functions implies that $\cT(\Lambda)$ is a Legendrian isotopy invariant. The tube torsion classes that we consider in the present article are a cohomological shadow of the a priori finer invariant $\cT(\Lambda)$.  \end{remark}

\begin{remark}
 By the stable parametrized stable h-cobordism theorem of Waldhausen $\cH_\infty$ is closely related to $K(\mathbf{S})$, the algebraic K-theory of the sphere spectrum, i.e. $A(\ast)$, the Waldhausen K-theory of a point. Hence by Theorem \ref{thm: wald fib} the stable space of tubes $\cT_\infty$ is also closely related to $A(\ast)$. As noted in the remark after Proposition 3.1 of \cite{W82}, there is a homotopy cartesian square
$$
\begin{tikzcd} \cT_\infty \ar[r] \ar[d] & BG\ar[d] \\
 \Omega^\infty S^\infty \ar[r]& A( \ast)
\end{tikzcd} $$
\end{remark}

\begin{remark}
There is another important positive integer $I_k \in \bZ_{>0}$, which is closely related to $J_k$. Waldhausen in \cite{W82} more generally discussed a homotopy commutative square where the columns are fibration sequences
$$
\begin{tikzcd}  G/ O\ar[r] \ar[d] & \cH_\infty \ar[d] \\
BO \ar[r] \ar[d] & \cT_\infty  \ar[d]\\
BG \ar[r,"="] & BG 
\end{tikzcd} $$
The vertical column is the $J$-map fibration $G/O \to BO \to BG$ and the top row is the {\em Hatcher-Waldhausen map} $G/O \to \cH_\infty$, which is a rational equivalence by that same theorem of B\"okstedt  \cite{B82}. For $n=4k$, modulo the torsion subgroups, we have a monomorphism of free abelian groups
$$(\pi_{4k} G/O) / \text{Tors} \xrightarrow{\sim} (\pi_{4k} \cH_\infty) / \text{Tors} $$ 
i.e.  up to $\pm 1$ a monomorphism $\bZ \to \bZ$. Let $I_k \in \bZ_{>0}$ be the index of the image. The integers $I_k$ and $J_k$ are thus closely related by the J-homomorphism. From private communication from S. Kupers we understand that in principle a careful comparison of the literature, including a careful comparison of conventions for regulators and for higher torsion classes, could lead to concrete information on the integers $I_k$ and $J_k$. 

The integer $I_k$ is also closely related to the mysterious finite homotopy group $\pi_{4k}\cM_\infty$ of the space $\cM_\infty$ considered by Kragh in \cite{Kra18}, who proved that the Hatcher map extends to a fibration sequence $\cM_\infty \to G/O \to \cH_\infty$ where $\cM_\infty$ is roughly speaking the stable space of functions $\bR^{2N} \to \bR$ equal to the standard quadratic form $\|x\|^2-\|y\|^2$ at infinity and which have a unique and non-degenerate critical point. The space $\cM_\infty$ is related to generating functions for {\em Long nearby Lagrangian knots} in $T^*\bR^k$.  Indeed, in \cite{Kra18} it is shown that if $L \subset T^*\bR^k$ is an exact Lagrangian submanifold agreeing with the zero section $\bR^k \subset T^*\bR^k$ outside of a compact set, then $L$ admits a generating function $f:\bR^k \times \bR^{2N} \to \bR$ which is {\em standard-quadratic} at infinity and hence by restricting to a large sphere in $\bR^k$ one obtains an element of $\pi_{k-1} \cM_\infty$ which is closely related to the torsion invariants considered in the present article. The Hatcher-Waldhausen map $G/O \to \cH_\infty$ was shown to be an infinite loop map by Rognes in \cite{R92}.
\end{remark}

\subsection{Legendrians with nontrivial tube torsion}\label{sec: nontriv}

We are now ready to establish for general $M$ the existence of Legendrians $\Lambda \subset J^1M$ such that the projection $\Lambda \to M$ is homotopic to a diffeomorphism and which have nontrivial tube torsion. In particular this shows that the vanishing of the tube torsion of nearby Lagrangians $L \subset T^*M$ predicted by the nearby Lagrangian conjecture cannot rely solely on the fact that the projection $L \to M$ is a homotopy equivalence (or even from a conjectural expectation that the projection $L \to M$ is homotopic to a diffeomorphism, which is not known).

Recall that given a bundle of h-cobordisms of a disk $V \to M$, the stable equivalence class of the bundle is classified by the homotopy class of a map $M \to \cH_\infty$, see Section \ref{sec: SpaceOfTubes}. By the fibration sequence $\cH_\infty \to \cT_\infty \to BG$ of Waldhausen \cite{W82} there is an associated tube bundle $W \to M$ which geometrically is obtained by attaching $V$ onto a trivial bundle of tubes $M \times T \to M$. 

We recall that Theorem \ref{thm: existence of example} states that given a bundle of h-cobordisms of a disk $V  \to M$, there exists a Legendrian $\Lambda \subset J^1M$ whose projection $\Lambda \to M$ is homotopic to a diffeomorphism and which admits a generating function on the tube bundle $W \to M$ obtained from $V$ via the map $\cH_\infty \to \cT_\infty$. 


\begin{proof}[Proof of Theorem \ref{thm: existence of example}]
We may apply Theorem \ref{thm: wrinkling} to the bundle of $h$-cobordisms $V$, where we take $A$ to consist of $\partial V$ and $g: Op(A) \to \bR$ to be a function on $Op(A)$ without critical points for which the two components of $\partial V$ (corresponding to the boundary components $D^k \times \{0\}$ and $D^k \times \{1\}$  of the fiber $D^k \times [0,1]$) are regular level sets of $g$. To verify the hypothesis of Theorem \ref{thm: wrinkling} that $d_vg$ extends to a non-vanishing section of $\text{Vert}^*$ we may argue as in Proposition 2.5 in \cite{EM00}, though note the attribution there to earlier work of Douady and Laudenbach, as well as to the second author's PhD thesis. 

After applying Theorem \ref{thm: wrinkling} we conclude that there exists a fibered wrinkled function $h:V \to \bR$ such that the lower boundary $\partial_0 V$ is a regular level set for $h$. Hence $h$ can be stitched up with a standard function on the trivial tube bundle $q: M \times T \to \bR$ to obtain a function with moderate singularities $h \# q :W \to \bR$ whose fiberwise critical set $\Sigma_f$ is the disjoint union of $\Sigma_h$ and $\Sigma_q$, where $\Sigma_q \to M$ is a diffeomorphism. Hence $h \# q$ generates a Legendrian $\Lambda \subset J^1M$ which is the disjoint union of finitely many Legendrian spheres $S_i \subset J^1M$, each diffeomorphic to the standard sphere and whose projection $S_i \to M$ is nullhomotopic, and the zero section $M \subset J^1M$, see Figure \ref{Figure1906a}. Note that each $S_i$ may sit as a complicated Legendrian knot in $J^1M$, and the union of the $S_i$ as a complicated Legendrian link, but each front $S_i \to J^1M$ has only cusp singularities (along an equator $C_i \subset S_i$). 

Next, we note using Proposition 3.21 of \cite{ACGK} one may construct a fibered function of tube type $f:M \times \bR^N \to \bR$ such that $W=T(f)$ and which generates the same Legendrian $\Lambda$ (the statement of Proposition 3.21 only states it up to Legendrian isotopy but one may appeal to the homotopy lifting property after further stabilization if necessary). After a further Legendrian isotopy of $\Lambda$, (which againis covered by a homotopy of functions $f$ having the same properties after stabilization if necessary), we may arrange it so the Legendrian $\Lambda$ has cusps of any Maslov index (spin a Reidemeister I move and use this to make pagodas of increasing and decreasing Maslov index at any given point of the graphical component of $\Lambda$, see Figure \ref{Figure1906b}). We can then use a surgery of corank 1 singularities to surger the spheres $S_i$ into the Legendrian $\Lambda$ via connect sum, again covering the deformation by a homotopy of functions $f$ having the same properties, except now the components $S_i$ have been absorbed into the graphical component of the Legendrian, see Figure \ref{Figure1907}. Thus the diffeomorphism type of $\Lambda$ is equal to the connect sum of $M$ with finitely many standard spheres, so the final Legendrian $\Lambda$ is diffeomorphic to $M$ and moreover the projection $M \to \Lambda$ is homotopic to a diffeomorphism.   \end{proof}

{
\begin{figure}[htbp]
\begin{center}
\begin{tikzpicture}[scale=.7]

{
\begin{scope}[yshift=5cm]
\begin{scope}
\draw[thick] (-5,0)--(5,0);
\draw (5,0) node[right]{$M$};
\end{scope}
\begin{scope}[yshift=2cm,xshift=-2.3cm] 
\draw[thick] (0,0)..controls (.6,0) and (1,.8)..(1.5,.8)..controls (2,.7) and (2.5,0.1)..(3,.1);
\draw[thick](0,0)..controls (.5,0)and(.9,-.55)..(1.6,-.6)..controls (2.2,-.6) and (2.6,.1)..(3,.1);
\end{scope}
\begin{scope}[yshift=1.8cm,xshift=-1.3cm] 
\draw[thick] (0,0)..controls (.6,0) and (1,.8)..(1.5,.8)..controls (2,.7) and (2.5,0.1)..(3,.1);
\draw[thick](0,0)..controls (.5,0)and(.9,-.55)..(1.6,-.6)..controls (2.2,-.6) and (2.6,.1)..(3,.1);
\end{scope}
\begin{scope}[yshift=1cm,xshift=-.2cm] 
\draw[thick] (0,0)..controls (.4,0) and (.9,.8)..(1.5,.8)..controls (2,.7) and (2.5,0.1)..(3,.1);
\draw[thick](0,0)..controls (.6,0)and(1.1,-.55)..(1.8,-.6)..controls (2.2,-.6) and (2.6,.1)..(3,.1);
\end{scope}
\begin{scope}[yshift=1.2cm,xshift=-2cm] 
\draw[thick] (0,0)..controls (.6,0) and (1,.6)..(1.8,.7)..controls (2.3,.7) and (3,-0.2)..(3.5,-.2);
\draw[thick](0,0)..controls (.5,0)and(.9,-.55)..(1.6,-.6)..controls (2.2,-.6) and (3,-.2)..(3.5,-.2);
\end{scope}
\end{scope}
}

\end{tikzpicture}
\caption{Proof of Theorem \ref{thm: existence of example}, Step 1: Construct a function on the tube bundle $W$ whose Cerf diagram has a copy of the zero section and a bunch of wrinkles.}
\label{Figure1906a}
\end{center}
\end{figure}
}

{
\begin{figure}[htbp]
\begin{center}
\begin{tikzpicture}[scale=.7]

{
\begin{scope}[xshift=1cm]
\begin{scope}
\draw[thick] (-3.5,0)--(4,0);
\draw (4,0) node[right]{$\Lambda$};
\draw[thick] (-3.5,0)..controls (-4,0) and (-4.4,-.5)..(-4.6,-.5);
\draw[thick] (-4.6,-.5)..controls (-4.2,-.5) and (-3.8,-1)..(-3.3,-1);
\draw[thick] (-3.3,-1)..controls (-3.9,-1) and (-4.3,-1.5)..(-4.5,-1.5);
\draw[thick] (-4.5,-1.5)..controls (-4.1,-1.5) and (-3.8,-2)..(-3.4,-2);
\draw[thick] (-3.4,-2)..controls (-3.9,-2) and (-4.1,-2.5)..(-4.5,-2.5);
\draw[thick] (-4.5,-2.5)..controls (-4.25,-2.5) and (-4,-2.8)..(-3.8,-2.8)..controls (-3.6,-2.8) and (-3.4,-2.6)..(-3.2,-2.6);
\draw[thick](-3.2,-2.6)..controls (-3.5,-2.6) and (-4.2,-2)..(-4.6,-2);
\draw[thick] (-4.6,-2)..controls (-4.3,-2) and (-3.7,-1.5)..(-3.4,-1.5);
\draw[thick] (-3.4,-1.5)..controls (-3.7,-1.5) and (-4.2,-1)..(-4.6,-1);
\draw[thick] (-4.6,-1)..controls (-4.2,-1) and (-3.7,-.5)..(-3.5,-.5);
\end{scope}
\begin{scope}
\draw[thick] (-3.5,-.5)..controls (-3.7,-.5) and (-5.3,.5)..(-5.5,.5);
\draw[thick] (-5.5,.5)..controls (-5.3,.5) and (-4.8,1)..(-4.4,1);
\draw[thick] (-4.4,1)..controls (-4.8,1) and (-5.2,1.5)..(-5.5,1.5);
\draw[thick] (-5.5,1.5)..controls (-5.2,1.5) and (-4.8,2)..(-4.5,2)..controls (-4.7,2) and (-4.9,2.2)..(-5,2.2)..controls (-5.2,2.2) and (-5.4,2)..(-5.6,2);
\draw[thick] (-5.6,2)..controls (-5.4,2) and (-4.8,1.5)..(-4.5,1.5);
\draw[thick] (-4.5,1.5)..controls (-4.8,1.5) and (-5.3,1)..(-5.5,1);
\draw[thick] (-5.5,1)..controls (-5.3,1) and (-4.8,.5)..(-4.4,.5);
\draw[thick] (-4.4,.5)..controls (-4.7,.5) and (-5.5,0)..(-6,0);
\end{scope}
\begin{scope}[yshift=2cm,xshift=-2.3cm] 
\draw[thick] (0,0)..controls (.6,0) and (1,.8)..(1.5,.8)..controls (2,.7) and (2.5,0.1)..(3,.1);
\draw[thick](0,0)..controls (.5,0)and(.9,-.55)..(1.6,-.6)..controls (2.2,-.6) and (2.6,.1)..(3,.1);
\end{scope}
\begin{scope}[yshift=1.8cm,xshift=-1.3cm] 
\draw[thick] (0,0)..controls (.6,0) and (1,.8)..(1.5,.8)..controls (2,.7) and (2.5,0.1)..(3,.1);
\draw[thick](0,0)..controls (.5,0)and(.9,-.55)..(1.6,-.6)..controls (2.2,-.6) and (2.6,.1)..(3,.1);
\end{scope}
\begin{scope}[yshift=1cm,xshift=-.2cm] 
\draw[thick] (0,0)..controls (.4,0) and (.9,.8)..(1.5,.8)..controls (2,.7) and (2.5,0.1)..(3,.1);
\draw[thick](0,0)..controls (.6,0)and(1.1,-.55)..(1.8,-.6)..controls (2.2,-.6) and (2.6,.1)..(3,.1);
\end{scope}
\begin{scope}[yshift=1.2cm,xshift=-2cm] 
\draw[thick] (0,0)..controls (.6,0) and (1,.6)..(1.8,.7)..controls (2.3,.7) and (3,-0.2)..(3.5,-.2);
\draw[thick](0,0)..controls (.5,0)and(.9,-.55)..(1.6,-.6)..controls (2.2,-.6) and (3,-.2)..(3.5,-.2);
\end{scope}
\end{scope}
}
\end{tikzpicture}
\caption{Proof of Theorem \ref{thm: existence of example}, Step 2: Obtain cusps of any Maslov index on the copy of the zero section by spinning Reidemeister I moves.}
\label{Figure1906b}
\end{center}
\end{figure}
}

{
\begin{figure}[htbp]
\begin{center}
\begin{tikzpicture}[scale=.7]
{
\begin{scope}[yshift=7.5cm]
\begin{scope}
\draw[thick] (-3.5,0)--(4,0);
\draw (4,0) node[right]{$\Lambda$};
\draw[thick] (-3.5,0)..controls (-4,0) and (-4.4,-.5)..(-4.6,-.5);
\draw[thick] (-4.6,-.5)..controls (-4.2,-.5) and (-3.8,-1)..(-3.3,-1);
\draw[thick] (-3.3,-1)..controls (-3.9,-1) and (-4.3,-1.5)..(-4.5,-1.5);
\draw[thick] (-4.5,-1.5)..controls (-4.1,-1.5) and (-3.8,-2)..(-3.4,-2);
\draw[thick] (-3.4,-2)..controls (-3.9,-2) and (-4.1,-2.5)..(-4.5,-2.5);
\draw[thick] (-4.5,-2.5)..controls (-4.25,-2.5) and (-4,-2.8)..(-3.8,-2.8)..controls (-3.6,-2.8) and (-3.4,-2.6)..(-3.2,-2.6);
\draw[thick](-3.2,-2.6)..controls (-3.5,-2.6) and (-4.2,-2)..(-4.6,-2);
\draw[thick] (-4.6,-2)..controls (-4.3,-2) and (-3.7,-1.5)..(-3.4,-1.5);
\draw[thick] (-3.4,-1.5)..controls (-3.7,-1.5) and (-4.2,-1)..(-4.6,-1);
\draw[thick] (-4.6,-1)..controls (-4.2,-1) and (-3.7,-.5)..(-3.5,-.5);
\end{scope}
\begin{scope}
\draw[thick] (-3.5,-.5)..controls (-3.7,-.5) and (-5.3,.5)..(-5.5,.5);
\draw[thick] (-5.5,.5)..controls (-5.3,.5) and (-4.8,1)..(-4.4,1);
\draw[thick] (-4.4,1)..controls (-4.8,1) and (-5.2,1.5)..(-5.5,1.5);
\draw[thick] (-5.5,1.5)..controls (-5.2,1.5) and (-4.8,2)..(-4.5,2)..controls (-4.7,2) and (-4.9,2.2)..(-5,2.2)..controls (-5.2,2.2) and (-5.4,2)..(-5.6,2);
\draw[thick] (-5.6,2)..controls (-5.4,2) and (-4.8,1.5)..(-4.5,1.5);
\draw[thick] (-4.5,1.5)..controls (-4.8,1.5) and (-5.3,1)..(-5.5,1);
\draw[thick] (-5.5,1)..controls (-5.3,1) and (-4.8,.5)..(-4.4,.5);
\draw[thick] (-4.4,.5)..controls (-4.7,.5) and (-5.5,0)..(-6,0);
\end{scope}
\begin{scope}[yshift=2cm,xshift=-2.3cm] 
\draw[thick] (0,0)..controls (.6,0) and (1,.8)..(1.5,.8)..controls (2,.7) and (2.5,0.1)..(3,.1);
\draw[thick](0,0)..controls (.5,0)and(.9,-.55)..(1.6,-.6)..controls (2.2,-.6) and (2.6,.1)..(3,.1);
\end{scope}
\begin{scope}[yshift=1.8cm,xshift=-1.3cm] 
\draw[thick] (0,0)..controls (.6,0) and (1,.8)..(1.5,.8)..controls (2,.7) and (2.5,0.1)..(3,.1);
\draw[thick](0,0)..controls (.5,0)and(.9,-.55)..(1.6,-.6)..controls (2.2,-.6) and (2.6,.1)..(3,.1);
\end{scope}
\begin{scope}[yshift=1cm,xshift=-.2cm] 
\draw[thick] (0,0)..controls (.4,0) and (.9,.8)..(1.5,.8)..controls (2,.7) and (2.5,0.1)..(3,.1);
\draw[thick](0,0)..controls (.6,0)and(1.1,-.55)..(1.8,-.6)..controls (2.2,-.6) and (2.6,.1)..(3,.1);
\end{scope}
\begin{scope}[yshift=1.2cm,xshift=-2cm] 
\draw[thick] (0,0)..controls (.6,0) and (1,.6)..(1.8,.7)..controls (2.3,.7) and (3,-0.2)..(3.5,-.2);
\draw[thick](0,0)..controls (.5,0)and(.9,-.55)..(1.6,-.6)..controls (2.2,-.6) and (3,-.2)..(3.5,-.2);
\end{scope}
\draw[dashed] (-3.2,-2.6)..controls (-2.8,-2.6) and (-2.4,1.2)..(-2,1.2);
\draw[dashed] (-4.4,.5)..controls (-3,.5)and (-1,1)..(-.2,1);
\draw[dashed] (-4.5,2)--(-2.3,2);
\draw[dashed] (-4.4,1)..controls (-4,1)and(-1.9,1.8)..(-1.3,1.8);
\end{scope}
}

{
\begin{scope}
\begin{scope}
\draw[thick] (-3.5,0)--(4,0);
\draw (4,0) node[right]{$\Lambda$};
\draw[thick] (-3.5,0)..controls (-4,0) and (-4.4,-.5)..(-4.6,-.5);
\draw[thick] (-4.6,-.5)..controls (-4.2,-.5) and (-3.8,-1)..(-3.3,-1);
\draw[thick] (-3.3,-1)..controls (-3.9,-1) and (-4.3,-1.5)..(-4.5,-1.5);
\draw[thick] (-4.5,-1.5)..controls (-4.1,-1.5) and (-3.8,-2)..(-3.4,-2);
\draw[thick] (-3.4,-2)..controls (-3.9,-2) and (-4.1,-2.5)..(-4.5,-2.5);
\draw[thick] (-4.5,-2.5)..controls (-4.25,-2.5) and (-4,-2.8)..(-3.8,-2.8)..controls (-3.6,-2.8) and (-3.4,-2.65)..(-3.2,-2.65);
\draw[thick](-3.2,-2.5)..controls (-3.5,-2.5) and (-4.2,-2)..(-4.6,-2);
\draw[thick] (-4.6,-2)..controls (-4.3,-2) and (-3.7,-1.5)..(-3.4,-1.5);
\draw[thick] (-3.4,-1.5)..controls (-3.7,-1.5) and (-4.2,-1)..(-4.6,-1);
\draw[thick] (-4.6,-1)..controls (-4.2,-1) and (-3.7,-.5)..(-3.5,-.5);
\end{scope}
\begin{scope}
\draw[thick] (-3.5,-.5)..controls (-3.7,-.5) and (-5.3,.5)..(-5.5,.5);
\draw[thick] (-5.5,.5)..controls (-5.3,.5) and (-4.8,.9)..(-4.4,.9);
\draw[thick] (-4.4,1.05)..controls (-4.8,1.05) and (-5.2,1.5)..(-5.5,1.5);
\draw[thick] (-5.5,1.5)..controls (-5.2,1.5) and (-4.8,1.9)..(-4.5,1.9);
\draw[thick] (-4.5,2.05)..controls (-4.7,2.05) and (-4.9,2.2)..(-5,2.2)..controls (-5.2,2.2) and (-5.4,2)..(-5.6,2);
\draw[thick] (-5.6,2)..controls (-5.4,2) and (-4.8,1.5)..(-4.5,1.5);
\draw[thick] (-4.5,1.5)..controls (-4.8,1.5) and (-5.3,1)..(-5.5,1);
\draw[thick] (-5.5,1)..controls (-5.3,1) and (-4.8,.55)..(-4.4,.55);
\draw[thick] (-4.4,.4)..controls (-4.7,.4) and (-5.5,0)..(-6,0);
\end{scope}
\begin{scope}[yshift=2cm,xshift=-2.3cm] 
\draw[thick] (0,0.05)..controls (.6,0.05) and (1,.8)..(1.5,.8)..controls (2,.7) and (2.5,0.1)..(3,.1);
\draw[thick](0,-.1)..controls (.5,-.1)and(.9,-.55)..(1.6,-.6)..controls (2.2,-.6) and (2.6,.1)..(3,.1);
\end{scope}
\begin{scope}[yshift=1.8cm,xshift=-1.3cm] 
\draw[thick] (0,0.05)..controls (.6,0.05) and (1,.8)..(1.5,.8)..controls (2,.7) and (2.5,0.1)..(3,.1);
\draw[thick](0,-.1)..controls (.5,-.1)and(.9,-.55)..(1.6,-.6)..controls (2.2,-.6) and (2.6,.1)..(3,.1);
\end{scope}
\begin{scope}[yshift=1cm,xshift=-.2cm] 
\draw[thick] (0,0.05)..controls (.4,0.05) and (.9,.8)..(1.5,.8)..controls (2,.7) and (2.5,0.1)..(3,.1);
\draw[thick](0,-.1)..controls (.6,-.1)and(1.1,-.55)..(1.8,-.6)..controls (2.2,-.6) and (2.6,.1)..(3,.1);
\end{scope}
\begin{scope}[yshift=1.2cm,xshift=-2cm] 
\draw[thick] (0,0)..controls (.6,0) and (1,.6)..(1.8,.7)..controls (2.3,.7) and (3,-0.2)..(3.5,-.2);
\draw[thick](0.1,-.15)..controls (.5,-.15)and(.9,-.55)..(1.6,-.6)..controls (2.2,-.6) and (3,-.2)..(3.5,-.2);
\end{scope}
\draw[thick] (-3.2,-2.5)..controls (-2.8,-2.5) and (-2.4,1.2)..(-2,1.2);
\draw[thick] (-3.2,-2.65)..controls (-2.6,-2.65) and (-2.3,1.05)..(-1.9,1.05);
\draw[thick] (-4.4,.55)..controls (-3,.55)and (-1,1.05)..(-.2,1.05);
\draw[thick] (-4.4,.4)..controls (-3,.4)and (-1,.9)..(-.2,.9);
\draw[thick] (-4.5,2.05)--(-2.3,2.05);
\draw[thick] (-4.5,1.9)--(-2.3,1.9);
\draw[thick] (-4.4,1.05)..controls (-4,1.05)and(-1.9,1.85)..(-1.3,1.85);
\draw[thick] (-4.4,.9)..controls (-4,.9)and(-1.9,1.7)..(-1.3,1.7);
\end{scope}
}
\end{tikzpicture}
\caption{Proof of Theorem \ref{thm: existence of example}, Step 3: Surgering the wrinkles into the base.}
\label{Figure1907}
\end{center}
\end{figure}
}

\begin{remark}
Given $g:\bR^N \to \bR$ a function of quadratic tube type, $\nabla g:\bR^N \to \bR^N$ is homogeneous of degree 1 and has degree $\pm 1$, hence may be viewed as a map $S^N \to S^N$. This gives a map $\cT_{k,N} \to G_N$ for $G_N$ the space of self-homotopy equivalences of $S^N$ which stabilizes to a map $D: \cT_\infty \to G$ called the {\em Waldhausen derivative}. Waldhausen proved in \cite{W82} that $D$ factors as a composition $\cT_\infty \to BG \to G$ and this factorization played a major role in the article \cite{AACK}. It is not hard to modify the proof of Theorem \ref{thm: existence of example} using the full power of Theorem 2.3 of \cite{EM00} (rather than its special case Theorem \ref{thm: wrinkling}) to prove that the conclusion of Theorem \ref{thm: existence of example} holds for any tube bundle $W \to M$ with trivial Waldhausen derivative, i.e. for any such $W$ there exists a function of tube type $f:M \times \bR^N \to \bR$ with $W=T(f)$ and which generates a Legendrian $\Lambda \subset J^1M$ whose projection $\Lambda \to M$ is homotopic to a diffeomorphism. This result implies Theorem \ref{thm: existence of example} because if the classifying map $M \to \cT_\infty$ of $W$ lifts to a map $M \to \cH_\infty$ then the Waldhausen derivative $M \to G$ factors as a composition $M \to \cH_\infty \to \cT_\infty \to BG \to G$ and hence in particular factors through the composition of the two consecutive maps of the fibration sequence and thus is canonically nullhomotopic. 

\end{remark}

\section{Nearby Lagrangian torsion}\label{sec: nearby}

\subsection{Tube torsion of nearby Lagrangians}

Recall from Section \ref{sec: stable gauss map} that for a nearby Lagrangian $L \subset T^*M$ to have any chance of admitting any kind of generating function, the stable Gauss map $L \to U/O$ of $L$ must be trivial, i.e. nullhomotopic. As was also discussed in Section \ref{sec: gen fun nearby}, this condition is also sufficient to ensure existence of generating functions of tube type: 

\begin{theorem}[\cite{ACGK}]\label{thm: exist for stat triv}
Let $L \subset T^*M$ be a nearby Lagrangian with trivial stable Gauss map $L \to U/O$. Then $L$ admits a generating function of tube type $f:M \times \bR^N \to \bR$.\end{theorem} 

To be precise, in \cite{ACGK} this result was proved for a slightly different notion of function of tube type, the space of which is close to the space we denoted by $\cT_{N,k}^p$ in Section \ref{sec: SpaceOfTubes}. In \cite{AACK} it was shown that the statement of Theorem \ref{thm: exist for stat triv} implies the same result for the notion of a function of tube type considered in the present article, i.e. in the space we denoted by $\cT_{N,k}$. Strictly speaking, this is only proved for a $C^1$-close Hamiltonian isotopic nearby Lagrangian $L' \subset T^*M$ rather than for $L$ itself, but one may nevertheless conclude the desired conclusion as stated above by applying the homotopy lifting property for generating functions under Legendrian isotopies after stabilization. 

For a nearby Lagrangian $L$ with stably trivial Gauss map the tube torsion $\tau(L)$ consists of all (negatives of) higher torsion classes $\tau_{2k}(W) \in H^{4k}(M;\bR)$ of orientable tube bundles $W \to M$ on which $L$ admits a generating function. 


\begin{theorem}\label{thm: tube torsion nearby Lag}
Let $L\subset T^*M$ be a nearby Lagrangian whose stable Gauss map $L \to U/O$ is trivial. Then the tube torsion $\tau(L) \subset H^*(M;\bR)$ of $L$ is well-defined. It consists of a nonempty union of cosets of  $\text{im}(p)$ and is a Legendrian isotopy invariant of the Legendrian lift of $L$. In particular, if $\tau(L) \neq \text{im}(p)$ then $L$ is not Hamiltonian isotopic to the zero-section. 
\end{theorem}

\begin{proof} This follows from Theorem \ref{thm: main 1} and Theorem \ref{thm: exist for stat triv}. \end{proof} 

The condition that the stable Gauss map is trivial is automatically satisfied for all nearby Lagrangians if $M$ has the homotopy type of a sphere, since as was also discussed in Section Section \ref{sec: stable gauss map} it was also proved in \cite{ACGK} that the stable Gauss map $L \to U/O$ of a nearby Lagrangian is always zero on homotopy groups, hence automatically trivial when $M$ (and hence also $L$) has the homotopy type of a sphere.

\begin{theorem}\label{thm: tube tor Lag spheres}
Let $L\subset T^*\Sigma$ be a nearby Lagrangian, where $\Sigma$ (and therefore also $L$) is an $n$-dimensional homotopy sphere. Then the tube torsion $\tau(L) \subset H^*(M;\bR)$ of $L$ is well-defined. It consists of a nonempty union of cosets of $\text{im}(p)$ and is a Legendrian isotopy invariant of the Legendrian lift of $L$. In particular, if $\tau(L) \neq \text{im}(p)$ then $L$ is not Hamiltonian isotopic to the zero-section. 
\end{theorem}

\begin{proof} This follows just as Theorem \ref{thm: tube tor Lag spheres}  above but using the fact that nearby Lagrangian homotopy spheres all have trivial stable Gauss map \cite{ACGK} to remove the extra hypothesis. \end{proof}

The existence theorem for generating functions of tube type from \cite{ACGK} is proved by a construction which is canonical once a nullhomotopy of the stable Gauss map $L \to U/O$ is fixed. We will now show that this construction picks out a distinguished coset $\nu(L) \subset \tau(L)$ of  $\text{im}(p)$ which will turn out to be invariant under Hamiltonian isotopies of $L$, though not obviously so under Legendrian isotopies of the Legendrian lift of $L$. 

\subsection{The Guillermou trick} 
\label{sec: DoublingConstruction}

The following result is implicit in \cite{ACGK}.

\begin{proposition}\label{prop: natural coset}
Let $L \subset T^*M$ be a nearby Lagrangian with trivial stable Gauss map $L \to U/O$. The doubling construction determines a distinguished coset of the rigid tube map $[M,BO] \to [M,\cT_\infty ]$ which is invariant under Hamiltonian isotopy of $L$.
\end{proposition}

Indeed, \cite{ACGK} identifies the colimit of the left action of stabilization by $x^2-y^2$ on $\cT$ as the space $\bZ \times \bZ \times \cT_\infty$. This colimit still has a right action of $\cQ$, so we may form the (geometric realization) of the bar construction $B(\bZ \times \bZ \times \cT_\infty,\cQ)$. To any nearby Lagrangian $L \subset T^*M$ the (twisted) doubling construction associates a homotopy class of map  $M \to B(\bZ \times \bZ \times \cT_\infty,\cQ)$, invariant under Hamiltonian isotopy of $L$ (this map classifies a {\em twisted generating function}). We have a fibration sequence by Diagram 4.2 of \cite{ACGK} which reads
$$\bZ \times BO \to  \bZ  \times \cT_\infty \to B(\bZ \times \bZ \times \cT_\infty,\cQ) \to U/O$$
and such that the composition $M \to B(\bZ \times \bZ \times \cT_\infty,\cQ) \to U/O$ is the stable Gauss map of $L$ under a homotopy inverse for the projection $L \to M$. So the desired statement follows formally from the fact that a nullhomotopy of $M \to U/O$ is equivalent to a lift $M \to \bZ \times  \cT_\infty$ of the map $M \to B(\bZ \times \bZ \times \cT_\infty,\cQ)$, which is homotopically unique up to a lift $M \to \bZ \times BO$. In fact, if we know a priori that $L$ has trivial stable Gauss map we may avoid the bar construction and we briefly explain how this works. 

Given a Legendrian submanifold $\Lambda \subset J^1M$, a local generating function for $\Lambda$ consists of a generating function $f:U \to \bR$ where $U \subset M \times \bR^N$ is some open subset. Since $U$ could be an arbitrarily small neighborhood of the fiberwise singular set of $f$ this notion amounts to a purely linear algebraic question. Indeed, as discussed in Section \ref{sec: stable gauss map} it goes back to Giroux \cite{G90} that the existence of a so-called local generating function is equivalent to the triviality of the stable Gauss map $\Lambda \to U/O$. Moreover, as is explained in \cite{G90} and further discussed in \cite{ACGK}, one may take as input a nullhomotopy of the stable Gauss map $\Lambda \to U/O$ to construct such a local generating function. If one were to change the chosen null-homotopy of the stable Gauss map by concatenating with a map $\Lambda \to \Omega (U/O),$ then this would correspond to performing a twisted stabilization of the local generating function by a family of quadratic forms on $\Lambda.$ The negative eigenspaces of the fiberwise Hessian form a real vector bundle $E$ over $\Lambda$ classified by the map $\Lambda \to \bZ \times BO$ corresponding to the given map $\Lambda \to \Omega (U/O)$ under the model for Bott periodicity $\Omega(U/O) \simeq \bZ \times BO$ considered in \cite{G90}. 

 The construction for generating functions of tube type for nearby Lagrangians $L \subset T^*M$ with trivial stable Gauss map $L \to U/O$ takes a nullhomotopy of the Gauss map and the corresponding local generating function $f:U \to \bR$ where $U \subset M \times \bR^{N+1}$ as its starting point, and then argues as follows. Denote $\bR^{N+1} = \bR \times \bR$ with $x \in \bR$ the coordinate of the second factor. By a direct sum with the function $x^3+3x$ one may construct a local generating function for a {\em double} of $\Lambda$, which consists of two parallel copies of $\Lambda$ in $J^1M$, namely $\Lambda_0 \coprod \Lambda_1$ for $\Lambda_0=\Lambda$ and $\Lambda_\epsilon$ a shift of $\Lambda$ by a small but constant amount $\epsilon \in \bR_+$ in the $z$-coordinate $J^1M=T^*M \times \bR_z$. Further, by a judicious choice of cutoffs, one may arrange for the double of $\Lambda$ to be generated by a global, i.e. not local, generating function $h_s:M \times \bR^{N+1} \to \bR$ such that away from the critical set there holds $\partial h/ \partial x \neq 0$ for the distinguished coordinate $x$ of $\bR^{N+1}=\bR^N \times \bR$. 

When $\Lambda$ happens to be the Legendrian lift of a nearby Lagrangian $L \subset T^*M$, one may use the following remarkable trick due to Guillermou \cite{Gui12}. Consider the family $\Lambda_0 \coprod \Lambda_s$ as $s>0$ gets larger. Since $L$ is an embedded Lagrangian, $\Lambda$ has no Reeb chords and hence $\Lambda_0 \coprod \Lambda_s \subset J^1M$ is a Legendrian isotopy. Possibly after increasing $N$, by the homotopy lifting property for generating functions there continues to exists a family of functions $h_s:M \times \bR^N \to \bR$ generating  $\Lambda_0 \coprod \Lambda_s$ such that $\partial h_s / \partial x \neq 0$ outside of a compact subset. When $s=T$ is big enough so that there exists a value $c \in \bR$ such that the $z$-coordinate of $\Lambda$ is always below $c$ and the $z$-coordinate for $\Lambda_T$ is always above $c$, we may consider the codimension zero sub-bundle $\{h_T \leq c \} \subset M \times \bR^N$. It is shown in \cite{ACGK} that the resulting fiber bundle is a bundle of tubes in a sense which is close to Waldhausen's parition model $\cT_{N,k}^p$, which stably has the same homotopy type as the space of functions of tube types $\cT_{N,k}$ in the sense of the present article as discussed in Section \ref{sec: SpaceOfTubes}. The homotopy class of the resulting map $M \to \cT_\infty$ only depends on the choice of nullhomotopy of the stable Gauss map $L \to U/O$.


The monoid $[M,\Omega(U/O)]$ (with respect to concatenation of loops in $\Omega(U/O)$) acts transitively on the set of homotopy classes of nullhomotopies of the stable Gauss map $L \to U/O$ of any nearby Lagrangian $L \subset T^*M$, since we may concatenate any such nullhomotopy with the resulting map $L \to M \to \Omega (U/O)$. As discussed above, the effect on the resulting class in $[M, \cT_\infty]$ is given by the action of $[M, BO]$ with respect to the twisted stabilization, i.e. with respect to the rigid tube map $BO \to \cT_\infty$, under the model for Bott periodicity $\Omega(U/O) \simeq \bZ \times BO$ of \cite{G90}.  

Finally, the doubling construction works parametrically, as long as the path of Lagrangians is covered by a homotopy of null-homotopies of the stable Gauss map, so the resulting coset of the map $[M,BO] \to [M, \cT_\infty] $ is a Hamiltonian isotopy invariant.

\begin{definition} The coset of the map $[M,BO] \to [M, \cT_\infty] $  distinguished by Proposition \ref{prop: natural coset} is be called the {\em natural coset}.\end{definition}


\begin{definition}
Let $L \subset T^*M$ be a nearby Lagrangian with trivial stable Gauss map $L \to U/O$. The {\em nearby Lagrangian torsion} $\nu(L) \subset \tau(L)$ is the subset of $H^*(M;\bR)$ consisting of the (negatives of the) higher torsion classes of tube bundles in the natural coset of $L$. 
\end{definition}


\begin{proof}[Proof of Theorem \ref{thm: main 2} ]$ \,$
  \begin{enumerate} 
   \item that $\nu(L)$ is invariant under Hamiltonian isotopies of $L$ follows from Proposition \ref{prop: natural coset}.
  \item that $\nu(L)$ is a single coset of $p$ follows from Proposition \ref{prop: effect twisted stab}.
   \item for the zero section we have $\tau(M)= \text{im}(p)$ by Theorem \ref{thm: main 1}, so a fortiori $\nu(M) = \text{im}(p)$.
  \end{enumerate}
\end{proof}

\begin{definition}
\label{def: natural}
Let $L \subset T^*M$ be a nearby Lagrangian generated by a generating function of tube type $f:M \times \bR^N \to \bR$. We say that $f$ is {\em{natural}} if it arises from the doubling construction.
\end{definition}

\begin{remark}
To be precise, this means that the doubling construction of \cite{ACGK} produces a generating function $g$ for $L$ which is of tube type in the sense of \cite{ACGK} and such that via the modification of \cite{AACK} the function $g$ corresponds to $f$. So $f$ is a fibered function of tube type in the sense of the present article. Note that $f$ being natural implies, but is not a priori equivalent to, that the tube bundle $W=T(f)$ is in the natural coset of the rigid tube map. 
\end{remark}

\begin{remark}
For $i=0,1$ let $L_i \subset T^*M$ be a nearby Lagrangian and let $\Lambda_i \subset J^1M$ be corresponding Legendrian lifts. It could in principle be possible that $L_0$ and $L_1$ are not Hamiltonian isotopic but $\Lambda_0$ and $\Lambda_1$ are Legendrian isotopic. If so, then $\tau(\Lambda_0)=\tau(\Lambda_1)$, i.e. their tube torsions agree. However the authors do not know whether tube torsion need be a single $p$-coset, in principle it could be larger. In particular it could be the case that $\nu(L_0) \neq \nu(L_1)$, i.e. that the nearby Lagrangian torsion $p$-cosets are different, even though the (larger) collections of $p$-cosets $\tau(\Lambda_0)=\tau(\Lambda_1)$ coincide.
\end{remark}

\subsection{Nearby Lagrangians with moderate tangencies}\label{sec: gen nl mod}

Suppose that $L \subset T^*M$ is a nearby Lagrangian whose projection $L \to M$ only has fold type singularities, or equivalently such that the Legendrian front (of Legendrian lift) of $L$ in $J^0M =M \times \bR$ only has cusp singularities. Yet another equivalent condition is that $L$ has quadratic tangencies with respect to the foliation of $T^*M$ by the cotangent leaves $T_m^*M$, $m \in M$, see Figure \ref{Figure1905}.

\begin{definition} We will say that $L \subset T^*M$ has {\em moderate tangencies} when any (and therefore all) of the three equivalent conditions stated above hold. \end{definition}

Any nearby Lagrangian has trivial Maslov class \cite{A12} and so as was explained in Section \ref{sec: stable gauss map} whenever $L$ has moderate tangencies the stable Gauss map $L \to U/O$ is nullhomotopic. In fact we have a canonical nullhomotopy associated to each bounding chain for the cusp locus (co-oriented with the Maslov co-orientation).  In particular the tube torsion $\tau(L)$ of $L$ is well-defined, and so is the nearby Lagrangian torsion $\nu(L) \subset \tau(L)$.

Let $L \subset T^*M$ is a nearby Lagrangian with moderate tangencies, then any generating function of tube type $f:M \times \bR^N \to \bR$ for $L$ has moderate singularities.  As explained in Appendix \ref{sec: higher torsion of bundles}, for each generating function of tube type $f$ for $L$ we therefore have an associated cohomology class 
 $\tau_{2k}(f) \in H^{4k}(M;\bR)$ which is the pull-back of the universal class $\tau_{4k}$ in  $H^{4k}(\text{\em Wh}_\bullet^h(\bC,U(1));\bR)$ by the map
$$ \xi_\bullet: simp\,M\to \text{\em Wh}_\bullet^h(\bZ[\pi_1W],\pi_1W) $$
obtained from the family of fiberwise Morse chain complexes $C_*(f_m)$, $m \in M$.
\begin{theorem}
Let $L \subset T^*M$ is a nearby Lagrangian with moderate tangencies and $f:M \times \bR^N \to \bR$  a generating function of tube type for $L$. 
\begin{enumerate}
\item For $W=T(f)$ the class $-\tau_{2k}(W)$ is in the same coset of $\text{im}(p_k)$ as the class $\tau_{2k}(f)$.
\item There exists a framed generating function of tube type $h:M \times \bR^K \to \bR$ for $L$ such that $\tau_{2k}(f)$ and $\tau_{2k}(h)$ are in the same coset of $\text{im}(p_k)$.
\item If $L$ is Hamiltonian isotopic to the zero section, then $\tau_{2k}(f) \in \text{im}(p_k)$.
\end{enumerate}
\end{theorem}

\begin{proof} Here we use the fact that the projection $L \to M$ of a nearby Lagrangian to the base is a homotopy equivalence \cite{A12}.
\begin{enumerate} 
\item This follows from Proposition \ref{prop: torsion for special Leg}.
\item This follows from Proposition \ref{prop: framing for special Leg}.
\item This follows from item (1) above combined with Theorem \ref{thm: main 1}.
\end{enumerate}
\end{proof}

In particular, the tube torsion $\tau(L)$ of a nearby Lagrangian with moderate singularities tangencies $L \subset T^*M$ be computed directly from the parametrized Morse theory of generating functions of tube type for $L$. Moreover, the nearby Lagrangian torsion $\nu(L)$ of a nearby Lagrangian with moderate tangencies $L \subset T^*M$ be computed directly from the parametrized Morse theory of any single natural generating function of tube type for $L$. 

Let $\Lambda \subset J^1M$ be a Legendrian lift of the exact Lagrangian submanifold $L \subset T^*M$ (well-defined up to translation in the $z \in \bR$ direction) and let  $\pi_F:J^1M \to J^0M = M \times \bR$ be the front projection. A basic constraint on the parametrized Morse theory of $f:M \times \bR^N \to \bR$  is that the Cerf diagram of such a generating function, i.e. the subset $\{(m,z) \in M \times \bR : \, \text{$z$ is not a regular value of $f_m:\bR^N \to \bR$} \} $ of $M \times \bR$, is precisely the front projection $\pi_{F}(\Lambda) \subset  M \times \bR$ of $\Lambda$.

The front projection of a nearby Lagrangian $L \subset T^*M$ is dramatically restricted by the {\em absence of Reeb chords} in the Legendrian lift $\Lambda \subset J^1M$ of $L$, i.e. pairs of distinct points in $J^1M$ which project to the same point in $T^*M$. Geometrically this means the front $\pi_F(\Lambda) \subset M \times \bR$ can have no pairs of point vertically above each other in the $z \in \bR$ direction for which the tangent hyperplanes to $\pi_F(\Lambda)$ at those two points coincide. For example, this forbids any weaving of 1-dimensional fronts in $J^0S^1=S^1 \times \bR$. 

It is plausible that this strong constraint on the front $\pi_F(\Lambda) \subset M \times \bR$, and hence on the filtration of the fiberwise Morse chain complexes $C_*(f_m)$, $m \in M$, could force the higher torsion $\tau_{2k}(f)$ to be zero. Indeed, the filtration on the chain complexes precisely follows the front $\pi_F(\Lambda)$. However, the authors of the present article do not know the answer to the following question.

\begin{question} Let $L \subset T^*M$ be a nearby Lagrangian with moderate tangencies and $f:M \times \bR^N \to \bR$ a generating function of tube type for $L$. Must it be the case that $\tau_{2k}(f) \in \text{im}(p_k)$? \end{question}

An affirmative answer would follow from the nearby Lagrangian conjecture. Therefore a negative answer would disprove the nearby Lagrangian conjecture. 

In the future work \cite{AIS} we show that the answer is affirmative if we assume that the front of $L$ carries a {\em restricted ruling}, which is a higher dimensional version of the familiar notion of a ruling of a 1-dimensional Legendrian front. However, our notion of a {\em restricted ruling} is rather restrictive, and does not permit higher codimension bifurcations, which would be inevitable in any sufficiently general notion of higher-dimensional ruling, even under the assumption that $L \subset T^*M$ has moderate tangencies. Here {\em sufficiently general} is taken to mean such that the fiberwise Morse theory of a generic generating function of tube type $f:M \times \bR^N \to \bR$ with moderate singularities would produce a family of barcodes $B_m$, $m \in M$, whose upper endpoints satisfies the conjectural definition of a ruling.

\subsection{Nearby Lagrangian homotopy spheres}

We now discuss the special case of nearby Lagrangian homotopy spheres. First, note that for $L \subset T^*M$ a nearby Lagrangian, the projection $L \to M$ is a homotopy equivalence and hence if $M$ is a homotopy sphere then $L$ is also a homotopy sphere. The stable Gauss map of any nearby Lagrangian was shown to be zero on homotopy groups in \cite{ACGK} and hence when $L$ has the homotopy type of a sphere the map $L \to U/O$ is nullhomotopic. Combining this input with the the h-principle for the simplification of singularities of Lagrangian fronts \cite{AG18a,AG18b} the first author and Darrow proved the following result. 

\begin{theorem}[\cite{AD22}]\label{thm: near spheres sing}
If $L \subset T^*M$ is a nearby Lagrangian homotopy sphere, then $L$ is Hamiltonian isotopic to a nearby Lagrangian homotopy sphere $L' \subset T^*M$ with moderate tangencies. 
\end{theorem}

Hence for nearby Lagrangian homotopy spheres $L \subset T^*M$ we may always assume from the onset that $L$ has moderate tangencies, so in particular as discussed in Section \ref{sec: gen nl mod} for any generating function $f:M \times \bR^N \to \bR$ of tube type we may compute the corresponding higher torsion directly from the parametrized Morse theory of $f$.

Let us put $M=S^n$ for concreteness, and note that unless $n=4k$ all higher torsion classes of tube bundles are trivial. Note that any tube bundle over $S^{4k}$ is orientable. For a tube bundle $W \to S^{4k}$ there is exactly one higher torsion class which might be nontrivial, namely $\tau_{2k}(W) \in H^{4k}(S^{4k};\bR)$. So for a nearby Lagrangian $L \subset T^*S^{4k}$ the tube torsion $\tau(L)$ and nearby Lagrangian torsion $\nu(L)$ are subsets of $H^{4k}(S^{4k};\bR) \simeq \bR$, where the isomorphism is by pairing with the fundamental class. 

Since $KO^0(S^4) \simeq \bZ$, the image of $p:KO^0(S^4) \to H^{4k}(S^{4k};\bR)$ is of the form $\lambda_k \bZ \subset \bR$, namely $\lambda_k=p_k(E)$ for $E \to S^{4k}$ a generator. So $\nu(L) \in \bR/\lambda_k \bZ$ is just a real number modulo $\lambda_k \bZ$ and $\tau(L) \subset \bR/\lambda_k \bZ$ is just a collection of real numbers modulo $\lambda_k \bZ$. Let $J_k$ denote the index of the monomorphism $\bZ \to \bZ$ given by
$$ \pi_{4k} BO  \to (\pi_{4k} \cT_\infty )/ \text{Tors}  $$
As noted in Section \ref{sec: SpaceOfTubes}, higher torsion determines a monomorphism $ \bZ/J_k \bZ \to \bR / \lambda_k \bR $ given by
$$ \left (( \pi_{4k} \cT_\infty) / \text{Tors}  \right) /  \pi_{4k} BO  \to \bR/\lambda_k \bZ ,$$
and hence any tube bundle $W \to S^{4k}$ whose class in the quotient 
$$  \left ( (\pi_{4k} \cT_\infty) / \text{Tors}  \right) /  \pi_{4k} BO    $$
is non-trivial has a nontrivial higher torsion, i.e. $\tau_{2k}(W) \not\in \text{im}(p)$. With respect to the above identifications the possible values of higher torsion for any tube bundle $W \to S^{4k}$ are
$$ \tau_{2k}(W) \in \left\{ 0, \frac{1}{J_k} \lambda_k , \frac{2}{J_k} \lambda_k , \cdots \frac{J_k-1}{J_k} \lambda_k  \right\}  \simeq \bZ / J_k \bZ , \qquad  \frac{j}{J_k} \lambda_k \leftrightarrow j+J_k\bZ $$

We summarize the state of affairs in the following theorem, all of which has already been proved.

\begin{theorem} Let $L \subset T^*S^{4k}$ be any nearby Lagrangian homotopy sphere.
\begin{enumerate}
\item to $L$ is assigned its nearby Lagrangian torsion $\nu(L) \in \bZ/J_k\bZ$
\item $\nu(L)$ is a Hamiltonian isotopy invariant of $L$.
\item after an initial Hamiltonian isotopy we may assume that $L$ has moderate tangencies.
\item then $\nu(L)$ is computed by $\tau_{2k}(f)$ for $f$ any natural generating function of tube type for $L$. 
\item it suffices to consider framed natural generating functions of tube type.
\item $\nu(L)=0+J_k \bZ$ for $L=S^{4k} \subset T^*S^{4k}$ the zero section.
\item $\nu(L) \subset \tau(L)$ for a finite subset $\tau(L) \subset \bZ/J_k\bZ$, the tube torsion of $L$.
\item  $\tau(L)$ is computed by the collection of all $\tau_{2k}(f)$ for $f$ a generating function of tube type for $L$. 
\item it suffices to consider framed generating functions of tube type.
\item $\tau(L)$ is a Legendrian isotopy invariant of the Legendrian lift of $L$.
\item $\tau(L)=\{0+J_k \bZ\}$  for $L=S^{4k} \subset T^*S^{4k}$ the zero section.
\end{enumerate}
\end{theorem}

The simplest closed manifold on whose cotangent bundle one might search for a counterexample to the nearby Lagrangian conjecture using the theory of higher torsion is $S^4$. The first step would be to understand an explicit Legendrian $\Lambda \subset J^1S^4$ with nontrivial tube torsion, i.e. not in $\text{im}(p_1)$. Presumably one has an explicit function of tube $f:S^4 \times \bR^N \to \bR$ and understands the 4-parameter family of fiberwise Morse chain complexes enough to compute $\tau_{2}(f)\in H^4(S^4;\bR)$ and identify it as not in $\text{im}(p_1)$.  Then one  might try to modify the generating function $f$ to remove whatever Reeb chords are present in $\Lambda$ and, if lucky, arrive at a suitable generating function of tube type $g:S^4 \times \bR^N \to \bR$ which agrees with $f$ at infinity and which generates a Legendrian with no Reeb chords, i.e. a nearby Lagrangian $L \subset T^*S^4$.  Since the underlying tube bundle of $f$ and $g$ are the same, the class $\tau_2(f)$ is in $\tau(L)$, and hence $\tau(L) \neq \text{im}(p_1)$ so $L$ is the desired counterexample to the nearby Lagrangian conjecture. Of course, {\em one doesn't simply remove Reeb chords}, so the above strategy is highly dubious. However, it would still be interesting to understand an explicit Legendrian $\Lambda \subset J^1S^4$ with nontrivial tube torsion.

\subsection{Floer theoretic perspective}\label{sec: hol}

The nearby Lagrangian torsion of a nearby Lagrangian $L \subset T^*M$ can probably be defined via Floer theory, loosely speaking by considering the family of Floer cochain complexes $CF^*(T^*_mM,L)$, $m \in M$, and upgrading the family Floer perspective to the K-theoretic level. Somewhat more concretely, if the tangencies of $L$ with respect to the foliation by cotangent fibres $T^*_mM$, $m \in M$, then after a suitable bifurcation analysis in exact Lagrangian Floer theory, analogous to the corresponding analysis in parametrized Morse theory and generalizing the 1 and 2 parametric case treated in the work of the third author \cite{S02}, one may expect to use the family of Floer cochain complexes $CF^*(T^*_m M,\varphi_1(L))$, $m \in S^{4k}$, to define a map $M \to  |\text{\em Wh}_\bullet^h(\bC,1)|$.

\begin{conjecture}
Let $L \subset T^*M$ be a nearby Lagrangian with moderate tangencies with respect to the foliation by cotangent fibres. The family of Floer cochain complexes $CF^*(T^*_xS^{4k},\varphi_1(L))$, $x \in S^{4k}$ determines a map $S^{4k} \to  |\text{\em Wh}_\bullet^h(\bC,1)|$ whose homotopy class is independent of auxiliary choices (such as a suitable almost complex structure). 
\end{conjecture}

It is not known whether any nearby Lagrangian is Hamiltonian isotopic to one with moderate tangencies. A necessary condition for this to be possible is that the stable Gauss map is trivial. Let us concentrate on the case where $M=S^{4k}$, since the condition on the stable triviality of the Gauss map is automatically satisfied by \cite{ACGK} and moreover this condition is also sufficient for the simplification of tangencies to be possible \cite{AD22}. 

Given a Lagrangian homotopy sphere $L \subset T^*S^{4k}$, fix a choice $\gamma_t$ of nullhomotopy of the stable Gauss map $L \to U/O$. Then Theorem \ref{thm: near spheres sing} as formulated in \cite{AD22} yields a Hamiltonian isotopy $\varphi_t$ such that the singularities of tangency the Lagrangian $K=\varphi_1(L)$ with respect to the foliation by cotangent fibers are all of fold type. Although the nearby Lagrangian $K$ produced by \cite{AD22} does not depend uniquely on the nullhomotopy $\gamma_t$ (there are many choices involved), the parametric h-principle for the simplification of caustics \cite{AG18a,AG18b} implies that $K$ is unique up to a Hamiltonian isotopy in which the tangencies remain moderate (birth/death is allowed along double folds). In particular it is natural to expect the following.

\begin{conjecture} The homotopy class of map $S^{4k} \to  |\text{\em Wh}_\bullet^h(\bC,1)|$ corresponding to the nearby Lagrangian with moderate tangencies $K=\varphi_1(L)$ only depends on the nullhomotopy $\gamma_t$ of the stable Gauss map $L \to U/O$.
\end{conjecture}

By pullback of the universal class one therefore obtains a cohomology class $\tau_{2k}(L,\gamma_t) \in H^{4k}(S^{4k},\bR)$ associated to any nullhomotopy $\gamma_t$ of the stable Gauss map $L \to U/O$. Two different choices of nullhomotopy $\gamma_t$ will differ by a map $L \to \Omega (U/O) = \bZ \times BO$. From analogy with the theory developed in the present article, it is natural to expect the following additional property to hold:

\begin{conjecture} The  cohomology classes $\tau_{2k}(L,\gamma_t^i) \in H^{4k}(S^{4k};\bZ)$, $i=1,2$, corresponding to two different nullhomotopies $\gamma_t^0$ and $\gamma_t^1$of $L \to U/O$ differ by $p_k(E)$ for $E \to S^{4k}$ the stable real vector bundle over $M$ classified by the map $S^{4k} \to  BO$ which is the composition of
\begin{enumerate}
\item a homotopy inverse $S^{4k} \to L$ for the projection $L \to S^{4k}$.
\item the map $L \to \Omega(U/O)$ obtained by concatenating $\gamma_{1-t}^0$ with $\gamma_t^1$ at each point of $L$ to obtain a loop in $U/O$.
\item the Bott periodicity map $\Omega(U/O) \simeq \bZ \times BO$ as in \cite{G90}.
\item the forgetful map $\bZ \times BO \to BO$.
\end{enumerate}  
\end{conjecture}

Thus one would obtain a well-defined coset of $\text{im}(p_k) \subset H^{4k}(S^{4k}, \bR)$,  i.e. a Floer-theoretic nearby Lagrangian torsion $\nu^{\text{FL}}(L) \in \bR/\lambda_k\bZ$ which is invariant under Hamiltonian isotopy of $L \subset T^*S^{4k}$. 

\begin{conjecture} This coset is the same as the nearby Lagrangian torsion, i.e. $\nu^{\text{FL}}(L)=\nu(L)$.\end{conjecture}

It would be interesting to know if there is a Floer theoretic proof of the  triviality of nearby Lagrangian torsion, as predicted by the nearby Lagrangian conjecture. Indeed, one might expect the Floer theoretic setup to be more sensitive to the property of being an exact Lagrangian submanifold than the Morse theoretic approach via generating functions. 
In the Floer set-up, the embeddedness implies there are no pseudo-holomorphic disks with a single boundary puncture that arise in as limiting-bubbles in Gromov compactness. Algebraically, this enables one to avoid relying on so-called augmentations from Legendrian/symplectic field theory.
In contrast, it is far from clear how the lack of Reeb chords is related to the paramerized Morse theory of a generating function, which after all is a fiberwise phenomenon. 
Indeed, the Morse-theoretic theory of higher torsion via generating functions makes sense for Legendrians in $J^1M$, whether or not they have Reeb chords.

\appendix

\section{Moderate singularities}\label{sec: moderate singularities}

\subsection{Functions with moderate singularities}

For $F$ a smooth manifold, a function $f:F \to \bR$ has {\em moderate singularities} if its singularities are modeled on either the Morse normal form $$\sum_{i\leq k}x_i^2-\sum_{i>k}x_i^2$$ or on the cubic (birth/death) form $$x_1^3+\sum_{1<i\leq k}x_i^2-\sum_{i>k}x_i^2.$$ 
Note that at a birth/death point of $f$, the kernel of the hessian $d^2f$ is one dimensional and is trivialized by the direction in which $d^3f$ is positive. We call this the {\em canonical co-orientation}. See Figure \ref{Figure1904}.

{
\begin{figure}[htbp]
\begin{center}
\begin{tikzpicture}[scale=.7]
%
\begin{scope}
\draw[thick] (1,0.3)..controls (1.6,2.3) and (3.6,.1)..(3.9,2.1); 
\draw[thick] (1,0.3)..controls (1,-.5) and (3,-.5)..(5,-1); 
\draw[thick] (5,-1)--(5,1.3)..controls (4.4,1.5) and ( 3.9,1.6)..(3.9,2.1)..controls (3.9,2.5) and (5,2.5) .. (6,2.3)--(6,0); 
\draw[dashed,thick] (1,0.3)..controls (1,.5) and (2,.7)..(5,0.16);
\draw[thick] (5,0.16)--(6,0);
\draw (2.9,1) node{$\bullet$};
\draw[thick,dashed,->] (2.9,1)--(5.3,.7);
\draw (7.5,1.2) node[right]{Birth-death with canonical co-orientation};
\clip (5,0) rectangle (5.6,1);
\draw[thick,->] (2.9,1)--(5.3,.7);

\end{scope}
\begin{scope}[yshift=5cm,xshift=.5cm]
\draw (2.57,1.17) node{$\bullet$};
\draw[thick] (-.4,0)..controls (0,.7) and (0,1.3)..(0,2)..controls (1,2.1) and (1.5,2.1)..(1.5,2.5)..controls (1.5,2.7) and (1,3).. (.7,3)--(.7,2.1); 
\draw[thick] (1.5,2.4)..controls (1.5,1) and (3.9,1)..(3.9,2.1); 
\draw[thick] (-.4,0)..controls (1,0) and (3,-.5)..(5,-1); 
\draw[thick] (5,-1)--(5,1.3)..controls (4.4,1.5) and ( 3.9,1.6)..(3.9,2.1)..controls (3.9,2.5) and (5,2.5) .. (6,2.3)--(6,0); 
\draw[dashed,thick] (.7,2.1)--(.7,.8)--(5,0.13);
\draw[thick] (5,0.13)--(6,0);
\draw (7.5,1.5) node[right]{Morse point};
\end{scope}
%
\end{tikzpicture}
\caption{Morse and birth-death points.} 
\label{Figure1904}
\end{center}
\end{figure}
}

For $F \to W \to M$ a fiber bundle, a function $f:W \to \bR$ has {\em moderate singularities} if each restriction $f|_m:F_m \to \bR$ to the fiber over $m \in M$ has moderate singularities and if moreover the fibered singularities  of $f$ over $M$ 
\[
\begin{tikzcd}
W \arrow[dr]  \arrow[rr, "f"] &    & M \times \bR \arrow[dl] \\
& M & 
\end{tikzcd}
\]
are modeled on either a constant family of Morse critical points
$$ (t_1,\ldots,t_m,x_1,\ldots,x_N)\mapsto  (t_1,\ldots,t_m, \sum_{i\leq k}x_i^2-\sum_{i>k}x_i^2) $$
or a fibered birth/death of Morse critical points 
$$(t_1,\ldots,t_m, x_1,\ldots,x_N) \mapsto (t_1,\ldots, t_m, x_1^3+3t_1x_1 + \sum_{1<i\leq k}x_i^2 - \sum_{i >k}x_i^2 )$$
with both families fibered over $(t_1,\ldots,t_m)$ with respect to the obvious projection. If $M$ has boundary we will allow one-half of the above fibered models, i.e. restricting the fibration to a half space of the base. 

\subsection{The stable bundle}

Let $F \to W \to M$ be a fiber bundle and let $f:W \to \bR$ be a function with moderate singularities. The fiberwise singular set $\Sigma_f \subset W$ is a smooth submanifold of $W$ of the same dimension as $M$. It has a stratification $\Sigma_f = \bigcup_k \Sigma_{f,k}$ where $\Sigma_{f,k}$ is the closure of the set of fiberwise Morse critical points of index $k$. Let $C_{f,k} \subset \Lambda_f$ consist of the set of fiberwise cubic (birth/death) critical points of $f$ of index $k$ and $k+1$. So the boundary of $\Lambda_{f,k}$ consists of the disjoint union of $C_{f,k-1}$ and $C_{f,k}$. 

On the interior of $\Sigma_{f,k}$ we have a rank $k$ real vector bundle $\gamma_{f,k} \to \Sigma_{f,k}$ consisting of the negative eigenspace of the fiberwise Hessian along the fiberwise critical locus. Each $\gamma_{f,k}$ extends to the boundary $C_{f,k} \coprod C_{f,k-1}$ of $\Sigma_{f,k}$ as sub-bundles of the fiberwise tangent bundle of $W$, but along $C_{f,k}$ the dimensions of $\gamma_{f,k}$ and $\gamma_{f,k+1}$ jump by one. More precisely, we have the relation $\gamma_{f,k}\oplus \nu_k= \gamma_{f,k+1}$ where $\nu_k$ is the kernel of the Hessian, a rank 1 bundle along $C_{f,k}$. So we may form a vector bundle over $\Sigma_f=\bigcup_k \Sigma_{f,k}$ by using the canonical co-orientation to glue the vector bundles $\gamma_{f,k} \oplus \ve^{N-k}$ on $\Sigma_{f,k}$, where $\ve$ denotes the trivial rank 1 bundle, to a rank $N$ vector bundle on $\Sigma_f$. More precisely, the canonical co-orientation gives a trivialization $\nu_k \simeq \ve$ and so we may glue via the vector bundle isomorphism over $C_{f,k}$ given by
$$ \gamma_{f,k} \oplus \ve^{N-k} \simeq \gamma_{f,k} \oplus \nu_k \oplus \ve^{N-k-1} \simeq \gamma_{f,k+1} \oplus \ve^{N-k-1} .$$
We denote the stable class of this real vector bundle over $\Sigma_f$ by $\gamma_f$ and will refer to it as the {\em stable bundle} of $f$. 

\subsection{Framed functions}

For $F \to W \to M$ a fiber bundle, a framed function $f:W \to \bR$ is roughly speaking a function with moderate singularities whose stable bundle $\gamma_f$ is trivial and moreover equipped with a trivialization. More precisely, one requires the existence of isomorphisms $$\varphi_k: \gamma_{f,k} \xrightarrow{\simeq} \ve^k$$ which are compatible with the above gluing maps. We will refer to the following result as the {\em framed function theorem}.

\begin{theorem}\label{thm: framed function}
For $F \to W \to M$ a fiber bundle, given a closed subset $A \subset W$ and a framed function $g:Op(A) \to \bR$ there exists a framed function $f:W \to \bR$ which agrees with $g$ on $Op(A)$. 
\end{theorem}

In particular, for any fiber bundle $F \to W \to M$ framed functions $f:W \to \bR$ exist, and if $f_0$ and $f_1$ are two such then there exists a framed function $F:W \times [0,1] \to \bR$ on the product fiber bundle $F \times [0,1] \to W \times [0,1] \to M \times [0,1]$ such that $F|_{W \times 0} = f_0$ and $F_{W \times 1} = f_1$. 

Theorem \ref{thm: framed function}  in particular implies, and is essentially equivalent to, the fact that a suitable model for the {\em space of framed functions} on the fiber $h:F \to \bR$ is contractible, even in a strong relative form. The high connectivity of the space of framed functions was first established by work of the second author \cite{I87}. The strongest form is given by the Extension Theorem 1.1.1 of \cite{EM12}, slightly reformulated as Theorem \ref{thm: framed function} above. The fact that framed functions may be chosen and are independent of choices is central to the definition of higher torsion invariants. 

We note the inevitability of relying on some result in the theory of simplification of singularities to even establish the existence of functions $f:W \to \bR$ with moderate singularities. This existence is not a priori guaranteed because the condition of having only moderate singularities, though $C^3$-stable, is far from $C^\infty$-generic in general. Indeed, $C^\infty$-generic functions $f:W \to \bR$ will have terrible and completely unclassifiable fiberwise singularities.

\subsection{Wrinkling}

The proof of the framed function theorem \ref{thm: framed function} given in \cite{EM12} uses the {\em wrinkling philosophy} developed by Eliashberg and Mishachev. We will in fact have use for this theory in the proof of Theorem \ref{thm: existence of example} and so we record the relevant statement below.

For $F \to W \to M$ a fiber bundle, a {\em fibered wrinkled function} $f:W \to \bR$ is a fibered function with moderate singularities such that the fiberwise singular locus $\Sigma_f \subset W$ consists of a disjoint union of contractible spheres $S \subset W$, such that the restriction of $f$ to a sufficiently small open neighborhood of each of the spheres $S$ is fibered equivalent to the local model
$$ (y,z,x) \mapsto \left( y, z^3+3(\|y\|^2-1)z - \sum_{i=1}^s x_i^2 + \sum_{i=s+1}^{n-q}x_j^2 \right)$$
$y \in \bR^{q-1}$, $z \in \bR$, $x \in \bR^{n-q}$, which has singularity $S^{q-1} =S^{q-1} \times \{0 \} \subset \bR^q \times \bR^{n-q}$, with equator of cusps $\{\|y\|=1, z=0,x=0\}$. We emphasize that we only require an equivalence between the restriction of $f$ to $Op(S)$ and the restriction of the normal form to $Op(S^{q-1})$. Here and below we use the Gromov notation to denote $Op(A)$ an arbitrarily small but unspecified open neighborhood of $A$, i.e. the {\em germ} of an open neighborhood of $A$.

Note that each such sphere $S$ we have an equator $C \subset S$ (i.e. the pair is diffeomorphic to $S^{n-2} \subset S^{n-1}$, where $n=\dim M$), and that $f$ has Morse critical points on $S \setminus C$ (of Morse indices $k$ and $k+1$ for some $k$) with  cubic birth/death points along $C$. 

If $F \to W \to M$ is a fiber bundle of manifolds, we denote $\text{Vert} \to W$ the vector sub-bundle of $TW$ consisting of vectors tangent to the fibers.

\begin{theorem}[\cite{EM00}, Theorem 2.3]\label{thm: wrinkling}  Let $F \to W \to M$ be a fiber bundle of closed manifolds. Let $A \subset W$ be a closed set. Suppose that a function is given $g:Op(A) \to \bR$ without critical points such that the vertical differential $d_vg$ extends to a nowhere vanishing section of $\text{Vert}^*$ on all of $W$. Then there exists a fibered wrinkled function $f:W \to \bR$ which agrees with $g$ on $Op(A)$.
\end{theorem}
\section{Higher torsion of fiber bundles}\label{sec: higher torsion of bundles}

\subsection{Higher torsion classes}

{\em Higher torsion invariants} are characteristic classes of smooth fiber bundles $F \to W \to M$ of smooth manifolds, defined under suitable conditions. The theory has analytic treatments \cite{BL95, BG01}, homotopy-theoretic treatments \cite{DWW03} and axiomatic treatments \cite{Igu08}, in addition to the Morse-theoretic perspective we adopt in the present article, following the book \cite{I02}.

From the Morse-theoretic viewpoint, a fiber bundle may be studied by considering smooth functions $f:W \to \bR$ and analyzing the parametrized Morse theory of the restrictions $$f_m:=f|_{F_m}: F_m \to \bR$$ of $f$ to the fiber, parametrized by $m \in M$. When $f:W \to \bR$ has moderate singularities in the sense of Appendix \ref{sec: moderate singularities}, after fixing a generic fiberwise metric one may construct the family of Morse chain complexes $C^*(f_m:\bZ[\pi_1W])$ where we may take the coefficients in $\bZ[\pi_1 W]$. In this appendix we briefly recall how this data may be used to yield the higher torsion classes.

\subsection{The Whitehead category} 
The natural home for the parametrized family of Morse chain complexes is the {\em Whitehead category}. For $R$ an associative ring and $G$ a subgroup of the units of $R$, there is a simplicial category $ \text{\em Wh}_\bullet(R,G)$, called the {\em Whitehead category}, whose objects in degree $k$ are $\Delta^k$-families of filtered finite free chain complexes over $R$ with the basis only well-defined up to permutation and elements of $G$, i.e., up to monomial change of basis \cite[Def 2.7, Def 2.16]{I05}. We add a category direction giving elementary expansions of the chain complex. This comes from the birth-death points in the family of generalized Morse functions.

Given a bundle of smooth manifolds $F \to W \to M$, a fiberwise generalized Morse function $f:W \to \bR$, and a generic choice of vertical metric, the data $C^*(f_m:\bZ[\pi_1W])$, $m \in M$, amounts to a map
$$  M \to | \text{\em Wh}_\bullet(\bZ[\pi_1W],\pi_1W)|.$$
If the action of $\pi_1 M$ on $H_*(F_m;R)$ is trivial, then a fiberwise mapping cone construction will make each fiber complex $C^*(f_m:\bZ[\pi_1W])$ contractible and thus gives a mapping
$$ \xi:M \to |\text{\em Wh}_\bullet^h(\bZ[\pi_1W],\pi_1W)| $$
well-defined up to homotopy, where $\text{\em Wh}_\bullet^h(\bZ[\pi_1W],\pi_1W)$ is the simplicial full subcategory of $ \text{\em Wh}_\bullet(\bZ[\pi_1W],\pi_1W)$ whose objects in degree $k$ are $\Delta^k$-families of the acyclic chain complexes. The map $\xi$ comes from a simplicial map
\[
    \xi_\bullet: simp\,M\to \text{\em Wh}_\bullet^h(\bZ[\pi_1W],\pi_1W)
\]
taking (small) smooth $k$-simplices in $M$ to objects of $\text{\em Wh}_k^h(\bZ[\pi_1W],\pi_1W)$. When $R = \bQ$ or $\bC$, the construction of higher torsion extends to the case then $\pi_1M$
acts unipotently on the fiber homology $H_*(F_m;R)$, by which we mean that the $\pi_1M$-module $H_*(F_m;R)$ admits a filtration by $\pi_1M$ submodules so that the action of $\pi_1M$ on the successive subquotients is trivial.

At this point we are very close to defining higher torsion. Indeed we have
$$\text{Wh}_1(G) = \pi_0 \text{\em Wh}_\bullet^h(\bZ[G],G) $$
for the classical Whitehead group $\text{Wh}_1(G)$, so $|\text{\em Wh}_\bullet^h(\bZ[G],G)|$ may be thought of as a space-level generalization of Whitehead torsion. More precisely, it was shown by the second author and J. Klein that there is a homotopy fiber sequence
$$|\text{\em Wh}_\bullet^h(R,G)| \to \Omega^\infty \Sigma^\infty(BG_+) \to\ \bZ \times  BGL(R)^+. $$
For example, if $G$ is finite then $\Omega^\infty \Sigma^\infty(BG_+)$ is rationally trivial above degree zero and hence $|\text{\em Wh}_\bullet^h(R,G)|$ has the same rational homotopy type as $\Omega BGL(R)^+$. If $R = \bC$ and $K(\bC)=\bZ \times BGL(\bC)^+$, we get a rational equivalence
$$ |\text{\em Wh}_\bullet^h(\bC,1)| \simeq_\bQ \Omega K(\bC).$$
We use the Borel regulator maps
$$b_k:K_{2k+1}\bC \to  \bR.$$
These are given by cap product with continuous cohomology classes $\mu_k\in H_c^{2k+1}(GL(\bC),\bR)$:
\[
K_{2k+1}(\bC)\cong \pi_{2k+1}BGL(\bC)^+\xrightarrow h H_{2k+1}BGL(\bC)\xrightarrow{\cap \mu_k}\bR
\]
where $h$ is the Hurewicz map on $BGL(\bC)^+$, a space with the same homology as $BGL(\bC)$.
See \cite{Dinakar}. Since $|\text{\em Wh}_\bullet^h(\bC,1)|$ is rationally equivalent to $\Omega K(\bC)$, we are able to obtain universal higher torsion invariants
$$\tau_{k} \in H^{2k}(|\text{\em Wh}_\bullet^h(\bC,1)|;\bR). $$
These universal classes will be used to define higher torsion. The explicit cohomology class used in \cite{I02} comes from the Dupont formulation of the Kamber-Tondeur form \cite{Dupont}.

\subsection{Definition of higher torsion}

Consider again a fiber bundle $F \to W \to M$, where we now assume that $\pi_1 M$ acts unipotently on the rational fiber homology $H_*(F_b;\bQ)$. We would like to define cohomology classes in $H^{2k}(M;\bR)$ by pulling back the class $\tau_{k}$ via the map $$\xi:M \to |\text{\em Wh}_\bullet^h(\bC,1)| $$ given by a choice of fiberwise (generalized) Morse function $f:W \to \bR$ and choice of metric. 
As explained above, $\tau_k$ is the restriction to $|\text{\em Wh}_\bullet^h(\bC,1)|$ of the universal higher torsion class $\tau_k\in H^{2k}(\text{\em Wh}_\bullet^h(\bC,U(1));\bR)$ given in \cite[subsection 2.7]{I05} and $\xi$ is the ``expansion functor'', see \cite[Def 2.18]{I05}.

At this point questions of existence and uniqueness of such $f$ arise. Although it is a (nontrivial) fact that fiberwise generalized Morse functions exist on any fiber bundle, it is not true that any two such functions will be joined by a family of such functions. Indeed, the map $\xi:M \to |\text{\em Wh}_\bullet^h(\bC,1)| $ obtained by the above construction may not be unique up to homotopy. 

There is a canonical way to select the {\em correct} map $M \to |\text{\em Wh}_\bullet^h(\bC,1)| $ by means of a {\em framed function}. For any fiber bundle $F\to W \to M$, and any two framed functions $f_0,f_1:W \to \bR$, by Theorem \ref{thm: framed function} there exists a framed function on the product bundle $W \times [0,1] \to M \times [0,1]$ which restricts to $f_0$ and $f_1$ over $M \times \{0\}$ and $M \times \{1\}$ respectively. This gives a map 
$$M \times [0,1] \to  |\text{\em Wh}_\bullet^h(\bC,1)|$$
which may be viewed as a homotopy between the maps $M \to|\text{\em Wh}_\bullet^h(\bC,1)| $ given by $f_0$ and $f_1$. Hence by choosing a framed function on $W$ we get a well-defined and distinguished homotopy class of map
 $$\xi_0: M \to |\text{\em Wh}_\bullet^h(\bC,1)| .$$

\begin{definition}  Assume that $\pi_1 M$ acts unipotently on $H_*(F_b;\bQ)$. The {\em higher torsion classes} $\tau_{k}(W) \in H^{2k}(M;\bR)$  are defined as the pullback of the universal classes $\tau_{k} \in H^{2k}( |\text{\em Wh}_\bullet^h(\bC,1)| ;\bR)$ under the homotopically unique map $\xi_0: M \to |\text{\em Wh}_\bullet^h(\bC,1)| $ determined by a framed function on $W$. \end{definition} 

\begin{example} The degree zero class $\tau_0(W,\rho)  \in H^0(M;\bR)$ is just the usual Franz-Reidemeister torsion of the fiber twisted by the local system $\rho|_{\pi_1 F}$, maybe after taking absolute value and natural logarithm depending on conventions \cite{F35,R35}.  \end{example}

\subsection{Higher Reidemeister torsion with twisted coefficients}\label{ss: higher torsion with G coef} 

We can generalize the above discussion to fiber bundles $F \to W \to M$ equipped with a rank 1 unitary local system $\rho: \pi_1 M \to U(1)$ with finite image $G$ such that the action of $\pi_1 M$ on the twisted fiber homology $H_*(F_m;\bC^\rho)$ is unipotent.  The framed function theorem still gives us a well-defined map (up to homotopy)
 $$\xi_\rho: M \to |\text{\em Wh}_\bullet^h(\bC,G)| $$
 and we get classes $\tau_k(W,\rho)\in H^{2k}(M;\bR)$ by pulling back the universal classes. In fact, if real coefficients are used instead of complex coefficients, then the odd classes automatically vanish, i.e. the pullback of $\tau_{2k+1}$ to $|\text{\em Wh}_\bullet^h(\bR,G)|$ is zero, $G=\{\pm1\}$. So to get nontrivial odd torsion classes $\tau_{2k+1}(W,\rho)\in H^{4k+2}(M;\bR)$ one needs to use a nontrivial unitary local system. 
 
 The use of unitary local systems can be convenient depending on the application, for example in \cite{AI} the higher Reidemeister torsion of circle bundles over the sphere is shown to compute a stable version of the Euler number. However, for the specific version of higher torsion we will focus on in the present article there is no local system. Indeed we will use $\bZ$ coefficients, and so we will only get the even classes $\tau_{2k}$ in $H^{4k}(M;\bR)$.  
 
\subsection{The framing principle}\label{ss:FramingPrinciple} 

A useful result for computations of higher torsion is the {\em framing principle}, which gives the precise correction term needed to compute the higher Reidemeister torsion for a function $f:W \to \bR$ with moderate singularities, even if the stable bundle $E$ is not trivial (and hence $f$ is not frameable). The formula \cite[Theorem 4.11]{I02} is the following.
\begin{theorem}\label{thm: framing} If  $f:W \to \bR$ is a function with moderate singularities, then
\begin{equation}\label{eq:FramingPrinciple}
\tau(W,\rho)=\sum_j\tau_j(f,\rho)+ \frac12 \sum_k (-1)^k \zeta(2k+1) \pi_\ast^{\Sigma f} \left( ch_{2k}(\gamma_f\otimes\bC)\right) 
\end{equation}
where the terms are defined below.
\end{theorem}

\begin{itemize}
\item  $\tau_j(f,\rho)\in H^{2j}(M;\bR)$ is the pull-back of the universal class $\tau_j$ under $M \to |\text{\em Wh}_\bullet^h(\bC,G)|$, the mapping  given by the family of Morse chain complexes associated to the family of functions $F_m$, $,m\in M$. 
We are assuming that $\rho$ has a finite image $G\subset U(1)$.
\item $\Sigma_f$ is the fiberwise singular set of $f$. It has the same dimension as the base $M$ and is a union $\Sigma_f=\bigcup \Sigma_{f,k}$ as discussed in Appendix \ref{sec: moderate singularities}. 
\item $\gamma_f$ is  the stable bundle of $f$ as also discussed in Appendix \ref{sec: moderate singularities}.
\item The map $\pi_\ast^{\Sigma f}:H^{4k}(\Sigma)\to H^{4k}(M)$ is, at the cochain level, the alternating sum of push-down operators:
\[
	\pi_\ast^{\Sigma f}=\sum_i (-1)^i \pi_\ast:
	 C^{4k}(\Sigma_{f,i})\to C^{4k}(M).
\]
\end{itemize}

Explicitly, given a $4k$ cocycle $\varphi$ on $\Sigma_f$ we take the restriction $\varphi_i$ of $\varphi$ on each $\Sigma_{f,i}$. This is well-defined away from the boundary of $\Sigma_{f,i}$. The image of this in $C^{4k}(M)$ is the $4k$-cocycle which takes a small $4k$ simplex $\sigma$ in $M$, takes all sections $\tilde\sigma_i$ of $\sigma$ to $\Sigma_{f,i}$ and evaluates
\[
    \pi_\ast^{\Sigma f}(\varphi)(\sigma)=\sum (-1)^i \varphi_i(\tilde\sigma_i).
\]
The undefined terms cancel since, when $\tilde\sigma_i$ reaches the boundary of $\Sigma_{f,i}$, it is at a cusp which is the same location as a section $\tilde\sigma_{i+1}$ (or $\tilde\sigma_{i-1}$). The alternating signs makes these undefined terms equal with opposite sign. So, $\pi_\ast^{\Sigma f}$ is well-defined. See \cite[Proposition 4.8]{I02} for more details of this argument. It is shown there that $\pi_\ast^{\Sigma f}$ is a related to the Becker-Gottlieb transfer $\text{tr}_M^W:H^\ast(W)\to H^\ast(M)$ by $\text{tr}_M^W=\pi_\ast^{\Sigma f}\circ i^\ast$ where $i^\ast:H^\ast(W)\to H^\ast(\Sigma_f)$ is the restriction map induced by the inclusion $i:\Sigma_f \to W$.

Note that the correction term in \eqref{eq:FramingPrinciple} is nonzero only in degrees $4k$. In degrees $2j=4k+2$, when $j$ is odd, there is no correction term (i.e. the correction term is zero.)

Also note that the term $\sum_j\tau_j(f)$ changes by a sign if the indices of the fiberwise Morse chain complexes are all shifted by one in either direction. More precisely, the index shift operator on $\text{\em Wh}_\bullet^h(\bC,G)$ induces map $|\text{\em Wh}_\bullet^h(\bC,G)| \to |\text{\em Wh}_\bullet^h(\bC,G)|$ which sends $\tau_j$ to $-\tau_j$ \cite[Theorem 3.10]{I05}. However, in this paper we are always suspending an even number of times. So there is no change of sign.


To summarize: 

{\em We can use the family of Morse chain complexes $C(f_m)$, $m \in M$ of a function with moderate singularities $f:W \to \bR$ to compute the higher torsion of $W$ over $M$ if we add correction terms in degree $4k$ given by a suitable multiple of the push-down to $H^\ast(M)$ of the Pontryagin character of the stable bundle $\gamma_f$ over $\Sigma_f$.}

\subsection{Fibers with boundary}\label{sec: rel torsion}

Higher torsion invariants may also be defined for fiber bundles $F \to W \to M$ where the fiber $F$ is a compact manifold with boundary. In the {\em absolute} case we just use functions $f:(W,\partial W) \to ([0,1],1)$ with moderate singularities such that for each fiber $F$ we have that $\partial F$ is a regular level set. More generally, it suffices to impose the condition that $df(v)>0$ for outwards pointing tangent vectors $v$ along $\partial F$. However it is also convenient to consider {\em relative} versions in which the fiberwise homology of the function computes a relative homology group. As this relative case is directly relevant to the present article, we discuss some examples below. 

For example, the construction may be applied in the case in which the fiber $F$ is a cobordism on a manifold $X$, so that $\partial F = \partial_0 F \coprod \partial_1 F$ and we have an identification $\partial_0 F \simeq X$. More precisely, this means that $\partial W =\partial_0 W \coprod \partial_1 W $ and we have a trivialization of the fiber bundle $\partial_0 W \to M$, i.e. a fibered diffeomorphism $\partial_0 W \simeq M \times X$. In this case higher torsion is defined using functions $(W,\partial_0W,\partial_1 W) \to ([0,1],0,1)$ with moderate singularities and yields an invariant of such fiber bundles $W$ equipped with a trivialization  $\partial_0 W \simeq M \times X$, or more generally such that $df(v)>0$ for outwards pointing tangent vectors $v$ along $\partial_1 F$ and $df(v)>0$ for inwards pointing tangent vectors $v$ along $\partial_0 F$.

The fiber $F$ could be also be a compact manifold with corners, thought of as a cobordism on a compact manifold with boundary (or even corners) $X$. In this case the boundary $\partial F$ is divided into a {\em horizontal }boundary $\partial_h F$ and a {\em vertical boundary} $\partial_v F$, with $\partial_v F \simeq \partial_0 F \coprod \partial_1 F$ and an identification $\partial_0 F \simeq X$ as before. The boundary $\partial W$ of such a fiber bundle $F \to W \to M$ is similarly decomposed, and such fiber bundles should be equipped with a fibered diffeomorphism $\partial_0 W \simeq M \times X$. The torsion is again defined using functions $(W,\partial_0W,\partial_1 W) \to ([0,1],0,1)$ with moderate singularities, or more generally such that $df(v)>0$ for outwards pointing tangent vectors $v$ along $\partial_1 F$ and $df(v)>0$ for inwards pointing tangent vectors $v$ along $\partial_0 F$.

Finally, yet another variant is the case in which $F$ is a compact manifold with boundary and $\partial_0 F \subset \partial F$ is a codimension zero submanifold, identified with some fixed compact manifold with boundary $X$. Such bundles $F \to W \to M$ come equipped with a trivialization of the sub-bundle $\partial_0 W \subset \partial W$, i.e. a fibered diffeomorphism $\partial_0 W \simeq M \times X$. The torsion is computed using functions with moderate singularities $f:W \to \bR$ which have $\partial_0 W$ as a regular level set of $f|_{\partial W}$, such that inwards pointing vectors $v$ along the interior of $\partial_0 F$ satisfy $df(v)>0$ and such that outwards pointing vectors $v$ along  $\partial F \setminus \partial_0 F$ satisfy $df(v)>0$ (so the fiberwise Morse homology computes $H_*(F;\partial_0F)$ as before).

In any of these cases, we speak of the torsion of the {\em bundle pair} $(W,\partial_0W)$ and  there is a simple formula relating the higher torsion of $W$, of $\partial_0 W$ and of the pair $(W,\partial_0W)$. Indeed, from Corollary 3.18 on page 39 in \cite{Igu05} we have the following result:

\begin{theorem}[\cite{Igu05}]\label{thm: relative} For a bundle pair $(W,\partial_0W)$ we have
\[
        \tau_k(W)=\tau_k(W, \partial_0W)+\tau_k(\partial_0W)
\]
\end{theorem}


\subsection{Noncompact fibers} 

The same construction works if the fiber $F$ is noncompact and the bundle $W\to M$ is equipped with a reference fibration at infinity $f: W \to \bR$. In this case $\tau_k(W,\rho)$ depends on the choice of fibration at infinity, though this may be omitted from the notation if there is no danger of confusion. For example, one could define and compute higher torsion using functions with moderate singularities which are homotopic through fibrations at infinity to our reference fibration at infinity. The fiberwise Morse homology is computing $H_*(F,\partial_0F)$ where $\partial_0F$ is a sublevel set $f|_F<-c$ for $c>0$ very large, so this version should also be thought of as a relative form of higher torsion.

 In certain applications the bundle $W$ is assumed to be trivialized outside of a compact set and it makes sense to demand that the admissible functions $f:W \to \bR$ agree with some fixed fibration at infinity outside of this compact set. With such a restriction imposed, the class $\tau_k(W,\rho)$ does not depend on the choice of admissible function since deformations of admissible functions have compact support. 
 
 A very precise formulation of the restriction to compact sets is given in \cite[Lemma 3.20]{I02} where singularities of a family of functions are confined to a family of ``convex sets'' which are basically cobordisms embedded in the fiber. The family of convex sets gives a family of chain complexes which embeds in the chain complex of the entire fiber in the case when the fiber is compact. In general, we can define the family of total chain complexes to be that given by any family of convex sets which contains all the critical points of the function. By \cite[Lemma 3.20]{I02}, these all give the same chain complex. So, the family of chain complexes and thus the higher torsion is well-defined. The fiber also has a well-defined Euler characteristic $\chi(F)$ given by this finite chain complex. In the case of functions of tube type we have $\chi(F)=\pm1$.

 In any event, higher Reidemeister torsion is a {\em stable invariant}, which means that if $W$ is replaced by the fiberwise stabilization $W \times \bR^4$ and the function with mild singularities $f:W \to \bR$ is replaced with $f(w)+||x||^2-||y||^2$ for $(w,x,y) \in W \times \bR^2 \times \bR^2$, then the homotopy class of the map $M \to  |\text{\em Wh}_\bullet^h(\bC,1)| $ and sign of the fiber Euler characteristic $\chi(F)$ will remain unchanged. 
 
 Furthermore, by the framing principle, if $f:W \to \bR$ is replaced by $f\oplus q : W \times \bR^N \to \bR$, where $f\oplus q(w,v)=f(w)+q_{\pi(w)} (v)$ for $q_m:\bR^N \to \bR$ a family of non-degenerate quadratic forms on $\bR^N$ of even index parametrized by $m \in M$, then the homotopy class of the map $M \to  |\text{\em Wh}_\bullet^h(\bC,1)| $ changes but in a controlled way: the pullback of the universal class $\tau_k$ changes precisely by adding $\chi(F)$ times a (certain multiple) of the Pontraygin character of the stable bundle of the family $q_m$, $m \in M$, i.e. the stable vector bundle formed by the negative eigenspaces of $q_m$. The $\chi(F)$ factor comes from the fact that the composition of $p^\ast:H^\ast(M)\to H^\ast(W)$ with the Becker-Gottlieb transfer $tr_M^W:H^\ast(W)\to H^\ast(M)$ is multiplication by $\chi(F)$.

\begin{thebibliography}{AVGZ85}

\bibitem[A12]{A12} M. Abouzaid,
\newblock  Nearby Lagrangians with vanishing Maslov class are homotopy equivalent,
\newblock  Invent. Math. 189 (2012) 251–313


\bibitem[AACK]{AACK} M. Abouzaid, D. Alvarez-Gavela, S. Courte, T. Kragh, Normal invariant of nearby Lagrangians via twisted derivative, arXiv:2505.12515, submitted.


\bibitem[ACGK]{ACGK} M. Abouzaid, S. Courte, S. Guillermou, T. Kragh, Twisted generating functions and the nearby Lagrangian conjecture, to appear in Duke Math journal.


\bibitem[AK]{AK} M. Abouzaid, T. Kragh,  Simple homotopy equivalence of nearby
Lagrangians, Acta Math. 220 (2018)

  \bibitem[AG18a]{AG18a} D. \'{A}lvarez-Gavela, 
\newblock  Refinements of the holonomic approximation lemma,  
\newblock Algebraic \& Geometric Topology 18
(2018) 2265–2303.

\bibitem[AG18b]{AG18b} D. \'{A}lvarez-Gavela, 
\newblock The simplification of singularities of Lagrangian and Legendrian fronts,  
\newblock Inventiones Mathematicae, 214(2) (2018) 641–737.      

\bibitem[AD22]{AD22} D. Alvarez-Gavela, D. Darrow
\newblock Caustics of Lagrangian homotopy spheres with stably trivial Gauss map, Journal of Symplectic Geometry
 


\bibitem[AI]{AI} D. Alvarez-Gavela, K. Igusa, 
\newblock A Legendrian Turaev torsion via generating families. Journal de l’\'Ecole polytechnique—Mathématiques 8 (2021): 57-119.

\bibitem[AIS]{AIS}  D. Alvarez-Gavela, K. Igusa, M. Sullivan 
\newblock Existence of a restricted ruling implies vanishing of nearby Lagrangian torsion

\bibitem[A]{A}
V.I~Arnold
\newblock On a characteristic class entering quantization conditions

\bibitem[AGV]{AGV}
V.I.~Arnold, S. M.~Gusein-Zade, A. N.~Varchenko.
\newblock Singularities of differentiable maps. Volume 1
\newblock Mod. Birkhauser Class. Birkhauser/Springer, New York, 2012, xii+382 pp.

\bibitem[B27]{B27}
P.~ Benk
\newblock Caustics of Lagrangians with stably trivial Gauss map,
\newblock In preparation

\bibitem[BL95]{BL95}
J.M.~Bismut, J.~Lott.
\newblock Flat vector bundles, direct images and higher real analytic
torsion.
\newblock J. Amer. Math. Soc. 8 (1995), no. 2, 291–363.

\bibitem[BG01]{BG01}
J.M.~Bismut, S.~Goette. 
\newblock Families torsion and Morse functions.
\newblock  Soci\'et\'e math\'ematique de France, Ast\'erisque (275) 2001.

\bibitem[B82]{B82}
M. B\"okstedt
\newblock The rational homotopy type of $\Omega$WhDiff.
\newblock Algebraic Topology Aarhus 1982: Proceedings of a conference held in Aarhus, Denmark, August 1–7, 1982. Berlin, Heidelberg: Springer Berlin Heidelberg, 2006.


\bibitem[C96]{C96}
Y.~Chekanov.
\newblock Critical Points of Quasi-Functions and Generating Families of Legendrian Manifolds
{\em Funktsional. Anal. i Prilozhen 30(2) 56--69, 1996.}



\bibitem[CP25]{CP25}
S.~Courte and N.~Porcelli.
\newblock On the parametrised Whitehead torsion of families of nearby Lagrangian submanifolds.
\newblock {\em arXiv preprint} 2506.06110 (2025).




\bibitem[D76]{Dupont} Dupont, Johan L. "Simplicial de Rham cohomology and characteristic classes of flat bundles." Topology 15.3 (1976): 233--245.

\bibitem[DWW]{DWW03}
W.~Dwyer, M.~Weiss, and B.~Williams.
\newblock A parametrized index theorem for the algebraic K-theory Euler class.
\newblock {\em Acta Mathematica}, 190(1), 2003.


  
  \bibitem[E72]{E72}
  Y. Eliashberg
  \newblock Surgery of singularities of smooth mappings.
  \newblock Mathematics of the USSR-Izvestiya 6.6 (1972): 1302.

\bibitem[EG98]{EG98}
Y.~Eliashberg and M.~Gromov.
\newblock Lagrangian intersection theory : Finite dimensional approach.
\newblock {\em American Mathematical Society Translations}, 2(186), 1998.

  
  \bibitem[EM00]{EM00}
  Y. Eliashberg and N. Mishachev.
  \newblock Wrinkling of smooth mappings-II Wrinkling of embeddings and K. Igusa’s theorem. 
  \newblock Topology, 39(4), 711-732.


\bibitem[EM12]{EM12}
Y.~Eliashberg and N.~Mishachev.
\newblock The space of framed functions is contractible.
\newblock {\em Essays on mathematics and its applications, Springer}, pages
  81--109, 2012.
  




\bibitem[Fra35]{F35}
W.~Franz.
\newblock \"uber die torsion einer \"uberdeckung.
\newblock {\em J. Reine Angew. Math}, 173:245--254, 1935.

\bibitem[Fuk95]{F95}
K.~Fukaya.
\newblock The symplectic s-cobordism conjecture: a summary.
\newblock {\em Geometry and Physics, Aarhus}, pages 209--210, 1995.

\bibitem[FSS]{FSS}
K. Fukaya, P. Seidel, I. Smith
\newblock{Exact Lagrangian submanifolds in simply-connected
cotangent bundles}, 
Invent. Math. 172 (2008) 1–27

\bibitem[G90]{G90}  E. Giroux, Formes generatrices d’immersions lagrangiennes dans un espace cotangent, G\'eom\'etrie Symplectique et M\'ecanique. Lecture Notes in Mathematics, vol 1416. 1990.
    





\bibitem[GKS12]{GKS12}
S.~Guillermou, M.~Kashiwara, and P.~Schapira.
\newblock Sheaf quantization of {H}amiltonian isotopies and applications to
  nondisplaceability problems.
\newblock {\em Duke Mathematical Journal}, 161(2):201--245, 2012.

\bibitem[Gui12]{Gui12}
S.~Guillermou.
\newblock Quantization of conic Lagrangian submanifolds of cotangent bundles.
\newblock{\em  arXiv preprint}  1212.5818 (2012).





\bibitem[HW73]{HW73}
A.E. Hatcher and J.B. Wagoner.
\newblock Pseudo-isotopies of compact manifolds.
\newblock {\em Ast\'erisque}, 6, 1973.



\bibitem[Igua]{K}
K.~Igusa.
\newblock The generalized Grassmann invariant.
\newblock (preprint).



\bibitem[Igu79]{I79}
K.~Igusa.
\newblock The $Wh_3(\pi)$-obstruction for pseudoisotopy.
\newblock PhD thesis, 1979.

\bibitem[Igu82]{I82}
K.~Igusa.
\newblock What happens to Hatcher and Wagoner's formula for
  $\pi_0(\cC(M))$ when the first Postnikov invariant of $M$ is nontrivial?,
  volume 1046 of {\em Lecture Notes in Mathematics}.
\newblock Springer, 1982.

\bibitem[Igu87]{I87}
K.~Igusa.
\newblock The space of framed functions.
\newblock {\em Transactions of the American mathematical society},
  301:431--477, 1987.
  
  
\bibitem[Igu05]{Igu05}  Higher Complex Torsion and the Framing Principle
 {\em Memoirs of the American Mathematical Society}
Volume: 177; 2005; 94 pp


\bibitem[Igu02]{I02}
K.~Igusa.
\newblock Higher Franz-Reidemeister torsion, 
{\em AMS/IP  studies in advanced mathematics}.
\newblock Vol 31. American Mathematical Society, 2002.

\bibitem[Igu04]{I04}
K.~Igusa.
\newblock Combinatorial Miller–Morita–Mumford classes and Witten cycles, 
{\em Algebraic \& Geometric Topology}.
\newblock 4(1), 473-0520, 2004.

\bibitem[Igu05]{I05}
K.~Igusa.
\newblock Higher complex torsion and the framing principle, volume 177 of
  {\em Memoirs of the American Mathematical Society}.
\newblock American Mathematical Society, 2005.

\bibitem[Igu08]{Igu08}
K.~Igusa.
\newblock Axioms for higher torsion invariants of smooth bundles.
\newblock Journal of Topology 1.1 (2008): 159-186.

\bibitem[IK93a]{IK93a}
K.~Igusa and J.~Klein.
\newblock The borel regulator map on pictures, i: A dilogarithm formula.
\newblock {\em K-Theory}, 7(3):201--224, 1993.

\bibitem[IK93b]{IK93b}
K.~Igusa and J.~Klein.
\newblock The {B}orel {R}egulator {M}ap on {P}ictures {II}: {A}n {E}xample from
  {M}orse {T}heory.
\newblock {\em K-Theory}, 7(3):225--267, 1993.

\bibitem[JT]{JT}
J. Jordan and L. Traynor. 
\newblock Generating family invariants for Legendrian links of unknots
\newblock  Algebraic \& Geometric Topology 6.2 J(2006): 895-933.



\bibitem[K89]{K89}
J.~Klein
\newblock The cell complex construction and higher R-torsion for bundles with framed Morse function.
\newblock {\em PhD thesis, Brandeis University} 1989


\bibitem[Kra18]{Kra18} T.~Kragh
\newblock Generating Functions in $\mathbb {R}^{2n}$ and the Hatcher-Waldhausen map
\newblock  arXiv preprint arXiv:1804.02557v3 (2018)

\bibitem[Kra13]{Kra13} T.~Kragh.
\newblock Parametrized ring-spectra and the nearby Lagrangian conjecture
\newblock {\em Geometry \& Topology, 17(2), 639-731}




\bibitem[Kra18]{Kra18} 
T.~Kragh.
\newblock Generating families for {L}agrangians in $\mathbb{R}^{2n}$ and the
  {H}atcher-{W}aldhausen map.
\newblock https://arxiv.org/abs/1804.02557, 2018.






\bibitem[La11]{La11}
F.~Laudenbach
\newblock Transversali\'e, courants et th\'eorie de Morse : un cours de topologie diff\'erentielle
\newblock{\em  \'Editions de l'\'Ecole Polytechnique}, 2011.


\bibitem[Mil]{M65}
J.~Milnor.
\newblock Lectures on the H-Cobordism Theorem. 
\newblock Princeton University Press, Princeton. 1965

\bibitem[Mil]{M66}
J.~Milnor.
\newblock Whitehead torsion.
\newblock {\em Bulletin of the American Mathematical Society}, 72(3), 1966.






\bibitem[Mil61]{M61}
J.~Milnor.
\newblock Two complexes which are homeomorphic but combinatorially distinct.
\newblock {\em Annals of Mathematics}, 74:575–590, 1961.

\bibitem[Mur12]{Mur12}
E.~Murphy.
\newblock Loose Legendrian Embeddings in High Dimensional Contact Manifolds.
\newblock ProQuest LLC, Ann Arbor, MI, 2012, 60 pp.

\bibitem[Na09]{Na09}
D. Nadler
\newblock David. Microlocal branes are constructible sheaves.
\newblock Selecta Mathematica 15.4 (2009): 563-619.



\bibitem[Ra]{Dinakar}
\newblock D.~ Ramakrishnan.
\newblock Regulators, algebraic cycles, and values of L-functions
\newblock
{\em Contemp. Math} 83, 183-310, 1989.

\bibitem[Rei35]{R35}
K.~Reidemeister.
\newblock Homotopieringe und linsenr\^aume.
\newblock {\em Hamburger Abhandl}, 11:102--109, 1935.


\bibitem[RS73]{RS71}
D.B. Ray and I.M. Singer.
\newblock R-torsion and the Laplacian for Riemannian manifolds.
\newblock {\em Advances in Mathematics}, 7:145–210, 1973.

\bibitem[R92]{R92}
J. Rognes
\newblock The Hatcher-Waldhausen map is a spectrum map.
\newblock Preprint series: Pure mathematics http://urn. nb. no/URN: NBN: no-8076 (1992).














\bibitem[Sik86]{S86}
J.C. Sikorav.
\newblock Sur les immersions Lagrangiennes dans un fibr\'e cotangent.
\newblock {\em Comptes Rendus de l'Academie de Sciences Paris}, 1(32):119--122,
  1986.
  
  
  

\bibitem[Sma61]{S61}
S.~Smale.
\newblock Generalized Poincaré's conjecture in dimensions greater than four.
\newblock {\em Annals of Mathematics}, 74:391--406, 1961.

\bibitem[S02]{S02} M. Sullivan, $K$-theoretic invariants for Floer homology. {\em Geometric \& Functional Analysis} GAFA 12.4 (2002): 810-872.



\bibitem[Th99]{Th99} D. Th\'eret, A complete proof of Viterbo’s uniqueness theorem on generating functions, Topology and its Applications 96 (1999) 249–266

\bibitem[T01]{T01}
L. Traynor, Lisa. 
\newblock Generating function polynomials for Legendrian links.
\newblock Geometry \& Topology 5.2 (2001): 719-760.


\bibitem[Tr22]{Trotman22}
D.~Trotman.
\newblock Some geometric properties of Whitney stratifications.
\newblock Dalat Uni J. of Science, Vol 12.2 (2022) 78 — 85.


\bibitem[V92]{V92}  C. Viterbo, Symplectic topology as the geometry of generating functions, Math. Ann. 292
(1992) 685–710.






\bibitem[Wag76]{W76}
J.B. Wagoner.
\newblock Diffeomorphisms, $K_2$, and analytic torsion.
\newblock {\em Algebraic and Geometric Topology (Proc. Sympos. Pure Math.,
  Stanford Univ., Stanford, Calif.)}, 1:23--33, 1976.

\bibitem[Wal82]{W82}
F.~Waldhausen.
\newblock Algebraic K-theory of spaces, a manifold approach.
\newblock {\em Canadian Mathematical Society Conference Proceedings}, 2(1),
  1982.

\bibitem[Whi]{W50}
J.H.C. Whitehead.
\newblock Simple homotopy types.
\newblock {\em American Journal of Mathematics}, 72(1), 1950.

\end{thebibliography}
\end{document}